%% file: FlagGrassDerived.tex
\pdfoutput=1
\newif\ifpersonal
\personaltrue 
\documentclass[10pt,a4paper,reqno]{amsart} 
\usepackage[normalem]{ulem}
\usepackage{amsmath,amsthm,amssymb,mathrsfs,mathtools,bm,eucal,tensor} 
\usepackage[utf8]{inputenc} 
\usepackage[T1]{fontenc} 
\usepackage{enumitem,comment,braket,xspace,tikz-cd,csquotes} 
\usepackage[all,cmtip]{xy} 
\usepackage[centering,vscale=0.7,hscale=0.8]{geometry}
\usepackage[hidelinks]{hyperref}
\usepackage[capitalize]{cleveref}
\usepackage{wrapfig}
\usepackage{graphicx,scalerel}
\newcommand\wh[1]{\hstretch{2}{\hat{\hstretch{.5}{#1}}}}

\usepackage[utf8]{inputenc}
\usepackage{tikz-cd}
\usetikzlibrary{cd,decorations.pathreplacing}

\numberwithin{equation}{subsection} 

\makeatletter
\def\newlistingtheorem#1#2#3{%
\newtheorem{#1@outer}[equation]{#2}
\newlist{#1@inner}{enumerate}{1}  
\setlist[#1@inner]{topsep=2pt, beginpenalty=10000, label={\rm(\alph*)},  ref={\csname the#1@outer\endcsname\rm(\alph*)}}
\crefname{#1@inneri}{#3}{#2}
\Crefname{#1@inneri}{#3}{#2}
\newenvironment{#1}{\begin{#1@outer}\leavevmode\begin{#1@inner}}{\end{#1@inner}\end{#1@outer}}
}
\makeatother

\theoremstyle{plain}
\newtheorem{thm-intro}{Theorem}
\newtheorem{thm}[equation]{Theorem}

\newtheorem*{thm*}{Theorem}

\newtheorem{lem}[equation]{Lemma}
\newtheorem{sublem}[equation]{Sublemma}
\newtheorem{prop}[equation]{Proposition}

\newtheorem{cor}[equation]{Corollary}

\theoremstyle{definition}
\newtheorem{defin}[equation]{Definition}

\newtheorem{eg}[equation]{Example}

\theoremstyle{remark}
\newtheorem{rem}[equation]{Remark}
\newtheorem{construction}[equation]{\textbf{Construction}}
\newlistingtheorem{rems}{Remarks}{Remark}

\input{newcommands_all}

\newlist{assumptions}{enumerate}{10}
\setlist[assumptions]{label*={\upshape{(\alph*)}}}

\crefname{assumptionsi}{assumption}{assumptions}
\Crefname{assumptionsi}{Assumption}{Assumptions}

\newcommand{\flags}{\operatorname{\underline{\mathsf{Fl}}}}
\newcommand{\affinize}{\mathrm{aff}}
\newcommand{\doubletimes}[5]{%
#1 \hspace{1.5em} \underset{\makebox[0pt][r]{$\scriptstyle #3$} ~\underset{\makebox[0pt][c]{$\scriptscriptstyle #5$}}{\hspace{0.5em}\vphantom{\int}\times\hspace{0.5em}}~ \makebox[0pt][l]{$\scriptstyle #4$}}{\times} \hspace{1.5em} #2%
}

\newcommand{\fibf}{\mathfrak{F}}
\newcommand{\dAff}{\mathsf{dAff}}
\newcommand{\dAffover}[1]{\dAff_{/#1}}
\newcommand{\dAffX}{\dAffover{X}}
\newcommand{\Aff}{\mathsf{Aff}}

\newcommand{\FibF}{\mathsf{FibF}}
\newcommand{\dFibF}{\mathsf{dFibF}}

\newcommand{\inftyGpd}{\mathrm{Gpd}_\infty}
\newcommand{\inftyCat}{\mathrm{Cat}_\infty}
\newcommand{\sCAlg}{\mathbf{sCAlg}}
\newcommand{\sCAlgNoeth}{\sCAlg^\mathrm{Noeth}}
\newcommand{\CohX}{\mathbf{Coh}^-}
\newcommand{\PerfX}{\mathbf{Perf}}
\newcommand{\BunGX}{\mathbf{Bun}^{\mathbf{G}}}
\newcommand{\BunG}{\mathrm{Bun}^{\mathbf{G}}}

\newcommand{\Gr}{\underline{\mathsf{Gr}}}
\newcommand{\GrX}{\mathbf{Gr}}
\newcommand{\fibGr}{\mathcal{G}r}
\newcommand{\loc}{\mathrm{loc}}

\newcommand{\BG}{\mathrm{B}\mathbf{G}}
\newcommand{\dR}{\mathrm{dR}}

\newcommand{\stackofcatstacks}{\underline{\mathsf{dSt}}^{\inftyCat}}
\newcommand{\stackofstacks}{\underline{\mathsf{dSt}}}
\newcommand{\closedpairs}{\underline{\mathsf{Cl}}}
\newcommand{\namecell}[2]{\ar[phantom,start anchor=center, end anchor=center]{#1}[description]{#2}}

\begin{document}

\title{A flag version of Beilinson-Drinfeld Grassmannian for surfaces}

\author{Benjamin HENNION}
\address{Benjamin HENNION, Universit\'e Paris-Saclay, Paris, France}
\email{benjamin.hennion@universite-paris-saclay.fr }

\author{Valerio MELANI}
\address{Valerio MELANI, DIMAI, Firenze, Italy}
\email{valerio.melani@unifi.it}

\author{Gabriele VEZZOSI}
\address{Gabriele VEZZOSI, DIMAI, Firenze, Italy}
\email{gabriele.vezzosi@unifi.it}
\date{\today}

\begin{abstract}
	In this paper we define and study a generalization of the Belinson-Drinfeld Grassmannian to the case where the curve is replaced by a smooth projective surface $X$, and the trivialization data are given on loci suitably associated to a nonlinear flag of closed subschemes. In order to do this, we first establish some general formal gluing results for moduli of almost perfect complexes, perfect complexes and torsors. We then construct a simplicial object $\flags_X$ of flags of closed subschemes of a smooth projective surface $X$, naturally associated to the operation of taking union of flags. We prove that this simplicial object has the $2$-Segal property. For an affine complex algebraic group $G$, we finally define a derived, flag analog $\fibGr_X$ of the Beilinson-Drinfeld Grassmannian of $G$-bundles on the surface $X$, and show that most of the properties of the Beilinson-Drinfeld Grassmannian for curves can be extended to our flag generalization. In particular, we prove a factorization formula, the existence of a canonical flat connection, and define a chiral product on suitable sheaves on $\flags_X$ and on $\fibGr_X$. We also sketch the  construction of actions of flags analogs of the loop group and of the positive loop group on $\fibGr_X$. To fixed ``large'' flags on $X$, we associate ``exotic'' derived structures on the classical stack of $G$-bundles on $X$.  Analogs of the flag Grassmannian for other Perf-local stacks (replacing the stack of $G$-bundles) are briefly considered, and flag factorization is proved for them, too.

\end{abstract}

\maketitle

\tableofcontents

\ifpersonal
\presetkeys{todonotes}{inline}{}
\fi

\section{Introduction}
Given a connected smooth complex projective curve $C$ and $\mathbf{G}$ a smooth affine $\mathbb{C}$-group scheme, the corresponding (global) \emph{affine Grassmannian} is the functor $$\underline{\mathrm{Gr}}_{C,\mathbf{G}}: (\Aff_{\mathbb{C}})\op \to \mathbf{Sets} \, : \, S \longmapsto \left\{ (x, \mathcal{E}, \varphi) \, |\, x \in C(S), \mathcal{E}\in \mathrm{Bun}^{\mathbf{G}}(C \times S), \varphi: \mathcal{E}_{| (C\times S)\smallsetminus \Gamma_{x}} \simeq \mathcal{E}_0  \right\}$$ where $\mathrm{Bun}^{\mathbf{G}}(C \times S)$ is the groupoid of $\mathbf{G}$-torsors on $C \times S$, $\Gamma_x \subset C \times S$ denotes the graph of $x \in C(S)$, and $\mathcal{E}_0$ denotes the trivial $\mathbf{G}$-bundle (see, e.g. \cite[1.4]{Zhu2017}). A local version of $\underline{\mathrm{Gr}}_{C,\mathbf{G}}$ can also be defined using $\mathbf{G}$-bundles only defined on a formal neighbourhood of $\Gamma_x$ and trivialisations defined on the corresponding affine punctured neighbourhood (\cite[(1.2.1)]{Zhu2017}), and Beauville-Laszlo Theorem (\cite{Beauville-Laszlo}) ensures that such a local version of the affine Grassmannian is in fact equivalent to the global version recalled above (see, e.g. \cite[Theorem 1.4.2]{Zhu2017}).\\

A brilliant idea of A. Beilinson and V. Drinfeld was to generalise the construction of the affine Grassmannian by allowing an arbitrary, finite number of points $x$. This can be done by first introducing the so called \emph{Ran space} of the curve $C$ as the presheaf $$\mathsf{Ran}_C: (\Aff_{\mathbb{C}})\op \to \mathbf{Sets} \, : \, S \longmapsto \left\{ \textrm{finite non-empty subsets of } C(S)\right\},$$ and then defining
$$\underline{\mathrm{Gr}}^{\mathrm{BD}}_{C,\mathbf{G}}: (\Aff_{\mathbb{C}})\op \to \mathbf{Sets} \, : \, S \longmapsto \left\{ (\underline{x}, \mathcal{E}, \varphi) \, |\, \underline{x} \in \mathsf{Ran}_C (S), \mathcal{E}\in \mathrm{Bun}^{\mathbf{G}}(C \times S), \varphi: \mathcal{E}_{| (C\times S)\smallsetminus \Gamma_{\underline{x}}} \simeq \mathcal{E}_0  \right\}$$ where $\Gamma_{\underline{x}} \subset C \times S$ denotes the joint graph of $\underline{x} \in \mathsf{Ran}_C(S)$. The functor $\underline{\mathrm{Gr}}^{\mathrm{BD}}_{C,\mathbf{G}}$ is called the (Ran version of the) \emph{Beilinson-Drinfeld affine Grassmannian} of the pair $(C, \mathbf{G})$. There is an obvious morphism of presheaves $p: \underline{\mathrm{Gr}}^{\mathrm{BD}}_{C,\mathbf{G}} \to \mathsf{Ran}_C$, and the trivial $\mathbf{G}$-bundle endows it with a canonical section $\mathsf{triv}: \mathsf{Ran}_C \to \underline{\mathrm{Gr}}^{\mathrm{BD}}_{C,\mathbf{G}}$. Though $\mathsf{Ran}_C$ is not representable\footnote{Actually, $\mathsf{Ran}_C$ is not even an \'etale sheaf.}, not even by an ind-scheme, the map $p$ is itself ind-representable and it is even ind-proper when $\mathbf{G}$ is reductive (\cite[3.3]{Zhu2017}). The presheaf $\underline{\mathrm{Gr}}^{\mathrm{BD}}_{C,\mathbf{G}}$ is not represented by an ind-scheme but it is the presheaf (non-filtered) colimit of ind-schemes under proper maps (i.e. a pseudo-indscheme as in \cite[1.2]{Gaitsgory_contract}).  \\ Though passing from $\underline{\mathrm{Gr}}_{C,\mathbf{G}}$ to $\underline{\mathrm{Gr}}^{\mathrm{BD}}_{C,\mathbf{G}}$ might seem as an easy and obvious step, it does make it possible for $\underline{\mathrm{Gr}}^{\mathrm{BD}}_{C,\mathbf{G}}$ to exhibit an important property called \emph{factorization}. First of all notice that $\mathsf{Ran}_C$ carries a (non-unital)  semigroup structure given by the union $$\cup: \mathsf{Ran}_C \times \mathsf{Ran}_C \to \mathsf{Ran}_C: (\{\underline{x}\}, \{\underline{y}\}) \mapsto \{\underline{x}, \underline{y}\}.$$ Let $(\mathsf{Ran}_C \times \mathsf{Ran}_C)_{\mathrm{disj}} \subset \mathsf{Ran}_C \times \mathsf{Ran}_C $ is defined by those pairs $(\{\underline{x}\}, \{\underline{y}\})$ such that $\{\underline{x}\} \cap \{\underline{y}\}= \emptyset$. We define similarly $(\mathsf{Ran}_C \times \mathsf{Ran}_C \times \mathsf{Ran}_C )_{\mathrm{disj}} \subset \mathsf{Ran}_C \times \mathsf{Ran}_C \times \mathsf{Ran}_C $. Then, the factorization property of $\underline{\mathrm{Gr}}^{\mathrm{BD}}_{C,\mathbf{G}}$ can be stated as the existence of a natural isomorphism (called \emph{factorization isomorphism}) between the pullbacks of the two diagrams 
\begin{equation}\label{BDfactorization} \xymatrix{  \underline{\mathrm{Gr}}^{\mathrm{BD}}_{C,\mathbf{G}} \times \underline{\mathrm{Gr}}^{\mathrm{BD}}_{C,\mathbf{G}} \ar[d]_-{p\times p} &  & \underline{\mathrm{Gr}}^{\mathrm{BD}}_{C,\mathbf{G}} \ar[d]^-{p} \\ \mathsf{Ran}_C \times \mathsf{Ran}_C & (\mathsf{Ran}_C \times \mathsf{Ran}_C)_{\mathrm{disj}} \ar[r]^-{\cup} \ar[l]_-{\supset} & \mathsf{Ran}_C} \end{equation} which is compatible, in an obvious way, with the map $\mathsf{triv}: \mathsf{Ran}_C \to \underline{\mathrm{Gr}}^{\mathrm{BD}}_{C,\mathbf{G}}$ and satisfies a natural cocycle condition on $(\mathsf{Ran}_C \times \mathsf{Ran}_C \times \mathsf{Ran}_C )_{\mathrm{disj}}$ (\cite[Theorem 3.3.3]{Zhu2017}).
We may denote $\mathsf{Gr}_{C,\mathbf{G}; 2}^{BD}$ the common (up to canonical isomorphism) pullback of these two diagrams. The map $p:\mathsf{Gr}_{C,\mathbf{G}; 2}^{BD} \to \mathsf{Ran}_C $ carries a flat connection (\cite[Theorem 3.1.20]{Zhu2017}, \cite[2.8]{BD_Hitchin}) and actions of suitable Ran-versions $\mathrm{L}_{\mathsf{Ran}}\mathbf{G}$ and $\mathrm{L}^+_{\mathsf{Ran}}\mathbf{G}$ of the loop group $\mathrm{L}\mathbf{G}$ and of the positive loop group $\mathrm{L}^+\mathbf{G}$, respectively (\cite[1.1.2]{Gaitsgory2022} and \cite[Section 5, (5.5)]{MV_2007}).\\

In this paper we propose and study an analog of BD-Grassmannian, called the \emph{flag Grassmannian}, where the curve $C$ is replaced by a smooth complex projective \emph{surface} $X$, and points in the Ran space of $C$ are replaced with (families) of \emph{flags of subschemes} in $X$. Our investigation on the flag Grassmannian has its origins in some conversations of the third author with Mauro Porta\footnote{M. Porta, together with the first and third authors, is also responsible for our first version of formal and flag glueing results (see arXiv:1607.04503).} triggered by \S \, 3.4.6 of \cite{Beilinson_Drinfeld_Chiral_2004} where an equivalent approach to factorization algebras sheaves  is based on effective Cartier divisors (possibly considered up to reduced equivalence) on a curve, rather than on the Ran space of the curve. This remark immediately gave us the idea of a possible replacement of the Ran space in the higher dimensional case of a surface, with the further possibility of including \emph{arbitrary flags} consisting of an effective Cartier divisor together with a codimension 1 subscheme in this divisor.\\ Let us summarize the results of this paper.\\

\noindent\textbf{Formal glueing.} In proving several properties of our flag Grassmannian, we will need a higher dimensional version of Beauville-Laszlo gluing, therefore we devote the first three sections to establish a general version of \emph{formal gluing}. Let $k$ be field. For a $k$-scheme $Y$ and a closed subscheme $Z$, with open complement $U$, roughly speaking, \emph{formal gluing} (whenever it exists) is a way of describing geometric objects of some kind defined on $Y$, as geometric objects of the same kind defined on $U$ and on the formal completion $\hZ$ of $Z$ in $Y$, that are suitably \emph{compatible}. The notion of compatibility here is a delicate one. Morally, guided by a purely topological intuition, one would like to say that a geometric object $\mathcal{E}_U$ and a geometric object $\mathcal{E}_{\hZ}$ are compatible if they restrict to the same object (or to an equivalent one) on the intersection $U \cap \hZ= \hZ \smallsetminus Z$.
However, strictly speaking, in algebraic geometry this intersection is empty, so the notion of compatibility cannot be the naive one, at least if we want to remain within the realm of algebraic geometry, i.e. we want to keep viewing $Y$, $U$ and $\hZ$ as algebro-geometric objects. 
Many formal glueing results have been proved, and we give here only a probably non exhaustive list: \cite{Weil_Adeles_1982}, \cite{ArtinII}, \cite{Ferrand-Raynaud}, \cite{Beauville-Laszlo}, \cite{MB}, \cite{Ben-Bassat_Temkin_Tubular_2013}, \cite{Bhatt_algebraization_2014}, \cite{Schappi_descent_2015}, \cite{Hall_Rydh_2016}, \cite{mathew2019faithfully}. In the first sections, based on some of these previous works (mainly \cite{Bhatt_algebraization_2014}, \cite{mathew2019faithfully}) , we address the \emph{general formal glueing problem} when the geometric objects above are
\begin{itemize}
	\item almost perfect modules (also known as pseudo-coherent complexes) on $Y$, whose derived moduli stack is denoted by $\CohX_Y$ in the paper,
	\item perfect modules on $Y$, whose derived moduli stack is denoted by $\PerfX_Y$ in the paper,
	\item $\mathbf{G}$-bundles on $Y$, for $\mathbf{G}$ an affine $k$-group scheme, whose derived moduli stack is denoted by $\BunGX_Y$ in the paper,
\end{itemize}
and $Y$ is a locally Noetherian stack over $k$. If $\mathbf{F} = \CohX, \PerfX$ or $\BunGX$, and $Z\subset Y$ is a closed substack, restriction induces an equivalence of derived stacks over $Y$ (i.e. defined on the category $\dAff_{Y}$ of derived affine Noetherian schemes $\mathrm{Spec}\, R$ with a map $\mathrm{Spec}\, R \to Y$)
\begin{equation}\label{fglintro}
\begin{tikzcd}
\mathbf{F}_Y \ar{r}{\sim} & \displaystyle \mathbf{F}_{Y \smallsetminus Z} \newtimes_{\mathbf{F}_{\hZ^{\affinize} \smallsetminus Z}}  \mathbf{F}_{\hZ} \rlap{.}
\end{tikzcd}
\end{equation}
This is proved in Proposition \ref{prop:localformalglueing} and Theorem \ref{thm:formalglueing}.
In this formula, $\hZ$ denotes the formal completion of $Z$ inside $Y$, while $\hZ^\affinize \smallsetminus Z$ is defined using the \emph{fiber functor\footnote{This notion of fiber functor should not be confused with the one in the theory of Tannakian categories.} of affinization} (see definitions \ref{defin:fibf} and \ref{def-intornobucato}). In particular, when $R$ is an underived commutative $k$-algebra together with a map $\mathrm{Spec}\, R \to Y$, the value of $\hZ^\affinize \smallsetminus Z$ on $(\mathrm{Spec}\, R \to Y)$ is the expected one
$$(\hZ^\affinize \smallsetminus Z)(\mathrm{Spec}\, R \to Y) = \mathrm{Spec}(\widehat{R}_{\mathrm{I}_Z})\smallsetminus \mathrm{Spec}(R/\mathrm{I}_Z),$$ where $\mathrm{I}_Z\subset R$ denotes the ideal of the pullback of $Z$ to $\mathrm{Spec}\, R$, and $\widehat{R}_{\mathrm{I}_Z}$ is the ring completion of $R$ at the ideal $\mathrm{I}_Z$. Note that, by \emph{algebraisation} (Proposition \ref{prop:algebraization}), the canonical map $\mathbf{F}_{\hZ} \to \mathbf{F}_{\hZ^\affinize}$ is an equivalence of derived stacks over $Y$ (where $\hZ^\affinize(\mathrm{Spec}\, R \to Y)= \mathrm{Spec}(\widehat{R}_{\mathrm{I}_Z})$, when $\mathrm{Spec}\, R$ is a Noetherian underived affine $k$-scheme endowed with a map to $Y$).
In the presence of a longer \emph{non-linear flag} $(Z_k\subset Z_{k-1} \subset \cdots \subset Z_1 \subset Y)$ of closed substacks in $Y$, the formal glueing formula (\ref{fglintro}) can be iterated yielding what we call a \emph{flag decomposition} of $\mathbf{F}$ (see \S $\,$ \ref{sectionflagdecomp}).\\ Using the map $Y \to \mathrm{Spec}\, k$, and the induced functor 
\begin{equation}\label{functortoabsolutestacks}\dSt_Y \to \dSt_k : \mathbf{F} \mapsto \underline{\mathsf{F}},\end{equation} from derived stacks defined over $\dAff_{Y}$ to derived stacks defined over $\dAff_{k}$, formula (\ref{fglintro}) translates into the corresponding obvious formal glueing equivalence between derived stacks over $k$ \begin{equation}\label{fglintrounderline}\xymatrix{\underline{\mathsf{F}}_Y \ar[r]^-{\sim} & \underline{\mathsf{F}}_{Y \smallsetminus Z} \times_{\underline{\mathsf{F}}_{\hZ^\affinize \smallsetminus Z}}      \underline{\mathsf{F}}_{\hZ}  }\end{equation}
(see Remark \ref{rem:abs}).\\ In our analysis, the case of $G$-bundles follows from the case of perfect modules (since $\BunGX_Y$ is \emph{Perf-local} in the variable $Y$, see Lemma \ref{lem:perflocal}), while the case of almost perfect modules is similar to that of perfect modules. Therefore we have decided to present these three formal gluing results together, though in the rest of the paper we will mainly concentrate on the study of $G$-bundles. 
Moreover, for the sake of completeness, and since this does not require an excessive effort, we actually establish formal gluing and flag decomposition results, in the more general case where $Y$ a locally Noetherian \emph{derived} stack over $k$, $Z \subset Y$ is a derived closed substack.\\ Section \ref{sec_nonderived} explains how to deduce from the previous results formal gluing for the \emph{underived} versions of $\mathbf{F} = \CohX, \PerfX$ or $\BunGX$.  \\

\noindent \textbf{Nonlinear flags.} Starting with Section \ref{SECTIONflags}, we will let $k=\mathbb{C}$, and fix a smooth complex projective surface $X$. In Section \ref{SECTIONflags} we introduce the simplicial object $\underline{\mathsf{Fl}}_{X,\bullet}$ of families of nonlinear flags on $X$. Here is the basic idea of the construction of  $\underline{\mathsf{Fl}}_{X,\bullet}$. We first define  $\flags_{X,1}$ (which we also denote simply by $\flags_X$) as the presheaf of posets on $\Aff_{\mathbb{C}}$ assigning to $S$ the poset whose elements are pairs $(D, Z)$ consisting of a relative effective Cartier divisor $D$ on $X \times S/ S$, and an element $Z$ in the Hilbert scheme of $0$-dimensional subschemes of $X \times S /S$, subject to the nesting condition $Z \subset D$. The partial order on $\flags_{X,1}(S)$ is defined as follows: we say that $(D,Z) \leq (D',Z')$ (also denoted as a morphism $(D,Z) \to (D',Z')$) if there exists a commutative square of closed subschemes of $X \times S$
$$\xymatrix{D \ar[r] & D' \\ Z \ar[u] \ar[r] & Z' \ar[u]}$$ such that the canonical induced map $Z \to D\times_{D'} Z'$ yields an isomorphism on the induced reduced scheme structures. Note that by forgetting the partial order structure, $\underline{\mathsf{Fl}}_{X,1}$ is just a variant of the so-called \emph{flag Hilbert scheme} of $X$ (see e.g. \cite[4.5]{sernesi}). The main step in the construction of $\underline{\mathsf{Fl}}_{X,\bullet}$ is to define $\underline{\mathsf{Fl}}_{X,2}$ in such a way that it supports a union map $\bigcup: \underline{\mathsf{Fl}}_{X,2} \to \underline{\mathsf{Fl}}_{X,1}$ of the form $(D_1, Z_1) \cup (D_2, Z_2)= (D_1 +D_2, Z_1 \cup Z_2)$. In order to ensure that such a map is well-defined (e.g. that $Z_1 \cup Z_2$ is flat over the test scheme $S$), we define $\underline{\mathsf{Fl}}_{X,2}(S)$ as the sub-poset of pairs $((D_1, Z_1), (D_2, Z_2))$  in $\underline{\mathsf{Fl}}_{X,1}(S) \times \underline{\mathsf{Fl}}_{X,1}(S)$ satisfying the following \emph{good position} (or transversality) condition: $D_1 \cap D_2$ is a relative effective Cartier divisor inside both $D_1 /S$ and $D_2 / S$, and $Z_1 \cap Z_2 = D_1 \cap D_2$ (Definition \ref{defin:goodpair}). Note that, if $X$ is a \emph{curve}, then $\underline{\mathsf{Fl}}_{X,2}$ classifies pairs of \emph{disjoint} Cartier divisors. Finally, we build our full simplicial object $\underline{\mathsf{Fl}}_{X,\bullet}$ as the classifying simplicial object associated to the partially defined unital monoid structure $\cup$ on $\underline{\mathsf{Fl}}_{X,1}$ (see Section \ref{section:2SegalFlags}). Note that $\flags_{X, 0}= \left\{ * \right\}$, $\partial_0, \partial_2: \flags_{X,2} \to \flags_{X,1}$ are, respectively, the projection on the second and first flag, while $\partial_1: \flags_{X,2} \to \flags_{X,1}$ is the union map. The main result of Section \ref{SECTIONflags} (Theorem \ref{thm:2SegGood}) is that the simplicial object $\underline{\mathsf{Fl}}_{X,\bullet}$ has the so-called $2$\emph{-Segal} property (\cite{KapDyc}), i.e. the diagram $$\xymatrix{ \underline{\mathsf{Fl}}_{X,1} \times \underline{\mathsf{Fl}}_{X,1} & \underline{\mathsf{Fl}}_{X,2} \ar[l]_-{\supset} \ar[r]^-{\cup} & \underline{\mathsf{Fl}}_{X,1} }$$ defines an associative algebra structure in correspondences. This will induce, later in the paper, the same property on the \emph{simplicial flag Grassmannian} $\fibGr_{X, \bullet}$. \\

\noindent\textbf{Flag Grassmannian.} Fix, once and for all, a smooth affine group scheme $\mathbf{G}$ over $\mathbb{C}$, that will be most often omitted from the notations. In sections \ref{flagGrassfirst appearance} and \ref{section:Grass2Segal} we build the \emph{simplicial flag Grassmannian} $\fibGr_{X,\bullet}$ for $\mathbf{G}$-bundles, together with a simplicial map $\fibGr_{X,\bullet} \to \underline{\mathsf{Fl}}_{X,\bullet}$. We sketch here, in a slightly simplified way, the basic ideas of the constructions, where furthermore the schemes parametrising families of flags will be underived. The actual rigorous constructions in the main body of the paper follow a slightly different and more technical path (mainly in order to take care of higher coherences and derived structures), and giving here all such details would obscure the intuition. Moreover, in this Introduction, we will focus on the flag Grassmannian for $\mathbf{G}$-bundles; variations involving perfect complexes are considered in the main text.\\ For any affine (underived) Noetherian test scheme $S$, and any flag $(D,Z)\in \underline{\mathsf{Fl}}_{X,1}(S)$, so that $(D,Z)$ is a flag inside $X \times S$, we consider (Definition \ref{def:flagGrass}) the derived stack over $X\times S$ (notations being the one introduced in the previous paragraph) \begin{equation}\label{boldGrfixedflag} \GrX_X(S)(D,Z) := \mathbf{Bun}_{X \times S} \times_{\mathbf{Bun}_{(X \times S) \smallsetminus D} \times_{\mathbf{Bun}_{(X \times S)\smallsetminus D \times_{X\times S} \hZ^\affinize}} \mathbf{Bun}_{\hZ}} \left\{ * \right\} \,\,\, \in \, \dSt_{X \times S}\, ,
 \end{equation}  where the map from $\left\{ * \right\}$ represents the trivial $\mathbf{G}$-bundle $\mathcal{E}_0$. This derived stack over $X \times S$ is called the \emph{flag Grassmannian of} $X\times S \,/ S$ \emph{at the fixed flag} $(D,Z)$, and it classifies triples $(\mathcal{E}, \varphi, \psi)$ where $\mathcal{E}$ is a family of $\mathbf{G}$-bundles on $X \times S$, $\varphi$ a family of trivialisations of $\mathcal{E}$ on $(X\times S)\smallsetminus D$, 
 $\psi$ a family of trivialisations of $\mathcal{E}$ on $\hZ$, such that $\varphi$ and $\psi$ are compatible on the ``intersection'' $(X \times S)\smallsetminus D \times_{X\times S} \hZ^\affinize$. Like in the classical case where $X$ is a curve, there is also a corresponding \emph{local version} (Definition \ref{defin:localGr}) \begin{equation}\label{boldGrlocalfixedflag} \GrX^{\loc}_X(S)(D,Z) := \mathbf{Bun}_{\widehat{D}} \times_{\mathbf{Bun}_{\widehat{D} \smallsetminus D} \times_{\mathbf{Bun}_{\widehat{D}\smallsetminus D \times_{\widehat{D}} \hZ^\affinize}} \mathbf{Bun}_{\hZ}} \left\{ * \right\} \,\,\, \in \, \dSt_{X \times S}\, ,
 \end{equation} and formal gluing implies that the obvious restriction $$\GrX_X(S)(D,Z) \longrightarrow \GrX^{\loc}_X(S)(D,Z)$$ is an equivalence of derived stacks over $X \times S$ (Lemma \ref{lem:benjexo}). Translating these constructions into absolute derived stacks (i.e. stacks over $k=\mathbb{C}$) by using the functor (\ref{functortoabsolutestacks}) defined above, we get derived stacks $\Gr_X(S)(D,Z)$  and $\Gr^{\loc}_X(S)(D,Z)$ over $k$, together with a restriction equivalence $\Gr_X(S)(D,Z) \simeq \Gr^{\loc}_X(S)(D,Z)$. Since $D$ is a relative effective Cartier divisor in $X \times S \, / S$, one can prove that the truncation $\mathrm{t}_0(\Gr_X(S)(D,Z))$, hence $\mathrm{t}_0(\Gr^{\loc}_X(S)(D,Z))$, is actually a \emph{sheaf of sets} (Lemma \ref{GrIsSetsValued}). In particular the global sections $\mathrm{Gr}_X(S)(D,Z)$ of both $\Gr_X(S)(D,Z)$ and $\Gr^{\loc}_X(S)(D,Z)$, form a set.  Since $(X\times S)\smallsetminus D$, $\widehat{D}$, $\widehat{D}^\affinize$, $\hZ$, $\hZ^\affinize$, and $\widehat{D}\smallsetminus D$ are all insensitive to the scheme structure on $D$ and $Z$, $\GrX_X(S)(D,Z)$, $ \GrX^{\loc}_X(S)(D,Z)$, $\Gr_X(S)(D,Z)$, and  $\Gr^{\loc}_X(S)(D,Z)$ are \emph{topologically invariant} (\S $\,$ \ref{topinv}).  With the test scheme $S$ fixed, by using again formal gluing, one can prove that the assignment 
 $(D,Z) \mapsto \mathrm{gr}_X(S)(D,Z)$ defines a functor $\mathrm{Gr}_X(S): \flags_{X}(S) \to \mathbf{Sets}, $ that are suitably natural in the argument $S$ (see \S $\,$ \ref{funtorialitaGr}). \\ The Grothendieck construction applied to $\mathrm{Gr}_X(S): \flags_{X}(S) \to \mathbf{Sets} $ defines $p_1(S): \fibGr_X(S) \to \flags_X(S)$, and the fuctoriality in $S$ of $\mathrm{Gr}_X(S)$ (established in \S $\,$ \ref{funtorialitaGr}), builds for us $\fibGr_X$ as a functor $\dAff\op \to \inftyCat$, together with a morphism of functors\footnote{Here we use the same symbol both for $\flags_X$ and for its canonical derived extension (i.e. its Kan extension to derived affine schemes). }
\begin{equation}\label{intromaptoflags} p_1: \fibGr_X \longrightarrow \flags_X.\end{equation}
In other words, when $S$ is an underived test affine scheme, the objects of $\fibGr_X(S)$ consist of pairs $(F, \underline{\mathcal{E}})$ where $F$ is a flag in $X \times S$, and $\underline{\mathcal{E}} \in \mathrm{Gr}_X(S)$.  
Note that the trivial bundle induces an obvious section $\mathsf{triv}: \flags_X \to \fibGr_X$ of $p_1$.\\ 
One can show that the map $p_1: \fibGr_X \to \flags_X$ carries a canonical \emph{flat connection} (\S $\,$ \ref{funtorialitaGr}). This is conveniently and compactly rephrased by asserting the existence of a map of derived stacks $\fibGr'_X \to \flags^{DR}_X$, to the de Rham stack of $\flags_X$, whose pullback along the canonical map $\flags_X \to \flags^{DR}_X$ recovers the morphism $p_1$. This should be seen as the analog for surfaces of the existence of a canonical flat connection on Beilinson-Drinfeld Grassmannian (see \cite[Prop. 3.1.20]{Zhu2017}). 
A further important fact (Theorem \ref{thmRepresentabilityfibGr}) is that $\fibGr_X \to \flags^{DR}_X$ (and also $p_1: \fibGr_X \to \flags_X$) is, in an appropriate sense, represented by a derived \emph{ind-scheme} (not of finite type\footnote{A stronger finiteness is obtained by restricting to the sub-functor $\underline{\mathsf{Car}}_{X} \hookrightarrow \flags_{X}$ of relative effective Cartier divisors on $X$ (i.e. restricting to flags of type $(D, \emptyset)$): such a restriction is, in fact, represented by an \emph{ind-quasi-projective} ind-scheme, hence of finite type.}).\\
We are now able to construct the simplicial flag Grassmannian $\fibGr_{X, \bullet}$ together with a simplicial map $p_{\bullet}: \fibGr_{X, \bullet} \to \flags_{X, \bullet}$ (see \S $\,$ \ref{subsec:2segalflaggrass}). 
First of all we put $\fibGr_{X, 0}= \left\{ * \right \}$, $\fibGr_{X, 1}=\fibGr_X$, and define $\fibGr_{X, 2}$ via the following pull-back diagram \begin{equation}\label{introGr2} \xymatrix{\fibGr_{X,2} \ar[d]_-{p_2 :=} \ar[r]^-{=:\partial_1} & \fibGr_{X} \ar[d]^-{p_1} \\
\flags_{X, 2} \ar[r]_-{\partial_1 = \cup} & \flags_{X}.}\end{equation} 
A non-trivial result (see Section \ref{mappediproiezione}), obtained through formal gluing inside punctured formal neighbourhoods, allows us to define\footnote{For simplicity, we only state here the results for the global version $\mathrm{Gr}_X(S)$. } \emph{projection maps} 
\begin{equation}\label{introproj} q_i (S; F_1, F_2): \mathrm{Gr}_X(S)(F_1 \cup F_2) \longrightarrow \mathrm{Gr}_X(S)(F_i) \,, \, i=1,2,\end{equation} for arbitrary flags $(F_1, F_2) \in \flags_{X,2}(S)$ that are suitably functorial both with respect to maps $(F_1, F_2) \to (F'_1, F'_2)$ in $\flags_{X,2}(S)$, and with respect to maps $S' \to S$ between test schemes. Thus, the projection maps of (\ref{introproj}) induce morphisms $\partial_0, \partial_2: \fibGr_{X,2} \to \fibGr_{X,1}$ such that the diagram $$\xymatrix{\fibGr_{X,2} \ar[d]_-{p_2} \ar[r]^-{\partial_j} & \fibGr_{X,1} \ar[d]^-{p_1} \\ \flags_{X, 2} \ar[r]_-{\partial_j} & \flags_{X,1}} $$ commutes, for $j=0, 2$. Using the empty flag to build degeneracy maps, we get a 2-truncated simplicial object $\fibGr_{X,\leq 2}$.  Now, via a general extension result (\S $\,$  \ref{BenjExtension}), we complete $\fibGr_{X,\leq 2}$ to a full simplicial object $\fibGr_{X,\bullet}$, called the \emph{simplicial flag Grassmannian} of the pair $(X, \mathbf{G})$, equipped with a canonical simplicial map $p_{\bullet}: \fibGr_{X, \bullet} \to \flags_{X, \bullet}$, extending $p_1$ and $p_2$. The $2$-Segal property of $\flags_{X, \bullet}$, and the properties of $p_{\bullet}$, imply that the simplicial flag Grassmannian $\fibGr_{X, \bullet}$ is again $2$\emph{-Segal}. One can furthermore prove that the simplicial map $p_{\bullet}$ is endowed with a canonical \emph{flat connection} (extending the one already established on $p_1$), or equivalently, that the previous construction yields in fact a map of simplicial objects $\fibGr'_{X, \bullet} \to \flags^{DR}_{X, \bullet}$ inducing $p_{\bullet}: \fibGr_{X, \bullet} \to \flags_{X, \bullet}$ via pullback along the canonical map  $\flags_{X, \bullet} \to \flags^{DR}_{X, \bullet}$.\\

\noindent \textbf{Flag-Factorization.} In Section \ref{sectionfactorization} we prove the \emph{flag factorization property} for the flag Grassmannian $\fibGr_X$. More precisely, and in complete analogy with (\ref{BDfactorization}), we prove (see Remark \ref{remfattorizzazione}) that the map $(\partial_0, \partial_2): \fibGr_{X,2} \to \fibGr_{X, 1} \times \fibGr_{X,1}$ establishes an isomorphism\footnote{This isomorphism is obviously compatible with the section $\mathsf{triv}: \flags_X \to \fibGr_X$ of $p_1$, introduced above. Moreover, Remark \ref{Grtildeandisos} gives the analog of the natural cocycle condition on $(\mathsf{Ran}_C \times \mathsf{Ran}_C \times \mathsf{Ran}_C )_{\mathrm{disj}}$ satisfied by the factorization isomorphism for Beilinson-Drinfeld Grassmannian (\cite[Theorem 3.3.3]{Zhu2017}).} (called the \emph{factorization isomorphism}) between the pullbacks of the diagrams 
\begin{equation}\label{introfactorizationforflagGr} \xymatrix{\fibGr_{X, 1} \times \fibGr_{X,1} \ar[d]_-{p_1 \times p_1} & & \fibGr_{X,1} \ar[d]^-{p_1} \\ \flags_{X, 1} \times \flags_{X,1}  & \flags_{X,2} \ar[l]_-{\supset} \ar[r]^-{\partial_1} & \flags_{X, 1}  } \end{equation}
The proof boils down to showing that the projection maps (\ref{introproj}) induce an isomorphism 
\begin{equation}\label{introprojmapsareisos} \xymatrix{(q_1 (S; F_1, F_2), q_2 (S; F_1, F_2)) : \mathrm{Gr}_X(S)(F_1 \cup F_2) \ar[r]^-{\sim} & \mathrm{Gr}_X(S)(F_1) \times \mathrm{Gr}_X(S)(F_2)}\end{equation}
for any pair $(F_1,F_2) \in \flags_{X, 2}(S)$, and for any $S$ (Theorem \ref{thm:factorization}). One important and non-trivial ingredient of the proof (Lemma \ref{lem:glueperfformal}) consists in the equivalence  \[
\xymatrix{  \Perf_{\widehat{T_1 \cup T_2}} \ar[r]^-{\sim} & \Perf_{\widehat T_1} \newtimes_{\Perf_{\widehat{T_{1} \cap T_2}}} \Perf_{\widehat T_2}}
 \] for arbitrary (derived) closed subschemes $T_1, T_2$ of an arbitrary affine (derived) scheme. \\

\noindent \textbf{Flag-dependent exotic structures on the stack of $\mathbf{G}$-bundles.} For a (underived) Noetherian scheme $S$ and a smooth morphism $Y \to S$ of relative dimension $\geq 2$, we say that a flag $(D,Z)$ in $Y/S$ is \emph{large} if $Z_s$ meets any irreducible component of $D_s$, for any $s \in S$. As A. Beilinson observed (private communication), the set of global sections $$\mathrm{Gr}(Y; D,Z)= \mathrm{Bun}^{\mathbf{G}}(Y) \times_{\mathrm{Bun}^{\mathrm{G}}(Y\setminus D) \times_{\mathrm{Bun}^{\mathbf{G}}(\widehat{Z} \setminus D)} \mathrm{Bun}^{\mathbf{G}}(\widehat{Z})    } \left\{ \overline{\mathcal{E}} \right\}$$ is a singleton\footnote{Note that this result is clearly false for the Perf-version of the flag Grassmannian.}, for any fixed $\mathbf{G}$-bundle  $\overline{\mathcal{E}}$ on $Y$ (Proposition \ref{beilinson}). This observation allows us, in Section \ref{sessioneesotica}, to associate to any large flag on a smooth projective surface $X$ a derived structure on the classical stack of $\mathbf{G}$-bundles on $X$ (Proposition \ref{lemexotic}). These derived structures (given by the corresponding \emph{derived Hecke stacks}, see Section \ref{esotiche}) are, in general, different from the usual one, i.e. $\underline{\mathsf{Bun}}_{X}$ (see Examples \ref{P2O3} and \ref{P2TP2}), and therefore we call them \emph{exotic} derived structures on the truncation  $\mathrm{t}_0(\underline{\mathsf{Bun}}_{X})$, i.e. on the classical stack of $\mathbf{G}$-bundles on $X$. We remark that different choices of large flags in $X$ yield different exotic structures.\\

\noindent \textbf{Flag chiral structures.} In Section \ref{sectionfusion} we define flag-analogs of chiral algebras and chiral categories. In the classical case, i.e. when $C$ is a smooth projective curve, one can define the so called \emph{chiral tensor product} on appropriate categories $\mathsf{Shv}(\mathsf{Ran}_C)$ of sheaves on the Ran space of $C$, essentially by performing the pull-push construction (of the exterior tensor product of such sheaves) along the following diagram
\begin{equation}\label{ranchiral}
\xymatrix{& (\mathsf{Ran}_C \times \mathsf{Ran}_C)_{\mathrm{disj}} \ar[rd]^-{p:=\cup} \ar[ld]_-{q} & \\ \mathsf{Ran}_C \times \mathsf{Ran}_C  && \mathsf{Ran}_C}
\end{equation}
where $q$ is the obvious inclusion (see \cite{FG} for details). When $X$ is a smooth projective surface, in order to get a flag-analog of this chiral product, one might naively perform the same construction by replacing (\ref{ranchiral}) by the following diagram 
\begin{equation}\label{flagchirale} \xymatrix{& \flags_{X,2} \ar[rd]^-{p:=\partial_1} \ar[ld]_-{q:=(\partial_2, \partial_0)} & \\ \flags_X \times \flags_X && \flags_X}
\end{equation} where recall that $\partial_1$ is the union map of good flags. This can certainly be done if we forget maps between flags (i.e. considering $\flags_{X}$ and $\flags_{X,2}$ as Sets-valued rather than Posets-valued), see $\S$ $\,$ \ref{warmup}. In order to deal with the general case, we use the tools developed by S. Raskin\footnote{Note that in \cite[Rem. 6.19.5]{samchiral} Raskin proves the equivalence of his approach with Francis-Gaitsgory's one.} (\cite{samchiral}), together with Beilinson-Drinfeld's idea in \cite[3.4.6]{Beilinson_Drinfeld_Chiral_2004}, to get a definition of \emph{flag-chiral category} and of \emph{quasi-coherent flag-factorization algebra}, together with their \emph{crystal} and \emph{$\mathcal{D}$-Modules} counterparts (Definition \ref{flagchiralcatsandalgs}). Some examples of these structures are given.\\

\noindent \textbf{Leftovers and future directions.} 
We list here some possible further lines of research naturally suggested by this paper.\\

\textsf{1) Actions.} In Appendix \ref{sectionactions} we sketch the construction of flag versions of the \emph{loop} and \emph{positive loop group} (both fibered over $\flags_X$), well known in the curve case, together with their actions on the flag Grassmannian. This construction should be rigorously enhanced to the derived setting, and one should then be able to make sense, at least for $\mathbf{Shv}$ denoting $\ell$-adic sheaves or $\mathcal{D}$-Modules, of the equivariant $\infty$-category  $\mathbf{Shv}^{\mathcal{L}^{+}G_X}(\fibGr_{X})$ of $\mathcal{L}^{+}G_X$-equivariant sheaves on $\fibGr_{X}$, endowed with an appropriate convolution monoidal structure. This should also be extended to the other falg Grassmannians introduced in 2) and 3) below.\\

\textsf{2) Mapping down and other flag Grassmannians.} Fix a smooth divisor $D_0$ in a smooth projective surface $X$. For each test scheme $S$, and any flag on $X\times S /S$ of the form $(D_0 \times S, Z)$, we may consider
$$\underline{\mathsf{Gr}}_{D_0}(S;Z) := \underline{\mathsf{Bun}}_{D_0 \times S} \times_{\underline{\mathsf{Bun}}_{(D_0 \times S)\setminus Z}} \{ \mathrm{triv} \}$$ i.e. the affine Grassmannian of the curve $D_0$ at the ``points'' $Z$. We may try to construct a map $\underline{\mathsf{Gr}}(X\times S)(D_0 \times S, Z) \to \underline{\mathsf{Gr}}_{D_0}(S;Z)$, and varying both $S$ and $Z \subset D_0 \times S$, a map \begin{equation}\label{mappingdown}\fibGr_{X;D_0} \longrightarrow \fibGr_{D_0}\end{equation} where $\fibGr_{D_0}$ is (a Cartier-divisors version of) Beilinson-Drinfeld's Ran Grassmannian of the curve $D_0 \subset X$.
 A drawback of the definition of flag Grassmannian in this paper is that this map is necessarily trivial (i.e. factors through the trivial bundle). This is essentially due to the compatibility required between the two trivializations in (\ref{boldGrfixedflag}). Thus, in order to get a non-trivial map, one may modify our definition of the flag Grassmannian by replacing (\ref{boldGrfixedflag}) with either of the following\footnote{We have written here derived stacks over $k$, using notation (\ref{functortoabsolutestacks}).  }
 $$\underline{\mathsf{Gr}}_X^{\mathrm{glob}}(S)(D,Z):= \underline{\mathsf{Bun}}_{X \times S} \times_{\underline{\mathsf{Bun}}_{\widehat{D}\setminus Z} \times \underline{\mathsf{Bun}}_{\widehat{Z}}} \{ \mathrm{triv} \} \,\, ,\qquad 
\underline{\mathsf{Gr}}^{\mathrm{glob, res}}_{X}(S)(D,Z):= \underline{\mathsf{Bun}}_{X \times S} \times_{\underline{\mathsf{Bun}}_{\widehat{D}\setminus Z}} \{ \mathrm{triv} \}$$
$$\underline{\mathsf{Gr}}_X^{\mathrm{fat}}(S)(D,Z):= \underline{\mathsf{Bun}}_{\widehat{D}} \times_{\underline{\mathsf{Bun}}_{\widehat{D}\setminus Z} \times \underline{\mathsf{Bun}}_{\widehat{Z}}} \{ \mathrm{triv} \} \,\, ,\qquad \underline{\mathsf{Gr}}^{\mathrm{fat, res}}_{X}(S)(D,Z):= \underline{\mathsf{Bun}}_{\widehat{D}} \times_{\underline{\mathsf{Bun}}_{\widehat{D}\setminus Z}} \{ \mathrm{triv} \}.$$
As done with the flag Grassmannian of this paper, one can show (\cite{hmv2}) that there are corresponding $2$-Segal simplicial objects $\fibGr^{\mathrm{glob}}_{X, \bullet}$, $\fibGr^{\mathrm{glob, res}}_{X, \bullet}$  in $\dSt^{\mathrm{Cat}_{\infty}}/\underline{\mathsf{Bun}}_{X}$ (with $\fibGr^{\mathrm{glob}}_{X, 0}= \fibGr^{\mathrm{glob, res}}_{X, 0} := \underline{\mathsf{Bun}}_{X}$)\footnote{Hence we have corresponding algebras in correspondences in $\dSt/\underline{\mathsf{Bun}}_{X}$.}, and $2$-Segal objects $\fibGr^{\mathrm{fat}}_{X, \bullet}$, $\fibGr^{\mathrm{fat, res}}_{X, \bullet}$ in $\dSt^{\mathrm{Cat}_{\infty}}$, together with simplicial maps $\fibGr^{\mathrm{glob}}_{X, \bullet} \to  \flags^{DR, \mathrm{op}}_{X, \bullet} \times \underline{\mathsf{Bun}}_{X}$, $\fibGr^{\mathrm{glob, res}}_{X, \bullet} \to \flags^{DR, \mathrm{op}}_{X, \bullet} \times \underline{\mathsf{Bun}}_{X}$, $\fibGr^{\mathrm{fat}}_{X, \bullet} \to \flags^{DR, \mathrm{op}}_{X, \bullet} $ and $\fibGr^{\mathrm{fat, res}}_{X, \bullet} \to \flags^{DR, \mathrm{op}}_{X, \bullet}$. Moreover, one can show that $\fibGr^{\mathrm{glob}}_{X, \bullet}$, $\fibGr^{\mathrm{glob, res}}_{X, \bullet}$, $\fibGr^{\mathrm{fat}}_{X, \bullet}$ and $\fibGr^{\mathrm{fat, res}}_{X, \bullet}$ all satisfy an appropriate flag-factorization property\footnote{And the same is true for the obvious corresponding Hecke stacks.}. For this modified flag Grassmannians, the corresponding maps (\ref{mappingdown}) are non-trivial.  This also suggests the possibility that suitably defined Hecke operators on $X$ induce Hecke operators on the curve $D_0$, perhaps generating interesting subalgebras of Hecke operators on $D_0$. Investigations in this directions are ongoing (\cite{hmv2}).\\

\textsf{3) More local flag Grassmannians and the double affine Grassmannian.}  By considering even more local Grassmannians with respect to those considered above in 2), one can try to give an algebro-geometrical interpretation of the affine Grassmannian, as discussed\footnote{At the very end of Section 3, the authors write ``It is not clear how to think about the affine Grassmannian of $\mathbf{G}_{\mathrm{aff}}$ in terms of algebraic geometry''.} in \cite[Section 3]{braverman2010pursuing}.  Indeed one can run the flag Grassmannian construction of this paper by replacing (\ref{boldGrfixedflag}) with $$\underline{\mathsf{Gr}}_X^{\mathrm{double}}(S)(D,Z)=\underline{\mathsf{Bun}}_{(\widehat{Z}^{\widehat{D}})^{\mathrm{aff}} \setminus D} \times_{\underline{\mathsf{Bun}}_{(\widehat{Z}^{\mathrm{aff, \widehat{D}}}\setminus Z)^{\mathrm{aff}} \setminus D}} \{ \mathrm{triv} \}$$ where notations need some explanation. First of all, $(\widehat{Z}^{\widehat{D}})^{\mathrm{aff}}$ is the affinization of the ind-scheme $\widehat{Z}^{\widehat{D}}= \{ \widehat{Z}^{D_n}\}_{n\geq 0}$, where $\widehat{Z}^{D_n}$ is the formal completion of $Z$ inside the $n$-th formal neighborhood $D_n$ of $D$ inside the ambient scheme $X \times S$. Secondly,  $\widehat{Z}^{\mathrm{aff, \widehat{D}}}$ is the ind-scheme $\{ \widehat{Z}^{\mathrm{aff}, D_n}\}_{n\geq 0} $ (where $\widehat{Z}^{\mathrm{aff}, D_n}$ is the affinization of the completion of $Z$ inside $D_n$); then we subtract $Z$ from this ind-scheme, and we affinize it before finally subtracting $D$ : this is the meaning\footnote{Another interesting way of interpreting $(\widehat{Z}^{\mathrm{aff, \widehat{D}}}\setminus Z)^{\mathrm{aff}} \setminus D$ is as the (suitably defined) triple intersection of the three ``basic adelic'' opens associated to the flag $D,Z$, i.e. $(X\times S)\setminus D$, $\widehat{D} \setminus Z$, and $\widehat{Z}$.} of $(\widehat{Z}^{\mathrm{aff, \widehat{D}}}\setminus Z)^{\mathrm{aff}} \setminus D$. Note that, with this definition, for $S= \mathrm{Spec}\, k$, $X= \mathbb{A}_k^2$, $D=\{ y=0 \}$, and $Z=\{ x=y=0 \}$, we get $(\widehat{Z}^{\widehat{D}})^{\mathrm{aff}} \setminus D = \mathrm{Spec} (k[[x]]((y)))$, and $(\widehat{Z}^{\mathrm{aff, \widehat{D}}}\setminus Z)^{\mathrm{aff}} \setminus D = \mathrm{Spec} (k((x))((y)))$, and this explains the relation with 
\cite[Section 3]{braverman2010pursuing}. Further relations between \cite{braverman2010pursuing, BRAVERMAN2012414, BraKaz} and this very local version of the flag Grassmannian are being investigated (\cite{MMV}).\\

\textsf{4) Replacing flags.} One could run the construction of this paper (also for the flag Grassmannians listed in 2) and 3) above) by replacing $\flags_{X,\bullet}$ with either the Ran space of nested Cartier divisors in $X$, the Ran space of $\flags_X$, or even the Ran space of families of nested closed \emph{subsets} of $X$. This approach might also indicate how to extend our constructions to smooth projective varieties $X$ of dimension $> 2$. This is currently  being investigated in \cite{hmv2}.\\

\textsf{5) Open questions.} Finally, here are a few questions (inspired by \cite{Gaitsgory_contract}, and \cite{MV_2007}) we are considering but have, so far, few clues about.\\
Let $\mathcal{G}r_X^{\mathrm{any}}$ be either $\mathcal{G}r_X$ or $\mathcal{G}r^{\mathrm{glob}}$ or $\mathcal{Gr}^{\mathrm{glob, res}}$, and consider the obvious map $\mathcal{G}r^{\mathrm{any}}_X \to \underline{\mathsf{Bun}}_X$. Are the fibers  of this map de Rham-contractible? If not, does this map induce at least an interesting relation between the de Rham cohomology of $\mathcal{G}r^{\mathrm{any}}_X$ and the de Rham cohomology of  $\underline{\mathsf{Bun}}_X$?\\ Using 1) above, given any flag Grassmannian $\mathcal{G}r_X$ (i.e. the one in the paper or the ones listed in 2) and 3) above), can we interpret in a Tannakian-like way the category\footnote{Note that even the definition of this category is not straightforward.} of $\mathcal{L}^{+}G_X$-equivariant perverse sheaves on $\mathcal{G}r_X$? To start with, one might fix a flag, and pose the same question for the flag Grassmannian at this fixed flag. \\

\noindent\textbf{Related works.} It is very much unclear what a topic named ``Geometric Langlands program for surfaces'' should precisely mean or whether it exists at all, at least in the broad structured sense it has in the case of curves. With no claim of exhaustivity, we will merely list, among the ones we are aware of, those sources that actually inspired and motivated our research project. The paper \cite {Ginzburg_Kapranov_Langlands_reciprocity_1995} was an inspiring and pioneering work in this direction but, at least to our knowledge, its line of research have not been pursued since. The article \cite{Kapranov1995} also contains very interesting suggestions about the possible existence of a Langlands program for higher dimensional varieties, though here the emphasis is more on higher local fields and formal analogies with topological quantum field theories. As we already mentioned, A. Beilinson and V. Drinfeld's  \cite{Beilinson_Drinfeld_Chiral_2004} (especially \S \, 3.4.6), though dealing with the curve case, was the actual starting point of our project.\\ Another question about a possible Geometric Langlands program for surfaces, passed on to the third author by M. Kapranov, is whether it should concern $\mathbf{G}$-bundles or $\mathbf{G}$-gerbes. In any case, the setup of formal gluing, nonlinear flags and the flag Grassmannian investigated in our paper for $\mathbf{G}$-bundles makes sense, mutatis mutandis, for $\mathbf{G}$-gerbes, too. However algebraization is likely to be false for $\mathbf{G}$-gerbes\footnote{Betrand To\"en, private communication.}, and this is a problem in developing this direction of research. \\
While finishing the paper, S. Raskin informed us that, in a 2010 conversation with him, I. Mirkovic proposed to use flags of the sort we consider here, in order to extend Belinson-Drinfeld's Grassmannian to higher dimensional varieties. We are not aware of the specific definition of the flag Grassmannian he proposed nor of subsequent publicly available work along these lines.\\ On the other hand, all of the interesting literature on the \emph{double affine Grassmannian} (e.g. \cite{braverman2010pursuing, BRAVERMAN2012414, BraKaz} to list a few papers) seems to go in a different but related direction with respect to our paper. An ongoing collaboration between the second and third authors and Andrea Maffei (Pisa) is aiming at establishing a link between the two approaches (\cite{MMV}).\\

\noindent\textbf{Acknowledgments.} Several people helped in various ways the realization of this paper. M. Porta started this project together with the first and third author, and he contributed some important ideas that shaped the current paper; in particular, sections 2 and 3 should be considered as joint work with him. M. Kapranov shared with us some of his stimulating ideas and visions on the subject of a Geometric Langlands program on higher dimensional varieties. D. Beraldo was great in carefully and patiently explaining us some tricky details in the geometric Langlands program. G. Nocera followed closely the recent evolution of this paper, and generously shared his ideas about the curve case. A. F. Herrero made very interesting comments, and suggested Lemma \ref{GrIsSetsValued}. S. Raskin gave us extremely useful and knowledgeable feedbacks: we have taken into account some of them in the current version (e.g. about the chiral tensor product in \S \ref{sectionfusion}), the rest will be addressed in a subsequent paper. A. Beilinson sent us a detailed proof of Proposition \ref{beilinson}: we thank him a lot both for this and for his interest in our work. N. Rozenblyum and L. Ramero helped us in the proof of a preliminary underived version of Lemma \ref{lem:fiberprodKoszul}. We owe a lot to D. Beraldo's and D. Gaitsgory's brilliant papers on the Geometric Langlands program. C. Barwick, R. Donagi, E. Elmanto, T. Pantev, M. Pippi, and B. To\"en offered valuable comments and encouragement. We sincerely thank all of them.\\

\noindent \textbf{Notations.} Our base field will be $k=\mathbb{C}$. All our schemes will, unless otherwise stated, be $k$-schemes. For a scheme $T$, $\mathbf{Sh}_T$ will denote the category of flat sheaves of sets on $T$. Analogously, $\dSt_T$ will denote the $(\infty,1)$-category of flat derived stacks in $\infty$-groupoids on $T$. We write $\mathbf{Sh}$ and $\dSt$ when $T=\Spec \, k$. We will denote by $\mathbf{PoSets}$ the $1$-category of (small) posetal categories, i.e. (small) categories whose $\Hom$-sets are either empty or a singleton. In a 1-category $C$, if $: u: Z\to X$, and $u': Z'\to X$ are monos, and there exists an isomorphism $u \simeq u'$ in $C/X$, then this isomorphism is unique, and we will sometimes simply write $u=u'$. The other relevant notations will be introduced and explained in the main text.\\

\section{Sheaves and bundles on the punctured formal neighbourhood}\label{section:sheaves_on_PFN}

In this first section, we define and construct the stacks of bundles (or (pseudo-)perfect complexes) on punctured formal neighborhoods in a very general setting. This will allow us to state and prove, in the next section, a statement à la Beauville--Laszlo for general (derived) stacks.

\subsection{Recollection on derived stacks}

We start by fixing a handful of notations from derived algebraic geometry.
For a comprehensive review of the subject, we refer the reader to \cite{Toen_Derived_2014}.

\begin{defin} Let us recall a few definitions, and fix our notations.
\begin{itemize}
\item We will denote by $\sCAlg_k$ the $\infty$-category of simplicial commutative $k$-algebras up to homotopy equivalence.

\item For any $A \in \sCAlg_k$, we will denote by $\pi_i A$ its $i^\mathrm{th}$ homotopy group. Recall that $\pi_0 A$ is endowed with a commutative $k$-algebra structure, and that for any $i$, the abelian group $\pi_i A$ is endowed with a structure of $\pi_0 A$-module.

\item A simplicial commutative algebra $A$ is Noetherian if $\pi_0 A$ is Noetherian and $\pi_i A$ is of finite type over $\pi_0 A$.
We denote by $\sCAlgNoeth_k$ the $\infty$-category of Noetherian simplicial commutative algebras.

\item A (locally Noetherian) derived prestack is a functor $\sCAlgNoeth_k \to \inftyGpd$ with values in the category of $\infty$-groupoids. A derived stack is a derived prestack satisfying the étale hyperdescent condition. We will denote by $\dSt_k$ the $\infty$-category of derived stacks. For $X \in \dSt_k$, we denote by $\dSt_X := \dSt_k/X$ the category of stacks over $X$.

\item We denote by $\dAff_k \subset \dSt_k$ the full subcategory of (Noetherian) derived affine schemes (i.e. functors represented by some $A \in \sCAlgNoeth_k$).

\item For any derived stack $X$, we will denote by $\trunc X$ its restriction to (discrete) commutative $k$-algebras. The functor $\trunc X$ is then a (non-derived) stack. It comes with a canonical morphism $\trunc X \to X$. We say that $X$ is a derived enhancement of $\trunc X$.
\end{itemize}
\end{defin}

Given a derived affine scheme $X$ and a closed subset $Z$ of the scheme $\trunc X$, the open complement $\trunc X \smallsetminus Z$ has a canonical derived enhancement as a derived subscheme of $X$. We call it the open complement of $Z$ in $X$ and denote it by $X \smallsetminus Z$.

\begin{defin}
 Let $X \in \dSt_k$.
 \begin{itemize}
 \item A closed substack $Z \subset X$ is the datum of a non-derived stack $Z$ and a map $Z \to X$ such that for any $U \in \Aff_k$ and any $U \to X$, the projection $\trunc(Z \times_{ X} U) \to U$ is a closed immersion of (affine) schemes.
 \item Given a closed substack $Z \subset X$, its open complement is the derived stack
 \[
  X  \smallsetminus Z := \colim_{U \in \dAff /X} U \smallsetminus Z_U,
 \]
where $Z_U$ stands for the pullback $Z \times_X U$.
 \end{itemize}
\end{defin}

\begin{defin}\label{defin:dR}
Let $X$ be a derived stack. We define its de Rham stack $X_\mathrm{dR}$ as the derived prestack whose $A$-points are:
\[
X_\mathrm{dR}(A) = X((\pi_0 A)_\red).
\]
It is endowed with a canonical morphism $X \to X_\mathrm{dR}$.
\end{defin}

\begin{defin}\label{fcompl}
Let $Z \to X$ be a map of derived stacks. The (\emph{derived}) \emph{formal completion} of $Z$ in $X$ (or of $X$ along $Z$) is defined as the pullback
\[
\hZ = X \newtimes^{\mathrm{d}}_{X_\mathrm{dR}} Z_\mathrm{dR} \in \dSt_k.
\]
\end{defin}

\begin{rem}
The notation $\hZ$ is ambiguous as the formal completion of $Z$ in $X$ strongly depends on the map $Z\to X$.
On the other hand, the de Rham stack associated to $Z$ only depends on its truncation $\trunc Z$. In particular, the formal completion $\hZ$ does not depend on the derived structure on $Z$.
\end{rem}

\begin{lem}\label{lem:formalcompletion:basechange}
 Let $Z \to X \leftarrow Y \in \dSt_k$ and let $Z_Y$ be the fiber product $Z_Y = Z \times^\mathrm{d}_X Y$. Let $\hZ$ be the formal completion of $Z$ in $X$ and $\hZ_Y$ be the formal completion of $Z_Y$ in $Y$ (i.e. the formal completion along the projection map $Z_Y\to Y$). The canonical morphism
 \[
  \hZ_Y \to Y \newtimes_X \hZ
 \]
 is an equivalence.
\end{lem}

\begin{proof}
 This is immediate, since the functor $X \mapsto X_\mathrm{dR}$ preserves fiber products.
\end{proof}

The following statement justifies the above definition of the derived formal completion:
\begin{thm} \label{thm:dag_completion_truncation}
Let $X$ be a derived scheme and $Z \subset X$ a closed substack (subscheme). The following assertions hold.
\begin{enumerate}
\item (Gaitsgory-Rozenblyum) The formal completion $\hZ$ is a stack and is representable by a derived ind-scheme.\label{formalcompletion:indscheme}
\item The truncation $\trunc(\hZ)$ is canonically isomorphic to the (usual) formal completion of $\trunc(X)$ along $\trunc(Z)$. \label{formalcompletion:truncation}
\end{enumerate}
\end{thm}

\begin{proof}
Assertion (\ref{formalcompletion:indscheme}) is \cite[Proposition 6.3.1]{Gaitsgory-Rozenblyum:dgindschemes}.
Let us prove assertion (\ref{formalcompletion:truncation}). Using \cref{lem:formalcompletion:basechange}, we can assume that $X$ is affine. Let $A$ denote the algebra of functions on $\trunc X$ and let $I \subset A$ be an ideal defining $\trunc Z$ in $\trunc X$.

The truncation functor $\trunc$ preserves fiber products. It follows that a $B$-point of $\trunc(\hZ)$ is a commutative diagram
\[
 \begin{tikzcd}
  A \ar{r} \ar{d} & B \ar{d} \\ A/I \ar{r} & B/\sqrt{B},
 \end{tikzcd}
\]
where $\sqrt{B} \subset B$ is the nilradical of $B$.
Since $B$ is Noetherian, it amounts to a map $f \colon A \to B$ such that $f(I)$ is nilpotent in $B$. The functor $\trunc(\hZ)$ is therefore represented by the ind-scheme $\colim \Spec(A/I^p)$.
\end{proof}

\subsection{Categories of modules}
\begin{defin}
	Let $A \in \sCAlg_k$ be a simplicial commutative algebra.
	Let $\Perf(A)$ denote the full stable $\infty$-subcategory of $\QCoh(A)$ spanned by its compact objects. 
	Let $\Coh^-(A)$ denote the full stable $\infty$-subcategory of $\QCoh(A)$ spanned by almost perfect $A$-modules (\cite[Def. 7.2.4.10]{Lurie_Higher_algebra})
\end{defin}

\begin{rem}
 If $A \in \CAlg_k$, then a complex $\cF$ belongs to $\Coh^-(A)$ if and only if 
	\begin{enumerate}
		\item for every integer $i \in \mathbb Z$, the cohomology group $\rH^i(\cF)$ is coherent over $A$ and
		\item $\rH^i(\cF) = 0$ for $i \gg 0$.
	\end{enumerate}
\end{rem}

The assignment $\Spec(A) \mapsto \Coh^-(A)$ can be promoted to an $\infty$-functor
\[ \bfCoh^{\otimes,-} \colon \dAff_k\op \to \Catstmon , \]
using the $(-)^*$ functoriality. It comes with a pointwise fully faithful natural transformation $\bfCoh^{\otimes,-} \to \bfQCoh^\otimes$.

Since $\bfQCoh^{\otimes}$ is an hypercomplete sheaf for the \'etale topology, and since being coherent is a local property, the functor $\bfCoh^{\otimes,-}$ is an hypercomplete sheaf for the \'etale topology.
In particular, both $\bfQCoh^\otimes$ and $\bfCoh^{\otimes,-}$ extend uniquely into limit-preserving functors
\[
\bfQCoh^{\otimes},\,\bfCoh^{\otimes,-} \colon \dSt_k\op \to \Catstmon.
\]

\subsection{Fiber functors and punctured neighbourhoods}
The first issue arising in the study of punctured formal neighbourhoods is that whenever $Z$ is a closed subset of a scheme $X$, the formal scheme $\hZ\smallsetminus Z$ is empty. To circumvent this, we will see that any (formal) scheme is determined by what we call its fiber functor, and that the punctured formal neighbourhood admits a somewhat natural (and non-trivial) fiber functor (which is not representable).

\begin{defin}\label{defin:fibf}
 We fix $X\in \dSt_k$. Recall that $\dAffX$ is the category of Noetherian affine derived $k$-schemes equipped with a morphism to $X$. Denote by $\fibf_X$ the forgetful functor $\dAffX \to \dSt_k$ mapping $S \to X$ to $S$.
 A \emph{fiber functor over $X$} is a functor
 \[
  \fibf \colon \dAffX \to \dSt_k
 \]
 together with a natural transformation $\fibf \to \fibf_X$.
 We denote by $\dFibF_X = \Fun(\dAffX, \dSt_k)/\fibf_X$ the $\infty$-category of fiber functors over $X$.
 
 For any stack $Y$ over $X$, we denote by $\fibf_Y$ the functor
 \[
 \fibf_Y \colon 
  \begin{aligned}
   \dAffX & \to \dSt_k \\
   U & \mapsto U \newtimes_{X} Y
  \end{aligned}
 \]
 This construction being functorial, any $Y \in \dSt_k$ yields a fiber functor $\fibf_Y \to \fibf_X$.
 We call it the \emph{fiber functor of $Y$} (over $X$).
\end{defin}

\begin{rem}
 All in all, a fiber functor associates to every $U \to X$ a stack over $U$, with only weak compatibilities. They can be thought as ``non-quasi-coherent stacks'' over $X$.
 
 Remark also that the terminal fiber functor $\fibf_X$ maps to the constant functor with value $X$, so that fiber functors can be seen as having values in $\dSt_X$.
\end{rem}

\begin{eg}
 Let $X \in \dSt_k$ and $Z \subset X$ a closed substack. The formal completion $\hZ := \widehat X_Z$ (Definition \ref{fcompl}) yields a fiber functor $\fibf_\hZ$. \cref{lem:formalcompletion:basechange} implies that for any $U \in \dAffX$, the stack $\fibf_\hZ(U)$ is canonically equivalent to the completion $\hZ_U$ of $Z_U := Z \times_X U$ inside $U$.
\end{eg}

\begin{lem}\label{lem:fibf-ff}
Assume $X$ is locally Noetherian.
 The functor $\fibf_\bullet \colon \dSt_X \to \dFibF_X$ mapping $Y \in \dSt_X$ to $\fibf_Y$ is fully faithful. Its essential image consists of fiber functors $\fibf \to \fibf_X$ such that for any $U \to V \in \dAffX$, the induced $\fibf(U) \to \fibf(V) \times_V U$ is an equivalence.
\end{lem}

\begin{proof}
 The functor $\fibf_\bullet$ admits a left adjoint computing the colimit of the underlying functor $\dAffX \to \dSt_k$. Since colimits in $\dSt_k$ are universal, the adjunction morphism 
  \[
   \colim_{U \in \dAffX} U \newtimes_X Y \to Y
  \]
  is an equivalence. In particular $\fibf_\bullet$ is fully faithful.
  
  To compute its essential image, we interpret the notion of fiber functors in terms of Grothendieck constructions. Let $t \colon \dSt_k^{\Delta^1} := \Fun(\Delta^1, \dSt_k) \to \dSt_k$ be the target functor (i.e. evaluation at $1$). It is a Cartesian fibration classifying the functor $U \mapsto \dSt_U$ sending maps to pullbacks (see \cite[Cor. 2.4.7.12 and 5.2.2.5]{HTT}). A fiber functor $\fibf$ over $X$ is tantamount to a section $s_\fibf$ of $t$ over $\dAffX$:
 \[
  \begin{tikzcd}
   & \dSt_k^{\Delta^1} \ar{d}{t} \\ \dAffX \ar{r}[swap]{\fibf_X} \ar[dashed]{ur}{s_\fibf} & \dSt_k.
  \end{tikzcd}
 \]
 Given a fiber functor $\fibf$, the condition that for any $U \to V \in \dAffX$, the map $\fibf(U) \to \fibf(V) \times_V U$ is an equivalence is equivalent to the condition that the section $s_\fibf$ is Cartesian. By \cite[Cor. 3.3.3.2]{HTT}, the category of Cartesian sections is equivalent to the limit category $\lim_{U \in \dAffX} \dSt_U$. Since the functor $U \mapsto \dSt_U$ is a stack (the stack of stacks, see for instance \cite[Thm. 6.1.3.9]{HTT}), this limit is equivalent (through the base change functors) to $\dSt_X$. It follows that the essential image of $\fibf_\bullet$ is indeed as announced.
\end{proof}

\begin{defin}
 A fiber functor $\fibf$ is called \emph{representable} if it belongs to the essential image of $\fibf_\bullet$.
\end{defin}

\begin{defin}
 Let $X \in \dSt_k$ and let $Z \subset X$ be a closed substack. 
 The map $\fibf_Z \to \fibf_X$ of fiber functors is a pointwise closed immersion of derived affine schemes. Any morphism $U \to V \in \dAffX$ maps the open complement of $Z_U = Z \times_X U$ to the open complement of $Z_V = Z \times_X V$. In particular, we have a fiber functor
 \[
  \fibf_{X \smallsetminus Z} \colon U \mapsto U \smallsetminus Z_U.
 \]
 We easily check using \cref{lem:fibf-ff} that $\fibf_{X \smallsetminus Z}$ is representable. We denote by $X \smallsetminus Z$ its representative (so that $X \smallsetminus Z \simeq \colim_U U \smallsetminus Z_U$) and call it the open complement of $Z$ in $X$.
 \end{defin}
 
 \begin{defin}\label{def-intornobucato}
 Let $X \in \dSt_k$ and let $Z \subset X$ be a closed substack.
 \begin{itemize}
 \item The \emph{affinized formal neighbourhood of $Z$ in $X$} is the fiber functor $\fibf_{\hZ}^{\affinize}$ obtained by applying a pointwise affinization to $\fibf_{\hZ}$. More explicitly:
 \[
  \fibf_{\hZ}^{\affinize} \colon U \mapsto (\fibf_{\hZ}(U))^\affinize := \Spec\left(\Gamma\left(\hZ_U, \mathcal O_{\hZ_U}\right)\right).
 \]
 Since $\fibf_X$ has values in derived affine schemes, the structure map $\fibf_\hZ \to \fibf_\hZ^\affinize$ factors as $\fibf_\hZ \to \fibf_\hZ^\affinize \to \fibf_X$, giving $\fibf_\hZ^\affinize$ the structure of a fiber functor over $X$.
 \item The \emph{punctured formal neighbourhood of $Z$ in $X$} is the fiber functor $\fibf_{\hZ\smallsetminus Z}$
 \[
  \fibf_{\hZ\smallsetminus Z} \colon U \mapsto \fibf_\hZ^\affinize(U) \newtimes_X (X \smallsetminus Z) \simeq \fibf_\hZ^\affinize(U) \newtimes_U (U \smallsetminus Z_U).
 \]
 \end{itemize}
\end{defin}

\begin{eg}
 Assume $A \in \CAlg_k$ (i.e. $A$ is a usual, discrete commutative $k$-algebra), $U = \Spec A$ and $Z_U = Z \times _X U$ is the scheme $\Spec(A/I)$, for $I \subset A$ an ideal. We have
 \[
  \fibf_\hZ(U) = \Spf(\hA) \hspace{0.3em} \text{,} \hspace{1em} \fibf_\hZ^\affinize(U) = \Spec(\hA) \hspace{1em} \text{and} \hspace{1em} \fibf_{\hZ\smallsetminus Z}(U) = \Spec(\hA) \smallsetminus \Spec(A/I),
 \]
 where $\hA$ is the $I$-completion of $A$.
\end{eg}

\begin{rem}
 The affinization process forgets all about the $I$-adic topology that was contained in the formal neighbourhood. For this reason, the fiber functors $\fibf^\affinize_\hZ$ and $\fibf_{\hZ \smallsetminus Z}$ are (typically) not representable.
\end{rem}

\subsection{Restriction of fiber functors}
Any morphism of locally Noetherian derived stacks $f \colon X \to Y$ induces a projection functor $p_f \colon \dAffX \to \dAffover{Y}$. Since we have $\fibf_X = \fibf_Y \circ p_f$, this in turn induces a restriction functor $f^{-1} \colon \dFibF_Y \to \dFibF_X$ given by
\[
 f^{-1} \fibf \colon (U \to X) \mapsto \fibf(U \to X \to Y).
\]
This is trivially functorial, leading to a contravariant $\infty$-functor $X \mapsto \dFibF_X$.
It satisfies the following properties (whose straightforward proofs are left to the reader)
\begin{lem}\label{lem:restrictionfiberfunctor}
 Let $f \colon X \to Y$ be a morphism of locally Noetherian derived stacks.
 \begin{enumerate}[label=(\arabic*), ref={\cref{lem:restrictionfiberfunctor} (\arabic*)}]
  \item If $T \to Y$ is a derived stack, then $f^{-1}(\fibf_T) \simeq \fibf_{T \times_Y X}$.
  \item\label{sublem:restrictionfflimcolim} The functor $f^{-1}$ preserves small colimits as well as pullbacks.
  \item If $Z \to Y$ is a closed substack, then
  \[
   f^{-1}(\fibf_\hZ) \simeq \fibf_{\hZ_X},
   \hspace{1em}
   f^{-1}(\fibf_\hZ^\affinize) \simeq \fibf_{\hZ_X}^\affinize
   \hspace{1em}\text{and}\hspace{1em}
   f^{-1}(\fibf_{\hZ\smallsetminus Z}) \simeq \fibf_{\hZ_X \smallsetminus Z_X},
  \]
  where $Z_X := Z \times_Y X$.
 \end{enumerate}
\end{lem}

\subsection{Sheaves on punctured formal neighbourhoods}

\begin{defin}
 Let $X \in \dSt_k$ and let $\fibf \colon \dAffX \to \dSt_X$ be a fiber functor over $X$. The stacks of (right bounded) coherent sheaves and perfect complexes of $\fibf$  are respectively the functors $ \left( \dAffX \right)\op \to \Catstmon$ given by:
 \begin{align*}
  \CohX_{\fibf} \colon U &\mapsto \bfCoh^{\otimes,-}(\fibf(U))\\
  \PerfX_{\fibf} \colon U &\mapsto \bfPerf^{\otimes}(\fibf(U))
 \end{align*}
 When $\fibf$ is either $\fibf_Y$ for $Y \in \dSt_X$ or $\fibf^\affinize_{\hZ\smallsetminus Z}$  for $Z \subset X$ a closed substack, we will simply write
 \[
  \CohX_{Y} := \CohX_{\fibf_Y} \hspace{0.3em},\hspace{1em} \CohX_{\hZ\smallsetminus Z} := \CohX_{\fibf_{\hZ\smallsetminus Z}}  \hspace{0.3em},\hspace{1em}   \PerfX_{Y} := \PerfX_{\fibf_Y} \hspace{1em}\text{and} \hspace{1em} \PerfX_{\hZ\smallsetminus Z} := \PerfX_{\fibf_{\hZ\smallsetminus Z}}.
 \]
\end{defin}

\begin{thm}[Mathew \cite{mathew2019faithfully}] \label{thm:flat_descent}
  Let $X$ be a derived stack and $Z \subset X$ a closed substack.
  The functors 
  \[
   \CohX_{\hZ \smallsetminus Z}\hspace{0.7em}\text{and} \hspace{0.7em} \PerfX_{\hZ \smallsetminus Z} \colon \left( \dAffX \right)\op \to \Catstmon
  \]
  are hypercomplete sheaves with respect to the flat topology.
\end{thm}

\begin{proof}
 The case where $X$ is affine is dealt with in \cite[Thm. 7.8]{mathew2019faithfully}\footnote{Note that Mathew works in the context of ring spectra, while we are using simplicial rings. However, since the statement only deals with categories of modules, the difference between those contexts is harmless here. Indeed, the category of modules over a simplicial algebra agrees with that of modules over the geometric realization of said algebra in ring spectra.}.
 The question being local in $X$, the general case follows.
\end{proof}

\subsection{\texorpdfstring{$G$}{G}-bundles on punctured formal neighbourhoods}

In this paragraph, we will extend the results of the previous paragraph to the case of $\mathbf{G}$-bundles, for $\mathbf{G}$ an affine group scheme over $k$. 

\begin{defin}
 Let $X \in \dSt_k$ and $\fibf \in \dFibF_X$.
 We denote by $\BunGX_\fibf \colon (\dAffX)\op \to \inftyGpd$ the derived prestack
 \[
  \BunGX_\fibf \colon U \mapsto \Bun^\mathbf{G}(\fibf(U)).
 \]
Whenever $\fibf = \fibf_Y$ for $Y \in \dSt_X$ or $\fibf = \fibf_{\hZ \smallsetminus Z}$ for $Z \subset X$, we will write
\[
 \BunGX_Y := \BunGX_{\fibf_Y} \hspace{1em} \text{and} \hspace{1em} \BunGX_{\hZ \smallsetminus Z} := \BunGX_{\fibf_{\hZ \smallsetminus Z}}.
\]
\end{defin}

\begin{prop}\label{prop:BunGdescent}
 For $X \in \dSt_k$ and $Z \subset X$ a closed substack, the prestack
 \[
  \BunGX_{\hZ \smallsetminus Z} \colon (\dAffX)\op \to \inftyGpd
 \]
 is a hypercomplete stack for the flat topology.
\end{prop}

The next lemma will be used in the proof of Proposition \ref{prop:BunGdescent}.

\begin{lem}\label{lem:perflocal}
Let $f \colon Y_1 \to Y_2$ be a map of derived stacks.
We denote by $F_\Perf := f^* \colon \Perf(Y_2) \to \Perf(Y_1)$ and $F_\Bun := f^* \colon \Bun^\mathbf{G}(Y_2) \to \Bun^\mathbf{G}(Y_1)$ the pullback functors. Then
\begin{enumerate}[label={\textrm{(\alph*)}}, ref={\cref{lem:perflocal}~\textrm{(\alph*)}}]
 \item\label{sublem:bgperfinj} If $F_\Perf$ is fully faithful, then $F_\Bun$ is fully faithful;
 \item\label{sublem:bgperflocal} If $F_\Perf$ is an equivalence, then $F_\Bun$ is an equivalence.
\end{enumerate}
\end{lem}

\begin{proof}
To prove this lemma, we will use Tannaka duality as proven for instance in \cite[Theorem 3.4.2]{Lurie_Tannaka_duality} (see also \cite[Theorem 4.1]{bhatthalpernleistner}). It gives us a description of the groupoid of $\mathbf{G}$-bundles in terms of monoidal functors. Namely, for any stack $Y$, the map
\[
\alpha_Y \colon \Bun^\mathbf{G}(Y) \simeq \Map(Y, \rB \mathbf{G}) \to \Fun^{\otimes}(\QCoh(\rB \mathbf{G}),\QCoh(Y))
\]
is fully faithful. Note also that we have a canonical embedding
\[
\Fun^{\otimes}(\Perf(\rB \mathbf{G}),\Perf(Y)) \to \Fun^{\otimes}(\QCoh(\rB \mathbf{G}),\QCoh(Y))
\]
given by the inclusion $\Perf(Y) \to \QCoh(Y)$ and the left Kan extension functor along $\Perf(\rB \mathbf{G}) \to \QCoh(\rB \mathbf{G})$.
The functor $\alpha_Y$ actually factors through $\Fun^{\otimes}(\Perf(\rB \mathbf{G}),\Perf(Y))$.
Let now $f \colon Y_1 \to Y_2$ be a morphism.
We get a commutative diagram
\begin{equation}\label{diagramtannakianreconstruction}
\begin{tikzcd}
\Bun^\mathbf{G}(Y_2) \ar[hook,r] \ar[d, "F_{\Bun}"] & \Fun^{\otimes}(\Perf(\rB \mathbf{G}),\Perf(Y_2)) \ar[r,hook] \ar{d}{F_\Perf} & \Fun^{\otimes}(\QCoh(\rB \mathbf{G}),\QCoh(Y_2)) \ar[d, "F_\QCoh"] \\
\Bun^\mathbf{G}(Y_1) \ar[hook,r] & \Fun^{\otimes}(\Perf(\rB \mathbf{G}),\Perf(Y_1)) \ar[r,hook] & \Fun^{\otimes}(\QCoh(\rB \mathbf{G}),\QCoh(Y_1))
\end{tikzcd}
\end{equation}
From what precedes, we get that $F_{\Bun}$ is fully faithful if $F_\Perf$ is. This proves (a).

From now on, we assume $F_\Perf$ is an equivalence.
Recall the description of the essential images of the functors $\alpha_{Y_i}$ given in \cite[Theorem 3.4.2]{Lurie_Tannaka_duality}.
A monoidal functor $\beta \colon \Perf(\rB \mathbf{G}) \to \Perf(Y_1)$ then lies in the essential image of $\alpha_{Y_1}$ if and only if its preimage $F_\Perf^{-1}(\beta)$ lies in the essential image of $\alpha_{Y_2}$.
It follows that $F_{\Bun}$ is an equivalence.
\end{proof}

\begin{rem}
In the above lemma, one could replace the classifying stack of $\mathbf{G}$-bundles $\rB \mathbf{G}$ by any geometric stack with affine diagonal (so that \cite[Theorem 3.4.2]{Lurie_Tannaka_duality} applies).
Stacks satisfying the conclusion of \ref{sublem:bgperflocal} are often called \emph{$\Perf$-local stacks}.
\end{rem}

\begin{proof}[Proof of Proposition \ref{prop:BunGdescent}]
 Given $U \in \dAffX$ and $U_\bullet \to U$ a hypercovering, we consider the morphism in $\dSt_k$
 \[
  f \colon \colim_{[n] \in \Delta} \fibf_{\hZ \smallsetminus Z}(U_n) \to \fibf_{\hZ \smallsetminus Z}(U).
 \]
 By \cref{thm:flat_descent}, $\Perf(f)$ is an equivalence. Using \ref{sublem:bgperflocal}, we get that $\Bun^\mathbf{G}(f)$ is an equivalence.
\end{proof}

\section{Formal glueing and flag decomposition}
The goal of this section is to prove a very general version of the Beauville-Laszlo theorem \cite{Beauville-Laszlo}.
The Beauville-Laszlo theorem states that a vector bundle on a curve $C$ amount to the data of a bundle on the complement of a point $x \in C$, a bundle on the formal neighbourhood of said point $x$, and some glueing datum on the punctured formal neighbourhood. 

\subsection{Algebraization}
We start with an algebraization result, to explain how sheaves or bundles on a formal neighbourhood are to be considered. The main part of its proof is due to Lurie.

\begin{prop}[Lurie]\label{prop:algebraization}
 Let $X \in \dSt_k$ and $Z \subset X$ a closed substack.
 The morphism of fiber functors $\fibf_{\hZ} \to \fibf_{\hZ}^\affinize$ induces equivalences
 \begin{align*}
  \CohX_{\fibf_{\hZ}^\affinize} &\overset\sim\longrightarrow \CohX_{\fibf_\hZ} = \CohX_\hZ,\\
  \PerfX_{\fibf_{\hZ}^\affinize} &\overset\sim\longrightarrow \PerfX_{\fibf_\hZ} = \PerfX_\hZ,\\
  \hspace{1em} \BunGX_{\fibf_{\hZ}^\affinize} &\overset\sim\longrightarrow \BunGX_{\fibf_\hZ} = \BunGX_\hZ.
 \end{align*}
\end{prop}

\begin{proof}
 The case of 
 \[
  \CohX_{\fibf_\hZ^\affinize} \longrightarrow \CohX_\hZ
 \]
 is dealt with in \cite[Theorem 5.3.2]{DAG-XII}.
 This equivalence is symmetric monoidal and therefore induces an equivalence between the dualizable objects, so that we get the equivalence
 \[
  \PerfX_{\fibf_{\hZ}^\affinize} \overset\sim\longrightarrow \PerfX_\hZ.
 \]
 We deal with the last case using \ref{sublem:bgperflocal} on the morphism $\fibf_\hZ(U) \to \fibf_\hZ^\affinize(U)$ for every $U \in \dAffX$.
\end{proof}

\begin{lem}\label{lem:descenthZ}
The functors $\CohX_\hZ$, $\PerfX_\hZ$ and $\BunGX_\hZ$ are hypercomplete sheaves for the flat topology.
\end{lem}

\begin{proof}
 The fiber functor $\fibf_\hZ$ is representable. In particular, it preserves hypercoverings. The result hence follows from usual descent for coherent or perfect complexes, or $\mathbf{G}$-bundles.
\end{proof}

\subsection{The affine case}
We start with the affine case, due to Bhatt in the non-derived setting, and extended by Lurie to derived geometry.

We fix $X \in \dSt_k$ and $Z \subset X$ a closed substack. We have the following diagram of fiber functors
\begin{equation}
 \begin{tikzcd}
& \fibf_{\hZ \smallsetminus Z} \ar{r} \ar{d} & \fibf_{X \smallsetminus Z} \ar{d} \\
\fibf_\hZ \ar{r} & \fibf_\hZ^\affinize \ar{r} & \fibf_X.
 \end{tikzcd}
\end{equation}

\begin{prop}\label{prop:localformalglueing}
 Let $X \in \dSt_k$ and $Z \subset X$ a closed substack. Let $\mathbf{G}$ be an affine $k$-group scheme.
 The above diagram of fiber functors and \cref{prop:algebraization} induce natural equivalences of derived stacks 
 \begin{align*}
  \PerfX_X & \overset\sim\longrightarrow \PerfX_{X \smallsetminus Z} \newtimes_{\PerfX_{\hZ \smallsetminus Z}} \PerfX_\hZ \\
  \CohX_X & \overset\sim\longrightarrow \CohX_{X \smallsetminus Z} \newtimes_{\CohX_{\hZ \smallsetminus Z}} \CohX_\hZ \\
  \BunGX_X & \overset\sim\longrightarrow \BunGX_{X \smallsetminus Z} \newtimes_{\BunGX_{\hZ \smallsetminus Z}} \BunGX_\hZ.
 \end{align*}
\end{prop}

\begin{proof}
This can be checked pointwise. We therefore fix $U \in \dAffX$.
Using \cite[Prop. 7.4.2.1]{Lurie_SAG}, together with \cref{prop:algebraization}, imply that the base change functor
 \[
  \PerfX_X(U) \to \PerfX_{X \smallsetminus Z}(U) \newtimes_{\PerfX_{\hZ \smallsetminus Z}(U)} \PerfX_\hZ(U)
 \]
 is an equivalence. Consider now the morphism in $\dSt_k$
 \[
  f \colon \fibf_{X \smallsetminus Z}(U) \newamalg_{\fibf_{\hZ \smallsetminus Z}(U)} \fibf_{\hZ}(U) \to \fibf_X(U) = U.
 \]
 By the above, $\Perf(f)$ is an equivalence. \ref{sublem:bgperflocal} shows that so is $\Bun^\mathbf{G}(f)$.

 In the case of $\CohX$, it follows from \cite[Remark 7.4.2.2]{Lurie_SAG}, using that $U$ is Noetherian.
\end{proof}

\subsection{The global case}\label{section:global}
We can now focus on the global statement. 

\begin{construction}\label{extension} Recall that for any presentable $\infty$-category $\cC$, any hypercomplete étale sheaf $(\dAffX)\op \to \cC$ extends canonically to a limit-preserving functor $(\dSt_X)\op \to \cC$.
In particular, using \cref{thm:flat_descent} and \cref{prop:BunGdescent}, we can extend 
$\CohX_{\hZ \smallsetminus Z}$, $\PerfX_{\hZ \smallsetminus Z}$ and $\BunGX_{\hZ \smallsetminus Z}$ to every locally Noetherian derived stack over $X$.
Note that descent for coherent/perfect complexes and bundles implies that $\CohX_{T}$, $\PerfX_{T}$ and $\BunGX_{T}$ (for $T$ being either $X$, $X \smallsetminus Z$ or $\hZ$ -- see \cref{lem:descenthZ}) are also hypercomplete sheaves and can similarly be extended to all of $\dSt_X$.
\end{construction}

In the case where $X$ itself is locally Noetherian, we can now take global sections and get the following global version of \cref{prop:localformalglueing}.
\begin{thm}\label{thm:formalglueing}
Let $X \in \dSt_k$ and let $Z \subset X$ be a closed substack. We have three equivalences of $\infty$-categories 
\begin{align*}
  \Perf(X) & \overset\sim\longrightarrow \Perf(X \smallsetminus Z) \newtimes_{\Perf(\hZ \smallsetminus Z)} \Perf(\hZ), \\
  \Coh^-(X) & \overset\sim\longrightarrow \Coh^-(X \smallsetminus Z) \newtimes_{\Coh^-(\hZ \smallsetminus Z)} \Coh^-(\hZ), \\
  \Bun^\mathbf{G}(X) & \overset\sim\longrightarrow \Bun^\mathbf{G}(X \smallsetminus Z) \newtimes_{\Bun^\mathbf{G}(\hZ \smallsetminus Z)} \Bun^\mathbf{G}(\hZ),
\end{align*}
 where $\Perf(\hZ \smallsetminus Z)$, $\Coh(\hZ \smallsetminus Z)$ and $\Bun^\mathbf{G}(\hZ \smallsetminus Z)$ denote respectively $\PerfX_{\hZ \smallsetminus Z}(X)$, $\CohX_{\hZ \smallsetminus Z}(X)$ and $\BunGX_{\hZ \smallsetminus Z}(X)$. The first two equivalences are equivalences of symmetric monoidal $\infty$-categories.
\end{thm}

\subsection{Flag decomposition}\label{sectionflagdecomp}
 In \cref{thm:formalglueing}, the derived stack $X$ can be very general:  locally Noetherian schemes are allowed, as well as formal completions of locally Noetherian schemes along closed subschemes. More generally, if $X \in \dSt_k$ and $Z \subset X$ is a closed substack, then the formal completion $\hZ$ is also locally Noetherian.
 Indeed, this statement is local in $X$ so we can assume $X$ to be affine and Noetherian. In this case, we can see using \cite[Prop 6.7.4]{Gaitsgory-Rozenblyum:dgindschemes} that $\hZ$ is a colimit of Noetherian derived affine schemes and thus do belong to $\dSt_k$.
 
 In particular, both \cref{prop:localformalglueing} and \cref{thm:formalglueing} can be iterated as follows.

\begin{cor}\label{cor:flagdec}  Let $X \in \dSt_k$, $Z_1 \subset Z_2 \subset X$ be a flag of closed substacks, and $\mathbf{G}$ be an affine $k$-group scheme. We have natural equivalences of derived stacks:
 \begin{align*}
 \PerfX_X & \overset\sim\longrightarrow \PerfX_{X \smallsetminus Z_2} \newtimes_{\PerfX_{\hZ_2 \smallsetminus Z_2}}\left( \PerfX_{\hZ_2 \smallsetminus Z_{1}}\newtimes_{\PerfX_{\hZ_{1} \smallsetminus Z_{1}}} \PerfX_{\hZ_1} \right)\\
  \CohX_X & \overset\sim\longrightarrow \CohX_{X \smallsetminus Z_2} \newtimes_{\CohX_{\hZ_2 \smallsetminus Z_2}}\left( \CohX_{\hZ_2 \smallsetminus Z_{1}}\newtimes_{\CohX_{\hZ_{1} \smallsetminus Z_{1}}} \CohX_{\hZ_1} \right)\\
 \BunGX_X & \overset\sim\longrightarrow \BunGX_{X \smallsetminus Z_2} \newtimes_{\BunGX_{\hZ_2 \smallsetminus Z_2}}\left( \BunGX_{\hZ_2 \smallsetminus Z_{1}}\newtimes_{\BunGX_{\hZ_{1} \smallsetminus Z_{1}}} \BunGX_{\hZ_1} \right).
 \end{align*}
 If we suppose, moreover, that $X\in \dSt_k$, taking global sections of the previous equivalences yield equivalences of $\infty$-categories:
 \begin{align*}
\Perf(X) &\simeq \Perf(X \smallsetminus Z_2) \newtimes_{\Perf(\hZ_2 \smallsetminus Z_2)} \left( \Perf(\hZ_2 \smallsetminus Z_{1}) \newtimes_{\Perf(\hZ_{1} \smallsetminus Z_{1})} \Perf (\hZ_1) \right)\\
\Coh^-(X) &\simeq \Coh^-(X \smallsetminus Z_2) \newtimes_{\Coh^-(\hZ_2 \smallsetminus Z_2)} \left( \Coh^-(\hZ_2 \smallsetminus Z_{1}) \newtimes_{\Coh^-(\hZ_{1} \smallsetminus Z_{1})} \Coh^-(\hZ_1) \right) \\
\Bun^\mathbf{G}(X) &\simeq \Bun^\mathbf{G}(X \smallsetminus Z_2) \newtimes_{\Bun^\mathbf{G}(\hZ_2 \smallsetminus Z_2)} \left( \Bun^\mathbf{G}(\hZ_2 \smallsetminus Z_{1}) \newtimes_{\Bun^\mathbf{G}(\hZ_{1} \smallsetminus Z_{1})} \Bun^\mathbf{G}(\hZ_1) \right).
 \end{align*}
The first two equivalences are equivalences of symmetric monoidal $\infty$-categories.
 
 \end{cor}

Similar statements obviously hold  for flags $Z_1 \subset \cdots \subset Z_n \subset X$ of arbitrary length; we leave the interested reader to write them down explicitly.

\begin{rem} \label{rem:abs}
Let $\mathcal{C}$ denote either $\Catstmon$ or $\inftyGpd$, and let $X\in \dSt_k$. Observe that there is an equivalence between hypercomplete sheaves $\mathrm{St}_{\mathrm{fl}}(\dAffX, \mathcal{C})$ for the flat topology on $\dAffX$, and hypercomplete sheaves $\mathrm{St}_{\mathrm{fl}}(X_{\mathrm{fl}}, \mathcal{C})$ for the flat topology on the big Noetherian flat site $X_{\mathrm{fl}} := (\dSt_k/X, \mathrm{fl})$. Now, the map $X \to \Spec k$ induces a morphism of sites $u \colon (\dAff_k, \mathrm{fl}) \to (\dSt_k/X, \mathrm{fl})$. For $Z\subset X$ a closed substack, we will put 
\begin{align*}
\underline{\mathsf{Perf}}_X := u^{-1}\PerfX_X \, , \,\, \underline{\mathsf{Perf}}_{X \smallsetminus Z}:= u^{-1}\PerfX_{X \smallsetminus Z}\, , \,\, \underline{\mathsf{Perf}}_{\hZ \smallsetminus Z}:=u^{-1}\PerfX_{\hZ \smallsetminus Z}\, , \,\, \underline{\mathsf{Perf}}_\hZ:=u^{-1}\PerfX_\hZ \\
\underline{\mathsf{Coh}}^-_X := u^{-1}\CohX_X \, , \,\, \underline{\mathsf{Coh}}^-_{X \smallsetminus Z}:= u^{-1}\CohX_{X \smallsetminus Z}\, , \,\, \underline{\mathsf{Coh}}^-_{\hZ \smallsetminus Z}:=u^{-1}\CohX_{\hZ \smallsetminus Z}\, , \,\, \underline{\mathsf{Coh}}^-_\hZ:=u^{-1}\PerfX_\hZ \\
\underline{\mathsf{Bun}}^\mathbf{G}_X := u^{-1}\BunGX_X \, , \,\, \underline{\mathsf{Bun}}^\mathbf{G}_{X \smallsetminus Z}:= u^{-1}\BunGX_{X \smallsetminus Z}\, , \,\, \underline{\mathsf{Bun}}^\mathbf{G}_{\hZ \smallsetminus Z}:=u^{-1}\BunGX_{\hZ \smallsetminus Z}\, , \,\, \underline{\mathsf{Bun}}^\mathbf{G}_\hZ:=u^{-1}\BunGX_\hZ 
\end{align*}
and similar definitions for longer flags $Z_1 \subset \cdots \subset Z_n \subset X$. These are $\mathcal{C}$-valued hypercomplete sheaves on $\dAff_k$ for the flat topology, and both \emph{formal gluing} (\cref{prop:localformalglueing}) and \emph{flag decomposition} (\cref{cor:flagdec}) hold for them (since $u^{-1}$ preserves fiber products). In particular, for $X\in \dSt_k$ and $Z \subset X$ a closed substack, restriction induces equivalences in $\dSt_k$
\[
  \underline{\mathsf{Perf}}_X \simeq \underline{\mathsf{Perf}}_{X\smallsetminus Z} \newtimes_{\underline{\mathsf{Perf}}_{\hZ\smallsetminus Z}} \underline{\mathsf{Perf}}_{\hZ}
  \hspace{1cm}\text{and}\hspace{1cm}
  \underline{\mathsf{Bun}}^\mathbf{G}_X \simeq \underline{\mathsf{Bun}}^\mathbf{G}_{X\smallsetminus Z} \newtimes_{\underline{\mathsf{Bun}}^\mathbf{G}_{\hZ\smallsetminus Z}} \underline{\mathsf{Bun}}^\mathbf{G}_{\hZ}.
\]

\end{rem}

\subsection{Non-derived version}\label{sec_nonderived}
The results presented so far all involve derived geometry.
They imply similar results in classical geometry, which we will make explicit in this section, for the reader's convenience.\\

\paragraph{\textbf{Fiber functors}}
The notion of fiber functors introduced in \cref{defin:fibf} has an obvious non-derived analog: a fiber functor over a stack $X$ is a functor $\Aff_X \to \St_X$ equipped with a natural transformation to the inclusion $\fibf_X \colon \Aff_X \subset \St_X$.
We denote by $\FibF_X$ the $(2,1)$-category of fiber functors\footnote{Recall that by convention, $\St_X$ denotes the $(2,1)$-category of $1$-stacks over $X$.}.
Computing pointwise fibers defines a fully faithful functor $\tilde \fibf_\bullet \colon \St_X \to \FibF_X$

Classical and derived fiber functors are related by two functors:
\begin{itemize}
\item The inclusion $i \colon \FibF_X \to \dFibF_X^{\leq 1} \subset \dFibF_X$ induced by the inclusion $\St_X \subset \dSt_X$ and a Kan extension.
\item The truncation functor $\tau \colon \dFibF_X^{\leq 1} \to \FibF_X$ (restrict at the source and truncate the values),
\end{itemize}
where $\dFibF_X^{\leq 1}$ denotes the full subcategory of $\dFibF_X$ spanned by fiber functors $\fibf$ such that for any affine (classical) scheme $S$ over $X$, the classical truncation of $\fibf(S)$ is a $1$-stack (i.e. lies in $\St$).
There is a canonical equivalence $\tau \circ i \simeq \id$.
Note that $i$ does not commute with the functors $\tilde \fibf_\bullet \colon \St_X \to \FibF_X$ and $\fibf_\bullet \colon \St_X \subset \dSt_X \to \dFibF_X$ in general.
It does, however, in every case of interest to us: if the structural morphism $Y \to X$ is flat, then $i(\tilde \fibf_Y) = \fibf_Y$.\\

\paragraph{\textbf{Formal glueing}}
These relations between derived and non-derived fiber functors will allow to deduce a non-derived version of the algebraization and formal glueing theorems.
This is only possible due to the following observation.

Fix $X \in \St$. Using the inclusions $\Aff_X \hookrightarrow \dAff_X$ and $\St \hookrightarrow \dSt$, we get canonical functors
\[
 \begin{tikzcd}
  \dFibF_X \ar{r}{\mathrm{rest.}} & \Fun(\Aff_X, \dSt)/\fibf_X & \FibF_X \ar[hook']{l}.
 \end{tikzcd}
\]

\begin{lem}\label{lem:underivedfibf}
 Let $X \in \St$ and $Z \subset X$ a closed substack. The restrictions of the fiber functors
 \[
  \fibf_X, \hspace{1em} \fibf_{\hZ}, \hspace{1em} \fibf^\affinize_{\hZ} \text{ and } \hspace{1em} \fibf_{\hZ \smallsetminus Z}
 \]
 to $\Aff_X$ all belong to the essential image of $\FibF_X$, with pre-image given respectively by
 \[
  S \mapsto S, \hspace{1em} S \mapsto \hZ_S, \hspace{1em} S \mapsto \hZ_S^\affinize \text{ and } \hspace{1em} S \mapsto \hZ_S^\affinize \smallsetminus Z_S,
 \]
 where $Z_S$ is the closed subscheme $Z \times_X S \subset S$.
\end{lem}
\begin{proof}
 The first two cases follow from the flatness of $S$ (resp. $\hZ_S$) over $S$. The other two cases are trivial consequences.
\end{proof}

\begin{cor}
 Both algebraization \cref{prop:algebraization} and formal glueing \cref{prop:localformalglueing,thm:formalglueing} apply mutatis mutandis to the underived setting.
\end{cor}

\section{The \texorpdfstring{$2$}{2}-Segal object of good flags on a surface}\label{SECTIONflags}

In the previous section we worked by fixing a flag (in some fixed scheme or stack). In this Section we will allow the flags to move in flat families (over a fixed scheme), and then move on to consider tuples of families of such flags.\\

Throughout this section, we will fix a smooth projective complex surface $X$. Everything will depend on $X$, so we allow ourselves to sometimes suppress $X$ from the notations.\\

We denote by $\Hilb_X$ the Hilbert scheme functor of (families of) $0$-dimensional subschemes in $X$ (i.e. $\Hilb_X$ is the disjoint union of all the $\Hilb^P_X$ for Hilbert polynomials $P$ of degree $0$). We will write $\Hilb^{\leq 0}_X$ for the functor sending a $\mathbb{C}$-scheme $S$ to $\Hilb_X(S) \cup \{ \emptyset_S \} $, where $\emptyset_S$ denotes the empty subscheme $\emptyset \hookrightarrow X\times S$ (note that $\emptyset$ is flat over $S$).
We denote by $\mathrm{Car}_{X}$ the set valued functor of relative effective Cartier divisors on $X/\mathbb{C}$, sending a $\mathbb{C}$-scheme $S$ to the set of relative effective Cartier divisors on $(X\times S) /S$ (as in \cite[1.2.3]{Katz_Mazur_1985}).\\


\subsection{Flags}

\begin{defin}\label{defin:plaintopflags}  
\begin{itemize}
\item Let $S\in \Aff$. We define $\mathrm{Fl}_X(S)$ as the subset  of $\mathrm{Car}_{X}(S)\times \Hilb^{\leq 0}_X(S)$ consisting of pairs $(D,Z)$ with $Z$ a closed $(X\times S)-$subscheme of $D$ (in which case, we will simply write $Z\subset D$)\footnote{Since both $D\hookrightarrow X\times S$ and $Z\hookrightarrow X\times S$ are closed immersions, if there exists a morphism $Z\to D$ over $X\times S$, then it is unique and it is a closed immersion. This justifies our apparently sloppy notation $Z\subset D$.}. Such pairs will be called \emph{families of flags on $X$ relative to $S$}.
\item The category $\flags_X (S)$ is the category (poset) with set of objects $\mathrm{Fl}_X(S)$ and morphisms sets
\[
\Hom_{\flags_X(S)}((D', Z'), (D, Z)) 
\]
being a singleton if there is a commutative diagram of closed $(X\times S)$-immersions
\begin{equation}\label{mapsss}
\begin{tikzcd}Z'\ar[r] \ar[d] & D' \ar[d] \\ Z \ar[r] & D \end{tikzcd}
\end{equation}
such that $(Z' \to D' \times_D Z)_{\red}$ is an isomorphism; $\Hom_{\flags_X(S)}((D', Z'), (D, Z))$ is otherwise empty.
The composition is the unique one.
\end{itemize}
Since all maps are closed immersions, if a diagram like (\ref{mapsss}) exists, then it is unique, so that $\flags_X (S)$ is indeed a poset.
\end{defin}

\begin{rem} According to Definition \ref{defin:plaintopflags}, the empty flag $(\emptyset, \emptyset) $ is an initial object in $\flags_X(S)$.
However, if $(D,Z) \in \flags_X(S)$, then there is no morphism $(D, \emptyset) \to (D,Z)$ in $\flags_X(S)$ unless $Z=\emptyset$.
On the other hand, for any closed immersion $D'\hookrightarrow D$, $(D', \emptyset) \to (D,\emptyset)$ is in $\flags_X(S)$.
In particular, the category $\mathbf{Car}_X(S)$ of relative effective Cartier divisors on $X\times S/S$ (with maps given by inclusions) fully embeds into $\flags_X(S)$.
A typical example of a morphism in $\flags_X(S)$ is obtained from a chain of closed immersion $Z' \hookrightarrow D' \hookrightarrow D$ : $(D', Z') \to (D,Z')$ is in $\flags_X(S)$. 
\end{rem}

The inverse image functor preserves inclusions of closed subschemes, and for a map of schemes $Y'\to Y$ and a closed subscheme $T \hookrightarrow Y$, the canonical closed immersion 
\[
 \left(Y'\newtimes_Y T\right)_{\!\red} \longrightarrow Y'_{\red} \newtimes_{Y_{\red}} T_{\red}
\]
induces an isomorphism on the underlying reduced subschemes. Therefore, the inverse image functor along $S'\to S$ preserves morphisms between flags as defined in Definition \ref{defin:plaintopflags}, so that the following definition is well-posed.

\begin{defin}\label{defin:plaintopfunctor}  We denote by $\flags_X: \Aff\op \to \mathbf{PoSets}$ the induced functor, and call it the functor of \emph{relative flags on} $X$.
\end{defin}

\begin{rem}\label{rem:cmpBD} Our definitions \ref{defin:plaintopflags} and \ref{defin:plaintopfunctor} are flag analogs of  \cite[3.4.6]{Beilinson_Drinfeld_Chiral_2004} where the authors give an alternative definition of factorization algebras using effective Cartier divisors (on a curve) modulo reduced equivalence, instead of using the Ran space. It was exactly this remark in \cite{Beilinson_Drinfeld_Chiral_2004} that helped us in crystallizing our first ideas on the topics of the current section of this paper.
\end{rem}

\begin{prop} $\flags_X$ is a stack of posets on the Zariski site of $\Aff$.
\end{prop}

Let $\mathrm{forget}: \mathbf{PoSets} \to \mathbf{Sets}$ denote the forgetful functor, sending a poset $(S, \leq)$ to the set $S$.
\begin{lem}\label{Flflatisascheme} The composite functor $$\flags^{\flat}_X: \xymatrix{\left(\Aff\right)\op \ar[r]^-{\flags_X} & \mathbf{PoSets} \ar[rr]^-{\mathrm{forget}} && \mathbf{Sets} }$$ is represented by a scheme.
\end{lem} 
\begin{proof} Since $X$ is smooth projective, $\mathrm{Car}_{X} \times \Hilb^{\leq 0}_{X}$ is represented by a scheme, and the nesting condition $Z\subset D$ for flags is a closed subscheme condition (see also \cite[4.5]{sernesi}).
\end{proof}

\subsection{Good pairs of flags}

In order to be short, we will write that \emph{$D$ is a reCd on $X\times S/S$} to mean that $D$ is a relative effective Cartier divisor on $X \times S/S$.

\begin{defin}\label{defin:goodcartier}
Let $S \in \Aff$, and $D, D' \in \mathrm{Car}_{X}(S)$. We say that the pair of reCd's $(D,D')$ on $X\times S/S$ is \emph{good} if $D\cap D' \in \mathrm{Car}_{D/S}(S) \cap \mathrm{Car}_{D'/S}(S) $ (i.e. $D\cap D'$ is an effective Cartier divisor inside both $D$ and $D'$, and moreover $D\cap D'$ is flat over $S$).
\end{defin}
It is obvious that, for $D, D' \in \mathrm{Car}_{X}(S)$, $(D,D')$ is good if and only if $(D',D)$ is.
 
\begin{lem}\label{lem:gooddivisorsasymm}
 Let $B$ be a commutative ring and let $f, f'$ be non-zero-divisors in $B$. The following assertions are equivalent:
 \begin{enumerate}
  \item\label{ass:asymm1} $f'$ is not a zero-divisor in $B/f$;
  \item\label{ass:asymm2} $f$ is not a zero-divisor in $B/f'$.
 \end{enumerate}
 As a consequence, for any $S \in \Aff$ and $D, D' \in \mathrm{Car}_{X}(S)$, the following are equivalent:
 \begin{enumerate}[label=\rm(\alph*)]
  \item The pair $(D, D')$ is good;
  \item $D \cap D'$ is a relative effective Cartier divisor in $D$;
  \item $D \cap D'$ is a relative effective Cartier divisor in $D'$.
 \end{enumerate}
\end{lem}

\begin{proof}
 By symmetry, it is enough to prove that \eqref{ass:asymm1} implies \eqref{ass:asymm2}.
 Let $x, y \in B$ be such that $xf' = yf$. By assumption, the image of $f'$ in $B/f$ is a non-zero-divisor, so there exists $t \in B$ such that $x = tf$. Since $f$ is not a zero-divisor in $B$, the equality $yf = xf' = tff'$ implies $y = tf'$.
 As this holds for any $x$ and any $y$ as above, we get that $f$ is not a zero-divisor in $B/f'$.
\end{proof}

\begin{defin}\label{defin:goodpair}
Let $S \in \Aff$.
\begin{itemize}
	\item We define $\mathrm{Fl}_{X,2}(S)$ as the subset of $\mathrm{Fl}_X(S)\times \mathrm{Fl}_X(S)$ whose elements are \emph{good} pairs $((D,Z), (D',Z'))$ of families of flags of $X$ relative to $S$, i.e. those pairs satisfying the following conditions:
\begin{enumerate}
\item the pair of reCd's $(D, D')$ is good (Definition \ref{defin:goodcartier}).
\item $D\cap D' = Z\cap Z'$;
\end{enumerate}
\item The category $\flags_{X,2}$ is the full subcategory (which is again a poset) of $\flags_X(S)^{\times 2}$ spanned by $\mathrm{Fl}_{X,2}(S)$.
\end{itemize}
\end{defin}



\begin{rems}
\item \label{rem:emptyflag}
  Note that condition (1) \cref{defin:goodpair} easily implies (see, e.g. \cite[Lemma 0C4R]{stacks-project}) that for a good pair $((D,Z), (D',Z'))$, we have  $D + D'= D\cup D'$, so that, in particular, $D\cup D'$ is flat over $S$.
\item
  Let $S \in \Aff$. If we denote by $(\emptyset, \emptyset)\in \mathrm{Fl}(S)$ the empty flag, then the pair $((D,Z), (\emptyset, \emptyset))$  is a good pair, for any $(D,Z)\in \mathrm{Fl}(S)$ (indeed, note that the empty scheme is an effective Cartier divisor in itself, since $0$ is \emph{not} a zero divisor in the zero ring). Moreover, the pair $((D,Z), (D',Z')) \in \mathrm{Fl}_X(S) \times \mathrm{Fl}_X(S)$ is good if and only if $((D',Z'), (D,Z))$ is good.
\item \label{goodequaldisjointforcurves}
  If $X$ is a smooth projective curve (instead of our surface), and for obvious reasons we  limit ourselves to flags of type $(D, \emptyset)$, then a pair $((D,\emptyset), (D', \emptyset)) \in \flags_{X,2}(S)$ is good if and only if $D$ and $D'$ are disjoint divisors. If $X$ is a smooth surface, then any pair $((D',Z'),(D,Z))$ of flags with $D\cap D' =\emptyset$ is good (but the converse is clearly false).
\end{rems}

Note that all conditions (1)-(2) in \cref{defin:goodpair} are stable under arbitrary base-change: this is obvious for (2), while it follows for instance from \cite[Lemma 063U]{stacks-project} for (1). In other words, for every morphism $S' \to S$ of affine schemes, the induced functor
\[ \flags_X(S) \times \flags_X(S) \to \flags_X(S') \times \flags_X(S') \]
restricts to a functor
\[ \flags_{X,2}(S) \to \flags_{X,2}(S'). \]
So we actually get, using inverse images, a functor
 $$\flags_{X, 2}: \Aff\op \longrightarrow \mathbf{PoSets} \,\, , \,\, S \longmapsto \flags_{X, 2}(S).$$

\begin{lem}\label{lem:unionok} Let $S \in \Aff$.
\begin{enumerate}
\item If $Z$ and $Z'$ are closed subschemes of $X \times S$, both flat over $S$, and $Z\cap Z'$ is flat over $S$, then $Z\cup Z'$ is flat over $S$.
\item Let $Z$ and $Z'$ be closed subschemes of $X \times S$ such that  $Z\cap Z'$ is flat over $S$. Let $f:S' \to S$ an arbitrary morphism, and $F:= \mathrm{id}_X \times f: X \times S' \to X\times S$. Then, the canonical map $$ F^{-1}(Z) \cup F^{-1}(Z') \to F^{-1}(Z \cup Z')$$ is an isomorphism of closed subschemes of $X\times S'$
(or, equivalently, the $\mathcal{O}_{X\times S}$-ideals of $F^{-1}(Z) \cup F^{-1}(Z')$ and of $F^{-1}(Z \cup Z')$ coincide). 
\item Let $D_1, D_2 \in \mathrm{Car}_X(S)$, $f:S' \to S$ is an arbitrary morphism, and $F:= \mathrm{id}_X \times f: X \times S' \to X\times S$. Then, the canonical map  $$ \mathrm{Car}_X (f)(D_1) +\mathrm{Car}_X (f)(D_2)= F^{-1}(D_1)+ F^{-1}(D_2) \longrightarrow F^{-1}(D_1 +D_2)=\mathrm{Car}_X (f)(D_1 + D_2 )$$ is an isomorphism of relative effective Cartier divisors on $X\times S'/S'$ (or equivalently, the $\mathcal{O}_{X\times S}$-ideals of $F^{-1}(D_1)+ F^{-1}(D_2)$ and of $F^{-1}(D_1 +D_2)$ coincide). 
\end{enumerate}
\end{lem}

\begin{proof}  
To prove (1), consider the pushout (of schemes) $$\xymatrix{Z\cap Z' \ar[r] \ar[d] & Z \ar[d] \\ Z' \ar[r] & Z\cup Z'}$$ where we know that $Z, Z'$ and $Z\cap Z'$ are flat over $S$. To conclude that also the pushout $Z\cup Z'$ is flat over $S$, we may suppose  $X\times S$ is affine, so that both $Z$ and $Z'$ are affine, and thus so is $Z\cap Z'$. If we put $X \times S=\Spec A$, $Z\cap Z'=\Spec B$, we are therefore in the following situation: $B, C, D$ are flat $A$-algebras and $\pi_C: C\to B$, $\pi_D: D\to B$ are surjective morphisms of $A$-algebras, and we want to conclude that the fibre product $C\times_B D$ is flat over $A$. This follows from the exact sequence of $A$-modules $0\to C\times_B D \to C\times D \to B \to 0$, since both $C\times D$ and $B$ are then flat over $A$. To prove (2) we adopt the same reduction to the affine case and same notations as in the proof of (1), and let $\varphi: A \to A'$ such that $f=\Spec (\varphi)$. We need to prove that the canonical map of $A'$-algebras $\rho: (C\times_B D)\otimes_A A' \to (C\otimes_A A')\times_{B\otimes_A A'} (D\otimes_A A')$, sending $((c,d)\otimes a')$ to $ (c\otimes a', d\otimes a')$ is an isomorphism. It will be  enough to prove that $\rho$ is an isomorphism of $A'$-modules.  Apply the base change functor $(-)\otimes_A A'$ to  the exact sequence 
\[
 \xymatrix{0\ar[r] & \displaystyle C \newtimes_B D \ar[r] & C\times D \ar[rr]^-{\pi_C-\pi_D} && B \to 0}
\]
of $A$-modules, and observe that the canonical map $(C\times D)\otimes_A A'   \to (C\otimes_A A') \times (D\otimes_A A')$ is an isomorphism (distributivity of tensor product on direct sums in $\mathsf{Mod}_{A'}$). By hypothesis, $B$ is flat over $A$, so that $\mathrm{Tor}^A_1(B, A')=0$, and we get an exact sequence of $A'$-modules
\begin{equation}\label{eqq}
\xymatrix{0 \ar[r] & \displaystyle \left(C\newtimes_B D\right)\newotimes_A A' \ar[r] & \displaystyle \left(C\newotimes_A A'\right)\times \left(D\newotimes_A A'\right)  \ar[rrr]^-{\pi_C\otimes_A A' -\pi_D \otimes_A A'} & && \displaystyle B\newotimes_A A' \ar[r] & 0.
}
\end{equation}
Now, exactness of (\ref{eqq}) is equivalent to $\rho$ being an isomorphism of $A'$-modules.

Finally, (3) is a well-known result that can be proved again by reducing to the affine case: $S'=\Spec R' \to \Spec R =S$, $X\times S= \Spec A$, $X\times S' = \Spec (A\otimes_R R')$, $D_i =\Spec  (A/f_i)$, $i=1,2$. Then $D_1+D_2 = \Spec  (A/f_1f_2)$ and $F^{-1}(D_1 + D_2)= \Spec  (A'/(f_1f_2)')$ where $(-)': A \to A'=A\otimes_R R'$ is the canonical map. Since $(-)'$ is a morphism of algebras, and $F^{-1}(D_1) + F^{-1}(D_2)= \Spec  (A'/(f_1)'(f_2)')$, we conclude.
\end{proof}


\cref{lem:unionok} allows for the following\\

\begin{defin}\label{defin:thisistheunion}
For $S\in \Aff$ and $((D,Z), (D',Z')) \in \flags_{X,2}(S)$, define $$\cup_S((D,Z), (D',Z')):=(D+D', Z\cup Z') \in \flags_X(S).$$
called \emph{union of good pairs}. We will most often write $(D,Z) \cup (D',Z')$ for $\cup_S((D,Z), (D',Z'))$.
\end{defin}

Note that \cref{defin:thisistheunion} is well-posed since $D+D'$ is a relative effective Cartier divisor on $X$ over $S$ (this is classical, see, e.g. \cite[Ch. 1]{Katz_Mazur_1985}), $Z\cup Z'$ lies in $\Hilb_{X}^0(S)$ (since by \cref{lem:unionok} (1), \cref{defin:goodpair} (1) and (2)  ensure that if $((D,Z), (D',Z')) \in \mathrm{Fl}_2(S)$, then $Z \cup Z'$ is again flat over $S$), and obviously $Z\cup Z' \subset D+D'$. 

\begin{rem} The union of good pairs is obviously commutative and unital, with unit given by the empty flag $(\emptyset, \emptyset)$. Moreover, by \cref{rem:emptyflag},  we have
\[
 (D,Z)\cup (D',Z') := (D+D', Z\cup Z')= (D \cup D', Z\cup Z')
\]
for any good pair $((D,Z), (D',Z'))$.
\end{rem}

\begin{lem}\label{unionisfunctorial} The union $\cup: \flags_{X, 2} \to \flags_{X}$ is a morphism of prestacks.
\end{lem}
\begin{proof} For each fixed $S \in \Aff$, we have seem that $\cup_S$ is well defined on objects.  On morphisms, the only point that needs some argument is to show that if $(j_1, j_2): (F_1, F_2) \to (F'_1, F'_2)$ is a morphism in $\flags_{X, 2}(S)$ ($F_i=(D_i, Z_i), F'_i=(D'_i, Z'_i), \, i=1,2$), then the induced map $j_1 \cup j_2: F_1 \cup F_2 \to F'_1 \cup F'_2$ has the property that the canonical map $$Z_1 \cup Z_2 \to (D_1+D_2) \cap (Z'_1 \cup Z'_2)$$ is an isomorphism on the underlying reduced schemes. This follows from the fact that both $Z_1  \to D_1 \cap Z'_1$, and $Z_2 \to D_2 \cap Z'_2$ are isomorphism on the underlying reduced schemes, and from the fact that it is enough to prove the equality $Z_1 \cup Z_2 = (D_1+D_2) \cap (Z'_1 \cup Z'_2)$ at the level of the underlying topological spaces (which is an elementary verification, using that the pairs $(F_1, F_2)$ and $(F'_1, F'_2)$ are good).
Finally, naturality in $S$ is an easy consequence of Lemma \ref{lem:unionok} (2) and (3).
\end{proof}

\subsection{Good tuples of flags}
\begin{prop}\label{prop:goodtupledivisors}
Fix $S = \Spec(A)$ an affine test scheme and $(D_1, \dots, D_n) \in \mathrm{Car}_{X}(S)^n$. The following conditions are equivalent:
\begin{enumerate}
 \item\label{ass:2by2} For any $i \neq j$, the pair of divisors $(D_i, D_j)$ is good (in the sense of \cref{defin:goodcartier});
 \item\label{ass:incremental} For any $i$, the pair $(D_1 + \cdots + D_i ,D_{i+1})$ is good;
 \item\label{ass:any} For any disjoint subsets $I$ and $J$ of $\{1, \dots, n\}$, the pair $(\sum_{i\in I} D_i ,\sum_{j \in J} D_j)$ is good.
\end{enumerate}
We shall say that a family of such divisors is good if it satisfies the above conditions.
\end{prop}
\begin{proof}
 Let us first prove \eqref{ass:2by2} and \eqref{ass:incremental} are equivalent.
 This straightforwardly reduces to the case of three divisors $(D_1, D_2, D_3)$.
 
 Assume first that they are pairwise good and let us prove that $D_1 + D_2$ and $D_3$ are good.
 The statement being local, we choose an affine chart $\Spec B \subset X \times S$ and pick non-zero-divisors $f_i \in B$ cutting out $D_i$.
 As $D_i$ and $D_j$ are good (for $i \neq j$), the image of $f_i$ in $B/f_j$ is a non-zero-divisor as well.
 
 Let $g,h \in B$ such that $f_3 g = f_1f_2h$.
 Since $f_3$ is a non-zero-divisor in $B/f_1$, there exists $g' \in B$ such that $g = f_1 g'$.
 As $f_1$ is a non-zero-divisor in $A$, we get $f_3 g' = f_2 h$.
 Since $f_3$ is a non-zero-divisor in $B/f_2$, we get $g''$ such that $g' = f_2 g''$. This in turn implies $f_3 g'' = h$.
 This shows that $f_1f_2$ is a non-zero-divisor in $B/f_3$.

 In order to prove that $(D_1 + D_2) \cap D_3$ is a relative effective Cartier divisor in $D_3$, it remains to show it is flat over $A$.
 Consider then the following sequence
 \begin{equation}\label{exseq:flatness}
  \begin{tikzcd}
   0 \ar{r} & B/(f_1, f_3) \ar{r}{f_2\times-}[swap]{\alpha} & B/(f_1f_2, f_3) \ar{r}[swap]{\pi} & B/(f_2, f_3) \ar{r} & 0.
  \end{tikzcd}
 \end{equation}
 The morphism $\pi$ is obviously surjective, and its kernel coincides with the image of $\alpha$.
 To prove that $\alpha$ is injective, consider $b \in B$ such that $f_2b = 0 \in B/(f_1f_2, f_3)$.
 Pick $x, y \in B$ such that $f_2b = f_1f_2x + f_3y$. Since $f_3$ is a non-zero-divisor in $B/f_2$, we find $y'$ satisfying $f_2y' = y$.
 As $f_2$ is a non-zero-divisor in $B$, we get $b - f_1x = f_3 y'$.
 In particular $b = 0 \in B/(f_1, f_3)$ and $\ker(\alpha) = 0$.
 The above sequence is therefore exact. Since both $B/(f_1, f_3)$ and $B/(f_2, f_3)$ are flat over $A$, so is $B/(f_1f_2, f_3)$.
 This shows that $(D_1 + D_2) \cap D_3$ is a relative effective Cartier divisor in $D_3$.
 Using \cref{lem:gooddivisorsasymm}, we deduce that \eqref{ass:2by2} implies \eqref{ass:incremental}.
 
 \medskip
 We now assume that $(D_1, D_2)$ and $(D_1 + D_2, D_3)$ are good pairs of divisors.
 We fix as above non-zero-divisors $f_i \in B$ generating the ideal corresponding to $D_i$.
 Let $g, h \in B$ such that $gf_1 = hf_3$.
 Multiplying by $f_2$ and using that $f_1f_2$ is not a zero-divisor in $B/f_3$, we get $t \in B$ such that $g = tf_3$.
 In particular, $f_1$ is a non-zero-divisor in $B/f_3$.
 By \cref{lem:gooddivisorsasymm}, $f_3$ is not a zero-divisor in $B/f_1$ either.
 The roles of $f_1$ and $f_2$ being symmetric, the same holds for $f_2$ in place of $f_1$.
 
 The sequence \eqref{exseq:flatness} above is still exact. Since $B/(f_1f_2, f_3)$ is flat over $A$, the associated long Tor exact sequence yields
 \[
  \forall n \geq 1~ \Tor_n^A(B/(f_1, f_3), -) \simeq \Tor_{n+1}^A(B/(f_2, f_3), -).
 \]
 Again, the role of $f_1$ and $f_2$ being symmetric, we also get $\Tor_n^A(B/(f_2, f_3), -) \simeq \Tor_{n+1}^A(B/(f_1, f_3), -)$.
 As a consequence, the Tor-functors $\Tor_n^A(B/(f_1, f_3), -)$ are $2$-periodic:
 \[
  \forall n \geq 1~\Tor_n^A(B/(f_1, f_3), -) \simeq \Tor_{n+2}^A(B/(f_1, f_3), -).
 \]
 However, the exact sequence $0 \to B/f_1 \overset{f_3}{\to} B/f_1 \to B/(f_1, f_3) \to 0$ and the assumption that $B/f_1$ is flat over $A$ imply that $B/(f_1, f_3)$ has Tor-dimension at most $1$.
 Together with the above $2$-periodicity, we deduce that $B/(f_1, f_3)$ is flat over $A$.
 We proved that the pair $(D_1, D_3)$ is good. By symmetric, so is $(D_2, D_3)$ and thus \eqref{ass:incremental} implies \eqref{ass:2by2}.
 
 Clearly, \eqref{ass:any} implies \eqref{ass:incremental}. The converse implication is proven by induction, up to renumbering the divisors.
\end{proof}

An analog of \cref{prop:goodtupledivisors} for flags does not quite hold: the right notion of good families of flags cannot be defined in terms of pairs of elements.
\begin{defin}\label{def:goodflags}
Fix $S$ an affine test scheme. We define the notion of good families of flags inductively on the cardinal:
\begin{itemize}
 \item A pair $(F_1, F_2) \in \flags_X(S)^2$ is good if it is in the sense of \cref{defin:goodpair}.
 \item A family $(F_1, \dots, F_n) \in \flags_X(S)^n$ is good if for any $i < j$, the pair $(F_i, F_j)$ is good, and the family $(F_i \cup F_j, F_1, \dots, F_{i-1}, F_{i+1}, \dots, F_{j-1}, F_{j+1}, \dots, F_n)$ is good.
\end{itemize}
\end{defin}

\begin{prop}\label{prop:goodtupleflags}
Fix $S$ an affine test scheme and $(F_1, \dots, F_n) \in \flags_X(S)^n$. The following conditions are equivalent:
\begin{enumerate}[label=\rm(\arabic*), ref=\rm(\arabic*)]
 \item\label{ass:goodflags} The family $(F_1, \dots, F_n)$ is good.
 \item\label{ass:consecutive} The families $(F_2, \dots, F_n)$, $(F_1, \dots, F_{n-1})$ and $(F_1, \dots, F_i \cup F_{i+1}, \dots, F_n)$ for $0 < i < n$, are all good;
 \item\label{ass:incrementalflags} For any $i$, the pair $(F_1 \cup \cdots \cup F_i, F_{i+1})$ is good;
 \item\label{ass:delndeli} The family $(F_1, \dots, F_{n-1})$ is good and there exists $1 \leq i \leq n-2$ such that the family $(F_1, \dots, F_i \cup F_{i+1}, \dots, F_n)$ is good;
 \item\label{ass:del0deli}The family $(F_2, \dots, F_n)$ is good and there exists $2 \leq i \leq n-1$ such that the family $(F_1, \dots, F_i \cup F_{i+1}, \dots, F_n)$ is good.
\end{enumerate}
\end{prop}

\begin{rem}
 The definition of good families of flags is independent of the order of the flags. As a consequence, there also are characterizations similar to \ref{ass:consecutive} through \ref{ass:del0deli} but where the role of the flags are permuted.
 
 In particular, a family $(F_1, \dots, F_n)$ is good if and only if there exists a permutation $\sigma$ such that each pair $(F_{\sigma(1)} \cup \dots \cup F_{\sigma(i)}, F_{\sigma(i+1)})$ is good. In this case, this actually holds for any permutation $\sigma$.
 
 We can interpret this fact as: ``A family of flags is good if there exists a way of computing its union by adding flags one by one. In this case, every such way of computing the union is valid.''
\end{rem}

\begin{lem}\label{lem:triplegoodflags}
 Let $S$ be an affine scheme and $F_1$, $F_2$ and $F_3$ be three flags in $\flags_X(S)$. The following statements are equivalent:
 \begin{enumerate}[label=\rm(\roman*)]
  \item The pairs $(F_1, F_2)$ and $(F_1 \cup F_2, F_3)$ are good;
  \item The pairs $(F_1, F_3)$ and $(F_1 \cup F_3, F_2)$ are good;
  \item The pairs $(F_2, F_3)$ and $(F_2 \cup F_3, F_1)$ are good.
 \end{enumerate}
\end{lem}

\begin{proof}
 By symmetry, it suffices to show the first assertion implies the second one. Assume thus that $(F_1, F_2)$ and $(F_1 \cup F_2, F_3)$ are good pairs of flags.
 For each $i$, we denote by $(D_i, Z_i)$ the flag $F_i$.
 By \cref{prop:goodtupledivisors}, the pairs of divisors $(D_1, D_3)$ and $(D_1 + D_3, D_2)$ are good.
 Remark that this implies $D_1 + D_3 = D_1 \cup D_3$.
 It is therefore enough to verify $D_1 \cap D_3 = Z_1 \cap Z_3$ and $(D_1 \cup D_3) \cap D_2 = (Z_1 \cup Z_3) \cap Z_2$.
 This can be checked locally in some Zariski chart $\Spec B$ of $X \times S$. We denote by $f_k$ a non-zero-divisor of $B$ of whom $D_ik$ is the zero-locus, by $I_k$ the ideal $(f_k)$ and by $J_k$ the ideal of functions vanishing on $Z_k$ (so that by definition $I_k \subset J_k$). , and by assumption $I_1 + I_2 = J_1 + J_2$ and $I_3 + (I_1 \cap I_2) = J_3 + (J_1 \cap J_2)$.
 We have to prove:
\[
\begin{cases}
  I_1 +  I_3 = J_1 + J_3 & \mathrm{(a)}  \\
  I_2 + (I_1 \cap I_3)=J_2 + (J_1 \cap J_3) & \mathrm{(b)}
\end{cases}
\]
This is elementary. Here are the details:
\begin{enumerate}[label=\rm(\alph*)]
\item Let $ \xi_1 + \xi_3 \in J_1 + J_3$, $\xi_i \in J_i$. Since $J_1 \subset I_1 + I_2$, we get $\xi_1= x_1 + x_2 \in I_1 + I_2$, so that $x_2 =\xi_1 - x_1 \in I_2 \cap J_1 \subset J_2 \cap J_1$. Therefore $x_2+ \xi_3= x_3 + y_{12} \in I_3 + (I_1 \cap I_2)$ for some $x_3 \in I_3$ and $y_{12} \in I_1 \cap I_2$. This gives $\xi_1 + \xi_3= (x_1 + y_{12}) + x_3 \in I_1 + I_3$.\\
\item Let $\xi_2 + \xi_{13} \in J_2 + (J_1 \cap J_3)$, with $\xi_2 \in J_2$ and $\xi_{13} \in J_1 \cap J_3$. We have $\xi_{13} \in J_1 \subset I_1 + I_2$ so there exist $y_1 \in I_1$ and $y_2 \ni I_2$ such that $\xi_{13} = y_1 + y_2$.
We get $y_2 = \xi_{13} - y_1 \in I_2 \cap J_1 \subset J_2 \cap J_1$. 
On the other hand, $y_1 = \xi_{13} - y_2 \in J_3 + (J_2 \cap J_1) = I_3 + (I_1 \cap I_2)$, so that we have $y_1 = t_3 + t_{12}$ for some $t_3 \in I_3$ and $t_{13}$ in $(I_1 \cap I_2)$.
We find $t_3 \in I_1 \cap I_3$, since $y_1 \in I_1$. Therefore, $\xi_{13}= (t_3 + t_{12}) + y_2 \in I_2 + (I_1 \cap I_3)$.
Hence, we are left to prove that $\xi_2 \in I_2 + (I_1 \cap I_3)$ as well. Decomposing $\xi_2 = x_1 + x_2 \in I_1 + I_2$, we find $x_1 \in I_1 \cap J_2 \subset J_1 \cap J_2 \subset I_3 + (I_1 \cap I_2)$.
Let $\alpha_3 \in I_3$ and $\alpha_{12} \in I_1 \cap I_2$ such that $x_1 = \alpha_3 + \alpha_{12}$.
This implies $\alpha_3 \in I_1\cap I_3$ and in turn $\xi_2 = x_2 + x_1 = x_2 + \alpha_3 + \alpha_{12} \in I_2 + (I_1 \cap I_3)$.
\end{enumerate} 
\end{proof}

\begin{proof}[Proof of \cref{prop:goodtupleflags}]
Let us first observe that any subfamily of a good family is good. In particular \ref{ass:goodflags} implies \ref{ass:consecutive}.
Moreover, a rapid induction on $i$ shows that \ref{ass:consecutive} implies \ref{ass:incrementalflags}.

Let us now assume \ref{ass:incrementalflags}.
We will prove \ref{ass:incrementalflags} $\implies$ \ref{ass:goodflags} by induction on $n$.
For $n = 2$, the implication is tautological. 
Assume now that it holds for $n \geq 2$ and pick a family $(F_1, \dots, F_{n+1}) \in \flags_X(S)^{n+1}$ satisfying \ref{ass:incrementalflags}.
For any $1 \leq i < j \leq n+1$, we denote by $E_{ij}$ the $n$-tuple
\[
 E_{ij} = (F_i \cup F_j, F_k, k \in \{1, \dots, n+1\} \smallsetminus \{i,j\}).
\]
We have to show that each pair $(F_i, F_j)$ and each $E_{ij}$ is good.

The truncated family $(F_1, \dots, F_n)$ satisfies \ref{ass:incrementalflags} and thus, by induction, is good.
In particular, for any $1 \leq i < j \leq n$, the pair $(F_i, F_j)$ and the family $(F_i \cup F_j, F_k, k \in \{1, \dots, n\} \smallsetminus \{i,j\})$ are good.
With the original assumption that $(F_1 \cup \cdots \cup F_n, F_{n+1})$ is good, it follows that the $n$-tuple $E_{ij}$ satisfies \ref{ass:incrementalflags} and is thus good.

It remains to handle the pairs $(F_i, F_{n+1})$ and the families $E_{i,n+1}$.
The $n$-tuple $(F_1 \cup F_2, F_3, \dots, F_{n+1})$ satisfies \ref{ass:incrementalflags} and is thus good.
In particular, the pair $(F_i, F_{n+1})$ is good for $i \geq 3$.
The pairs $(F_1, F_2)$ and $(F_1 \cup F_2, F_{n+1})$ are also good, and thus so are the pairs $(F_1, F_{n+1})$ and $(F_2, F_{n+1})$ by \cref{lem:triplegoodflags}.

By induction, to prove that $E_{i,n+1}$ is a good $n$-tuple, it suffices to show it satisfies \ref{ass:incrementalflags}: i.e. that each pair $P_k^i := (F_1^i \cup \cdots \cup F_k^i, F_{k+1}^i)$ is good, where $F_m^i = F_m$ if $m \neq i$ and $F_i^i = F_i \cup F_{n+1}$.

Assume for now $i \geq 2$. By assumption, the pair $P_k^i$ is good if $k < i - 1$.
The $(n+2-i)$-tuple $(F_1 \cup \cdots \cup F_i, F_{i+1}, \dots, F_{n+1})$ satisfies \ref{ass:incrementalflags} and is thus good. In particular, the pair $(F_1 \cup \cdots \cup F_i, F_{n+1})$ is good.
Since $(F_1 \cup \dots \cup F_{i-1}, F_i)$ is a good pair, \cref{lem:triplegoodflags} implies that $P_{i-1}^i = (F_1 \cup \dots \cup F_{i-1}, F_i \cup F_{n+1})$ is a good pair as well.
Moreover $P_k^i = P_k^1$ as soon as $k \geq i$, so the case $i = 1$ induces the others.

We are left to prove that for any $1 \leq k < n$, the pair
$P_k^1 = (F_{n+1} \cup F_1 \cup \cdots \cup F_k, F_{k+1})$
is good. This also follows from \cref{lem:triplegoodflags}, since both $(F_1 \cup \dots \cup F_k, F_{k+1})$ and $(F_1 \cup \dots \cup F_{k+1}, F_{n+1})$ are good pairs of flags.
The family $(F_1, \dots, F_{n+1})$ is therefore good.
This concludes the proof of the implication \ref{ass:incrementalflags} $\implies$ \ref{ass:goodflags} by induction on $n$.

\medskip
It remains to prove \ref{ass:delndeli} $\iff$ \ref{ass:consecutive} $\iff$ \ref{ass:del0deli}.
By symmetry, dealing with the first equivalence is enough.
Moreover, we tautologically have \ref{ass:consecutive} $\implies$ \ref{ass:delndeli}.
Conversely, assume \ref{ass:delndeli} holds for a family $(F_1, \dots, F_n)$.
Since $(F_1, \dots, F_{n-1})$ is good, it satisfies \ref{ass:incrementalflags} and each pair $(F_1 \cup \dots \cup F_{p-1}, F_p)$ is good for $p < n$. 

Moreover, there exists $i \leq n-1$ such that $(F_1, \dots, F_i \cup F_{i+1}, \dots, F_n)$ is good.
This family thus also satisfies \ref{ass:incrementalflags} and thus the pair $(F_1 \cup \dots \cup F_{n-1}, F_n)$ is good as well.
In particular, the family $(F_1, \dots, F_n)$ satisfies \ref{ass:incrementalflags} and is thus good.
\end{proof}

\subsection{The \texorpdfstring{$2$}{2}-Segal object of good flags}\label{section:2SegalFlags}

The union operation of \cref{unionisfunctorial} is actually a partially defined commutative monoid structure on $\flags_{X}$: the operation is only defined for some pairs of elements. This sort of structure requires extra care when defining (and checking) the various constraints (here, associativity and commutativity).
For instance, we need to ensure that every operation is actually well defined when writing down the associativity constraint.

A concise way of writing those details consists in using $2$-Segal objects. We will thus define a $2$-Segal object $\flags_{X, \bullet}$ in prestacks that we will later use to construct a factorization structure on our flag version of Beilinson-Drinfeld's affine Grassmannian.

Let us start by fixing a test scheme $S\in \Aff$. We will  construct a simplicial object in $\mathsf{PoSet}$, which we denote by $\flags_{X,\bullet}(S)$.

\begin{defin}
 For any $n \geq 0$, we define $\flags_{X,n}(S)$ as the full sub-poset of $\flags_X(S)^n$ of good families of flags (in the sense of \cref{def:goodflags}).
 By convention, we set $\flags_{X,0}(S): = *$, $\flags_{X,1} (S):=\flags_X(S)$.
\end{defin}

By \cref{prop:goodtupleflags}, for $0\leq i \leq n$ we have (order preserving) \emph{face} maps $\partial_i^{n} : \flags_{X, n} (S) \to \flags_{X, n-1} (S)$ whose values on objects is defined as follows
\[ \partial_i^n\colon (F_1,\dots, F_n) \longmapsto \begin{cases}
(F_2, \dots , F_n) & \text{ if } i=0 \\
(F_1 , \dots , F_i \cup F_{i+1}, \dots F_n) & \text{ if } 0<i<n \\
(F_1 ,\dots , F_{n-1}) & \text{ if } i=n 
\end{cases}
\] 
Similarly, we have \emph{degeneracy} maps $\sigma_i^{n} : \flags_{X, n} (S) \to \flags_{X, n+1} (S)$ defined as
\[ \sigma_i^n\colon (F_1,F_2,\dots, F_n) \longmapsto (F_1,\dots, F_{i}, \emptyset, F_{i+1}, \dots F_n).
\] 
It is immediate to verify that these give $\flags_{X, \bullet} (S)$ the structure of a simplicial object in $\mathsf{PoSet}$.

These constructions are functorial in $S$: every $\flags_{X, n} (S)$ is a functor $\Aff\op \to \mathsf{PoSet}$, as they are subfunctors of $\flags_{X} (S)^{\times n}$. Moreover, functoriality of the degeneracy maps $\sigma_i^n$ with respect to maps in $\Aff\op$ is obvious, while the fact that also the face maps $\partial_{i}^n$ are functorial is a consequence of Lemma \ref{unionisfunctorial}. Therefore, $S \mapsto\flags_{X, \bullet}(S)$ defines a functor
\[  \flags_{X,\bullet} \colon \Aff\op \to s\mathsf{PoSet} \] 
or, equivalently, a simplicial object in $\Fun(\Aff\op, \mathsf{PoSet})$, which we still denote $\flags_{X, \bullet}$.

We will use the notion of $2$-Segal simplicial object from \cite{KapDyc}.

\begin{thm}\label{thm:2SegGood} 
The simplicial object $\flags_{X,\bullet}$ is a $2$-Segal object.
\end{thm}

\begin{proof}
First of all, recall (see \cite{KapDyc}) that we have to prove that  $\flags_{X,\bullet}$ satisfies the \emph{$2$-Segal conditions}, i.e. that for any $(i,n)$ with $0<i<n$, the following squares are Cartesian
\[
\begin{tikzcd}
  \flags_{X, n+1} \ar{r}{\partial_{i+1} } \ar{d}[swap]{\partial_0} \namecell{dr}{(\sigma)} & \flags_{X, n} \ar{d}{\partial_0} 
  \\
  \flags_{X, n} \ar{r}[swap]{\partial_i} & \flags_{X, n-1}
\end{tikzcd}
\hspace{3em}
\begin{tikzcd}
  \flags_{X, n+1} \ar{r}{\partial_{i}} \ar{d}[swap]{\partial_{n+1}} \namecell{dr}{(\tau)} & \flags_{X, n} \ar{d}{\partial_n} 
  \\
  \flags_{X, n} \ar{r}[swap]{\partial_i} & \flags_{X, n-1}.
\end{tikzcd}
\]
We focus on the square $(\sigma)$, the case $(\tau)$ being deduced by symmetry.
We will show that the natural map of posets
\[
\flags_{X, n+1} (S) \to H := \flags_{X, n} (S) \newtimes_{\partial_0,\, \flags_{X, n-1} (S), \, \partial_i} \flags_{X, n} (S)
\]
is an isomorphism.
To begin with, observe that the composition
\[
 \begin{tikzcd}
 \flags_{X, n+1} \ar{r}{\partial_0,\partial_{i+1}} & H \subset \flags_{X, n} \times \flags_{X, n} \subset \flags_X^n \times \flags_X^n \ar{r}{\id, p_1} & \flags_X^n \times \flags_X = \flags_X^{n+1}
 \end{tikzcd}
\]
is the canonical inclusion. Moreover, the restriction of $(\id, p_1) \colon \flags_X^n \times \flags_X^n \to \flags_X^{n+1}$ to $H$ is injective (as a map of posets).
It is then enough to prove that $\flags_{X, n+1}$ and $H$ have the same points.
This is exactly the content of the equivalence \ref{ass:goodflags} $\iff$ \ref{ass:del0deli} from \cref{prop:goodtupleflags}.
\end{proof}

\begin{rem}\label{Fliscomm} Since the union of flags is obviously strictly commutative, one can easily verify that the associative algebra in correspondences induced (via \cite{stern}) by the 2-Segal object $\flags_{X,\bullet}$ is, in fact, commutative.
\end{rem}

\section{The flag Grassmannian of a surface}\label{flagGrassfirst appearance}
\subsection{Setup and notations}\label{setupgen}
\newcommand{\Zprimehataff}{\hZ^{\prime\affinize}}
\newcommand{\Zhataff}{\hZ^{\affinize}}
\newcommand{\Thataff}{\widehat T^{\affinize}}
We will fix once and for all a smooth reductive affine group scheme over $\mathbb{C}$, and therefore we will simply write $\mathbf{Bun}$ and $\underline{\mathsf{Bun}}$ for $\mathbf{Bun}^{\mathbf{G}}$ or $\underline{\mathsf{Bun}}^{\mathbf{G}}$. respectively. We let $\inftyGpd$ denote the $\infty$-category of  $\infty$-groupoids (i.e. spaces or, equivalently, simplicial sets), and its objects will be often simply called \emph{groupoids} (rather than $\infty$-groupoids).\\

For a fiber functor $\fibf$ over a locally Noetherian (derived) scheme $Y$ (Definition \ref{defin:fibf}), and a (derived) $Y$-scheme $V \to Y$, we write
\begin{equation}\label{holimsections}
\mathbf{Bun}_{\fibf}(V)= \mathrm{Bun}(\fibf(V)):= \holim_{\substack{\Spec R \to V \\ R \ \mathrm{Noetherian}}} \mathrm{Bun}(\fibf(R)).
\end{equation}
In particular, we call $\mathrm{Bun}(\fibf(Y))$  the global sections of the stack $\mathbf{Bun}_{\fibf}$ and, in order to simplify our notations, we will often write simply  $\mathrm{Bun}(\fibf)$ instead of $\mathrm{Bun}(\fibf(Y))$ for these global sections.

We identify, even notationally, (Lemma \ref{lem:fibf-ff}) a derived scheme $V \to Y$ over $Y$ with its fiber functor $\fibf_V$ sending $\Spec R \to Y$ to $\Spec R \times_Y V$.

\begin{defin}\label{def_capcupfibf}
 Let $\fibf_1$ and $\fibf_2$ be fiber functors over $Y$.
 We denote by $\fibf_1 \cap \fibf_2$ (and call it intersection) the fiber functor
 \[
  \fibf_1 \cap \fibf_2 := \fibf_1 \newtimes_{\fibf_Y} \fibf_2 \colon S \mapsto \fibf_1(S) \newtimes_S \fibf_2(S).
 \]
 We denote by $\fibf_1 \cup \fibf_2$ (and call it the union) the fiber functor
 \[
  \fibf_1 \cup \fibf_2 := \fibf_1 \newamalg_{\fibf_1 \cap \fibf_2} \fibf_2 \colon S \mapsto \fibf_1(S) \newamalg_{\fibf_1(S) \cap \fibf_2(S)} \fibf_2(S).
 \]
 These notations are abusive, as we do not require the structural morphisms $\fibf_i(S) \to S$ to be immersions. We will thus restrict their use to situations where said structural morphisms are wannabe immersions (even though they might not formally be).
\end{defin}
 
Let $Y$ be a locally Noetherian derived scheme and $Z\subset T$ be closed (derived) subschemes in $Y$.
Recall from sections 2 and 3, the derived stacks in $\dSt_Y$
\begin{align*}
&& \dAffover{Y} &\longrightarrow \inftyGpd\\
&\mathbf{Bun}_Y : & (U \to Y) &\longmapsto \mathrm{Bun}(U),\\
&\mathbf{Bun}_{Y \smallsetminus T} : & (U \to Y) &\longmapsto \mathrm{Bun}(U \smallsetminus T_U) ,\\
&\mathbf{Bun}_{\widehat{T}} : & (U \to Y) &\longmapsto \mathrm{Bun}(\widehat T_U),\\
&\mathbf{Bun}_{\Thataff} : & (U \to Y) &\longmapsto \mathrm{Bun}(\fibf^{\affinize}_{\widehat{T}}(U))= \mathrm{Bun}(\widehat T_U^\affinize) ,\\
&\mathbf{Bun}_{\widehat{T}\smallsetminus T} : & (U \to Y) &\longmapsto \mathrm{Bun}(\fibf_{\widehat{T}\smallsetminus T}(U))= \mathrm{Bun}( \widehat T_U^\affinize \smallsetminus T_U ),\\
&\mathbf{Bun}_{\hZ} : & (U \to Y) &\longmapsto \mathrm{Bun}(\hZ_U),\\
&\mathbf{Bun}_{\Zhataff} : & (U \to Y) &\longmapsto \mathrm{Bun}(\fibf^{\affinize}_{\hZ}(U))= \mathrm{Bun}(\hZ_U^\affinize ),\\
&\mathbf{Bun}_{(Y\smallsetminus T) \cap \Zhataff } : & (U \to Y) &\longmapsto \mathrm{Bun}((U\smallsetminus T_U) \times_U \hZ_U^\affinize),\\
&\mathbf{Bun}_{(\widehat{T}\smallsetminus T) \cap \Zhataff } : & (U \to Y) &\longmapsto \mathrm{Bun}((\widehat T_U^\affinize \smallsetminus T_U) \times_U \hZ_U^\affinize),
\end{align*}
where $T_U := T \times_Y U$, $Z_U := Z \times_Y U$.

We have obvious restriction maps in $\dSt_Y$
\begin{equation}\label{resoverY}
\begin{gathered}
\begin{aligned}
\mathbf{Bun}_{Y} &\longrightarrow \mathbf{Bun}_{Y \smallsetminus T} &
\hspace{4em} 
\mathbf{Bun}_{Y} &\longrightarrow \mathbf{Bun}_{\Zhataff } \\
\mathbf{Bun}_{Y\smallsetminus T} &\longrightarrow \mathbf{Bun}_{(Y\smallsetminus T) \cap \Zhataff } &
\mathbf{Bun}_{\Zhataff} &\longrightarrow \mathbf{Bun}_{(Y\smallsetminus T) \cap \Zhataff } \\ \mathbf{Bun}_{\Zhataff} &\longrightarrow \mathbf{Bun}_{(\widehat{T}\smallsetminus T) \cap \Zhataff } &
\mathbf{Bun}_{\Thataff} &\longrightarrow \mathbf{Bun}_{\widehat{T}\smallsetminus T}
\end{aligned}
\\
\mathbf{Bun}_{\widehat{T}\smallsetminus T} \longrightarrow \mathbf{Bun}_{(\widehat{T}\smallsetminus T) \cap \Zhataff }.
\end{gathered}
\end{equation}
Moreover, algebraization (\cref{prop:algebraization}) gives equivalences
\begin{align*}
 \mathbf{Bun}_{\Zhataff } &\stackrel{\sim}{\longrightarrow} \mathbf{Bun}_{\hZ} &
 \mathbf{Bun}_{\Thataff } &\stackrel{\sim}{\longrightarrow} \mathbf{Bun}_{\widehat T}.
\end{align*}

\subsection{The flag Grassmannian of a surface}\label{genflaggrass}
%

In the next definition we use the notations established in \S \, \ref{setupgen}.
\begin{defin}\label{def:flagGrassgeneral} 
Let $Y$ be a locally Noetherian scheme and $Z\subset T $ be closed subschemes in $Y$.
\begin{itemize}
\item We will call \emph{Grassmannian of $Y$ relative to} $(T, Z)$ both  
\[
 \GrX(Y; T, Z) := \mathbf{Bun}_Y \newtimes_{\mathbf{Bun}_{Y\smallsetminus T \cup \Zhataff}} \{\mathrm{trivial}\} := 
 \doubletimes{\mathbf{Bun}_Y}{\{\mathrm{trivial}\}}{\mathbf{Bun}_{Y\smallsetminus T}}{\mathbf{Bun}_{\Zhataff}}{\mathbf{Bun}_{(Y\smallsetminus T) \cap \Zhataff}}
 \hspace{1em} \in \dSt_Y
\]
and
\[
 \Gr(Y; T, Z) :=\underline{\mathsf{Bun}}_Y \newtimes_{\underline{\mathsf{Bun}}_{Y\smallsetminus T \cup \Zhataff}} \{\mathrm{trivial}\} := 
 \doubletimes{\underline{\mathsf{Bun}}_Y}{\{\mathrm{trivial}\}}{\underline{\mathsf{Bun}}_{Y\smallsetminus T}}{\underline{\mathsf{Bun}}_{\Zhataff}}{\underline{\mathsf{Bun}}_{(Y\smallsetminus T) \cap \Zhataff}}
 \hspace{1em} \in \dSt_k,
\]
\item we put $\mathrm{Gr}(Y;T,Z):= \Gamma(Y , \GrX(Y; T, Z)) = \Gamma(\Spec \mathbb{C} , \Gr (Y; T, Z))   \in \inftyGpd.$
\end{itemize}
\end{defin}

\begin{rem}\label{comparisonbold-underlined} Since the functor $u: \dSt_Y \to \dSt_k$ is right adjoint, it commutes with homotopy limits, so we have $u(\GrX(Y; T, Z))\simeq \Gr(Y; T, Z)$.
\end{rem}

\begin{rem}\label{mostgeneralGr} It is useful to notice that the definition of $\GrX(Y; T, Z)$, and hence of $\Gr(Y; T, Z)$ does not require that $Z$ is a closed subscheme of $T$, but just that $T$ and $Z$ are, possibly unrelated, closed subschemes of $Y$.
\end{rem}

\begin{rem}\label{descriptionofGrassgenbold} Let $U:=\Spec R \to Y$, with $R$ Noetherian (and not derived), and consider the following pullback diagrams
\[
 \xymatrix{\makebox[0pt][r]{$\Spec(R/I) ={}$} T_U \ar[d] \ar[r] & U \ar[d]\\ T \ar[r] & Y}
 \hspace{10em}
 \xymatrix{\makebox[0pt][r]{$\Spec(R/J) ={}$} Z_U \ar[d] \ar[r] & U \ar[d]\\ Z \ar[r] & Y.}
\]
Denoting by $\mathcal E_0$ the trivial bundle (whatever the base), then $\GrX(Y; T, Z)(U)$ has
\begin{itemize}
\item objects $(\mathcal{E}, \varphi, \psi)$, where $\mathcal{E}\in \mathrm{Bun}(U)$, $\varphi: \mathcal{E}_{| U\smallsetminus T_U} \simeq \mathcal{E}_0$ is an isomorphism in $\mathrm{Bun}(U\smallsetminus T_U)$, and $\psi: \mathcal{E}_{| \Spec (\widehat{R}_J)} \simeq \mathcal{E}_0$ an isomorphism in $\mathrm{Bun}(\Spec (\widehat{R}_J))$, such that $\varphi_{| (U\smallsetminus T_U) \times_U \Spec (\widehat{R}_J) } = \psi_{| (U\smallsetminus T_U) \times_U \Spec (\widehat{R}_J) }$.
\item morphisms $(\mathcal{E}, \varphi, \psi) \to (\mathcal{E}', \varphi', \psi')$ are (iso)morphisms $\alpha: \mathcal{E} \to \mathcal{E}' $ such that the following diagrams commute $$\xymatrix{\mathcal{E}_{| U\smallsetminus T_U} \ar[rr]^-{\alpha_{|U\smallsetminus T_U}} \ar[rd]_-{\varphi} && \mathcal{E}'_{| U\smallsetminus T_U} \ar[ld]^-{\varphi '} \\ & \mathcal{E}_0 &} \qquad \xymatrix{\mathcal{E}_{|\Spec (\widehat{R}_J)} \ar[rr]^-{\alpha_{| \Spec (\widehat R_J)}} \ar[rd]_-{\psi} && \mathcal{E}'_{|\Spec (\widehat{R}_J)} \ar[ld]^-{\psi '} \\ & \mathcal{E}_0 &}$$
\end{itemize}
\end{rem}

\begin{rem}\label{flagGrassPerfgeneral} One can replace, in Definition \ref{def:flagGrassgeneral}, the stack of $\mathbf{G}$-bundles with the stack of perfect complexes (respectively, with the stack of almost perfect complexes), and the trivial $\mathbf{G}$-bundle on $Y$ with a fixed perfect complex (respectively, a fixed almost perfect complex)   $\mathcal{E}_0$ on $Y$; we thus obtain Perf-versions (resp., $\mathrm{Coh}^{-}$-versions) of the flag Grassmannian that may be denoted by $\GrX_{\mathcal{E}_0}^{\mathrm{Perf}}(Y; T, Z)$ and $\Gr_{\mathcal{E}_0}^{\mathrm{Perf}}(Y; T, Z)$ (resp., $\GrX_{\mathcal{E}_0}^{\mathrm{Coh}^{-}}(Y; T, Z)$ and $\Gr_{\mathcal{E}_0}^{\mathrm{Coh}^{-}}(Y; T, Z)$).

\end{rem}
As in the previous section, let $X$ be a complex projective smooth algebraic surface, $S$ an arbitrary Noetherian affine test scheme over $\mathbb{C}$, and $Y:=X \times S$.

\begin{defin}\label{def:flagGrass} 
For an arbitrary flag $(D,Z) \in \flags_{X}(S)$, 
\begin{itemize}
\item  we will call the \emph{Grassmannian relative to the pair} $(S,(D,Z))$ both
\begin{gather*}
\GrX_X (S)(D, Z):= \GrX(X\times S; D, Z)=  \mathbf{Bun}_{X \times S} \newtimes_{\mathbf{Bun}_{(X \times S \smallsetminus D) \cup \Zhataff}} \{\mathrm{trivial}\} \in \dSt_{X \times S} \subset \Fun(\dAff_{X \times S}\op, \inftyGpd)
\end{gather*}
\noindent and
\begin{gather*}
\Gr_X (S)(D, Z):= \Gr(X\times S; D, Z)=  \underline{\mathsf{Bun}}_{X \times S} \newtimes_{\underline{\mathsf{Bun}}_{(X \times S\smallsetminus D) \cup \Zhataff}}
\{\mathrm{trivial}\} \in \dSt_k \subset \Fun(\dAff_k\op, \inftyGpd);
\end{gather*}
\item We define $$\mathrm{Gr}_X(S)(D,Z):= \Gamma(X\times S , \GrX_X (S)(D, Z))= \Gamma(\Spec \mathbb{C}, \Gr_X (S)(D, Z)) \in \inftyGpd.$$
\end{itemize}
\end{defin}

We also give a local version of definitions \ref{def:flagGrassgeneral} and \ref{def:flagGrass}:

\begin{defin}[Local Grassmannian] \label{defin:localGr}
  Let $Y$ be a locally Noetherian scheme, and $(Z\subset T \subset Y)$ be closed subschemes. 
Denote by $\hZ \smallsetminus T$ (or $\fibf_{\hZ \smallsetminus T}$) the fiber functor
\[
 \hZ \smallsetminus T := (Y \smallsetminus T) \cap \Zhataff \simeq \widehat T \smallsetminus T \newtimes_{\Thataff} \Zhataff,
\]
and by $\hZ \cup_{\widehat T} \widehat T \smallsetminus T$ the fiber functor
\[
 \hZ \cup_{\widehat T} \widehat T \smallsetminus T := \Zhataff \newamalg_{\hZ \smallsetminus T} \widehat T \smallsetminus T.
\]
We call the local Grassmannian associated to $Z \subset T \subset Y$ both the following derived stacks:
\begin{gather*}
 \GrX^{\loc}(Y; T, Z) := 
 \mathbf{Bun}_{\Thataff} \newtimes_{\mathbf{Bun}_{\hZ \cup_{\widehat T} \widehat T \smallsetminus T}} \{\mathrm{trivial}\} =
 \doubletimes{\mathbf{Bun}_{\Thataff}}{\{\mathrm{trivial}\}}{\mathbf{Bun}_{\widehat T \smallsetminus T}}{\mathbf{Bun}_{\Zhataff}}{\mathbf{Bun}_{\hZ \smallsetminus T}}  \,\,\, \in \St_{Y}
 \\
 \Gr^{\loc}(Y; T, Z) := 
 \underline{\mathsf{Bun}}_{\Thataff} \newtimes_{\underline{\mathsf{Bun}}_{\hZ \cup_{\widehat T} \widehat T \smallsetminus T}} \{\mathrm{trivial}\} =
 \doubletimes{\underline{\mathsf{Bun}}_{\Thataff}}{\{\mathrm{trivial}\}}{\underline{\mathsf{Bun}}_{\widehat T \smallsetminus T}}{\underline{\mathsf{Bun}}_{\Zhataff}}{\underline{\mathsf{Bun}}_{\hZ \smallsetminus T}}  \,\,\, \in \St.
\end{gather*}
 If $S$ is a test scheme and $(D, Z) \in \flags_X(S)$, we also set
 \[
  \GrX^{\loc}_X(S)(D,Z) := \GrX^{\loc}(X \times S; D, Z) 
  \hspace{2em} \text{and} \hspace{2em}
  \Gr^{\loc}_X(S)(D,Z) := \Gr^{\loc}(X \times S; D, Z).
 \]
\end{defin}

The following result is a flag Grassmannian version of the well-known equivalence between the local and the global affine Grassmannian of a curve (see e.g. \cite[Theorem 1.4.2]{Zhu2017}).

\begin{lem}\label{lem:benjexo} Let $Y$ be a locally Noetherian scheme, and let $Z\subset T \subset Y$ be closed subschemes.
There are canonical restriction equivalences
\[
 \begin{tikzcd}
 \GrX(Y; T, Z) \ar{r}{\sim} & \GrX^{\loc}(Y; T, Z)
 \hspace{2em}\text{and}\hspace{2em}
 \Gr(Y; T, Z) \ar{r}{\sim} & \Gr^{\loc}(Y; T, Z)
 \end{tikzcd}
\]
If $(D, Z) \in \flags_X(S)$, they specialize to equivalences
\[
 \GrX^{\loc}_X(S)(D,Z) \simeq \GrX_X(S)(D,Z)
 \hspace{2em}\text{and}\hspace{2em}
 \Gr^{\loc}_X(S)(D,Z) \simeq \Gr_X(S)(D,Z).
\]
 \end{lem}

\begin{proof} It is enough to prove the statement for $\GrX (Y; T, Z)$, since then the statement for $\Gr (Y; T, Z)$ follows by applying the holim-preserving functor $\dSt_Y \to \dSt_k$. The special statement for $\Gr_X(S)(D, Z)$ then follows.

Consider now the commutative diagram
\[
 \begin{tikzcd}
  \GrX(Y, T, Z) \ar{r} \ar{d} \namecell{dr}{(\alpha)} &
  \mathbf{Bun}_Y \ar{r} \ar{d} \namecell{dr}{(\beta)} &
  \mathbf{Bun}_{\Thataff} \ar{d} \\
  \{\mathrm{trivial}\} \ar{r} &
  \mathbf{Bun}_{Y \smallsetminus T \cup \Zhataff } \ar{r} \ar{d} \namecell{dr}{(\gamma)} &
  \mathbf{Bun}_{\hZ \cup_{\widehat T} \widehat T \smallsetminus T} \ar{r} \ar{d} \namecell{dr}{(\delta)} &
  \mathbf{Bun}_{\Zhataff} \ar{d} \\
  &
  \mathbf{Bun}_{Y \smallsetminus T} \ar{r} &
  \mathbf{Bun}_{\widehat T \smallsetminus T} \ar{r} &
  \mathbf{Bun}_{\hZ \smallsetminus T}.
 \end{tikzcd}
\]
By definition of $Y \smallsetminus T \cup \Zhataff$ and $\hZ \cup_{\widehat T} \widehat T \smallsetminus T$, the square $(\delta)$ and the bottom rectangle (formed by $(\gamma)$ and $(\delta)$) are cartesian. It follows that $(\gamma)$ is cartesian as well.
The vertical rectangle formed by $(\beta)$ and $(\gamma)$ is also cartesian by formal glueing \cref{thm:formalglueing}. We deduce that $(\beta)$ itself is cartesian.
The square $(\alpha)$ is cartesian as well, by definition of the Grassmannian. We conclude that the upper rectangle formed by $(\alpha)$ and $(\beta)$ is also cartesian, so that:
\[
 \GrX(Y; T, Z) \simeq \mathbf{Bun}_{\Thataff} \newtimes_{\mathbf{Bun}_{\hZ \cup_{\widehat T} \widehat T \smallsetminus T}} \{\mathrm{trivial}\} =: \GrX^{\loc}(Y; T, Z).
\]
\end{proof}
 By taking global sections over $Y$ (or $X \times S$), we immediately get the following\footnote{Recall that, for a fiber functor $\fibf$ over $Y$, we denote by $\mathrm{Bun}(\fibf)$ or, equivalently, by $\mathrm{Bun}(\fibf(Y))$, the groupoid of global sections of the stack over $Y$ $\mathbf{Bun}_{\fibf}$, or equivalently the groupoid of global (i.e. over $\Spec \mathbb{C}$) sections of the corresponding stack $\underline{\mathsf{Bun}}_{\fibf}\simeq u(\mathbf{Bun}_{\fibf})$ (see beginning of \S \, \ref{setupgen}). }
 
\begin{cor}\label{cor:benjexo} Let $Y$ be a locally Noetherian scheme, and $(Z \subset T \subset Y)$ be closed subschemes. There is a natural equivalence of groupoids $$\mathrm{Gr}(Y; T, Z) \simeq \mathrm{Bun}(\widehat T) \times_{\mathrm{Bun}(\widehat T \smallsetminus T) \times_{\mathrm{Bun}((\widehat T \smallsetminus T) \times_{\Thataff} \Zhataff)}  \mathrm{Bun}(\hZ)} \,\, \{\mathrm{trivial}\} =: \mathrm{Gr}^{\loc}(Y; T,Z) .$$
In particular, for $(D,Z) \in \flags_X(S)$, there is a natural bijection $$\mathrm{Gr}_X(S)(D, Z) \simeq \mathrm{Bun}(\widehat{D}) \times_{\mathrm{Bun}(\widehat{D} \smallsetminus D) \times_{\mathrm{Bun}((\widehat{D} \smallsetminus D) \times_{\wh{D}^{\affinize}} \Zhataff)}  \mathrm{Bun}(\hZ)} \,\, \{\mathrm{trivial}\} =: \mathrm{Gr}^{\loc}_X(S)(D,Z).$$
\end{cor}

\begin{rem}\label{localglobalGrPerfCohminus} Continuing Remark \ref{flagGrassPerfgeneral}, there are obvious \emph{local} versions of $\GrX_{\mathcal{E}_0}^{\mathrm{Perf}}(Y; T, Z)$, $\Gr_{\mathcal{E}_0}^{\mathrm{Perf}}(Y; T, Z)$, $\GrX_{\mathcal{E}_0}^{\mathrm{Coh}^{-}}(Y; T, Z)$ and $\Gr_{\mathcal{E}_0}^{\mathrm{Coh}^{-}}(Y; T, Z)$ (analogs of Definition \ref{defin:localGr}). Moreover, since  only formal gluing (\cref{thm:formalglueing}) have been used in the proof of Lemma \ref{lem:benjexo}, a similar ``local=global'' statement holds for these Perf and $\mathrm{Coh}^{-}$ versions.
\end{rem}

\begin{lem}\label{GrIsSetsValued} 
Let $Y$ be a locally noetherian scheme and $Z\subset T $ be closed subschemes in $Y$.
If $Y\setminus T$ is quasi-compact and schematically dense in $Y$\,\footnote{For example, if there exists an effective Cartier divisor $D \hookrightarrow Y$ such that $D$ and $T$ share the same underlying topological subspace of $|Y|$ (\cite[TAG07ZU]{stacks-project}).}, then the truncation $\mathrm{t}_0(\Gr(Y; T, Z))$  is actually a sheaf of sets. In particular, its global sections $\mathrm{Gr}(Y;T,Z)$ form a set.
\end{lem}

\begin{proof}Let $S=\Spec R$ be a Noetherian affine scheme.. We will use the notations of Remark \ref{descriptionofGrassgenbold}. We have to prove that $\Gr(Y; T, Z)(S)$ is a set (i.e. equivalent to a discrete groupoid). Now the obvious functor $\Gr(Y; T, Z)(S) \to \Gr(Y; T, \emptyset)(S)$ sending $(\mathcal{E}, \varphi, \psi)$ to $(\mathcal{E}, \varphi)$ (and identical on morphisms) is faithful, therefore it will be enough to show that $\Gr(Y; T, \emptyset)(S)$ is a set. Define $T_S:= T \times_Y (S\times Y)$. Observe that, since being quasicompact open and scheme-theoretically dense is stable under flat pullbacks (\cite[TAG081I]{stacks-project}) and $S\times Y \to Y$ is flat, we have that $(Y\setminus T) \times_Y (S \times Y)= (S \times Y)\setminus T_S$ is quasi-compact and schematically dense inside $S \times Y$.

Let $(\mathcal{E}, \varphi), (\mathcal{E'}, \varphi') \in \Gr(Y; T, \emptyset)(S)$, where $\mathcal{E}\to S \times Y$ is a $\mathbf{G}$-bundle, $\varphi$ a trivialization of $\mathcal{E}$ on  $(S \times Y) \setminus T_S$ (and similarly for $(\mathcal{E'}, \varphi')$).
Let $\alpha, \beta: (\mathcal{E}, \varphi) \to (\mathcal{E'}, \varphi')$ be morphisms (hence isomorphisms) in the groupoid $\Gr(Y; T, \emptyset)(S)$.
Since a groupoid is equivalent to a set if and only if its Hom-sets are either empty or consist of a singleton, it will be enough to show that $\alpha=\beta$.
Consider $\beta^{-1}\circ \alpha: (\mathcal{E}, \varphi) \to (\mathcal{E}, \varphi)$ and $\mathrm{id}: (\mathcal{E}, \varphi) \to (\mathcal{E}, \varphi)$.
We want to prove $\beta^{-1}\circ \alpha= \mathrm{id}$.
Observe that $\mathrm{Aut}_{\mathrm{Bun}(S\times Y)}(\mathcal{E})$ injects into $\mathrm{Hom}_{\mathbf{Sch}_{\mathbb{C}}}(\mathcal{E}, \mathbf{G})$ (the image being $\mathbf{G}$-equivariant maps). Let us denote by $\rho_{\alpha, \beta}: \mathcal{E} \to \mathbf{G}$, and $\rho_{\mathrm{id}}: \mathcal{E} \to \mathbf{G}$ the images of $\beta^{-1}\circ \alpha$, and of $\mathrm{id}$ inside $\mathrm{Hom}_{\mathbf{Sch}_{\mathbb{C}}}(\mathcal{E}, \mathbf{G})$. It will be enough to show that $\rho_{\alpha, \beta}= \rho_{\mathrm{id}}$.\\
We know, by definition of morphisms in the groupoid $\Gr(Y; T, \emptyset)(S)$, that $\beta^{-1}\circ \alpha$ and $\mathrm{id}$ agree when restricted to $\mathcal{E}_{|\,(S\times Y) \setminus T_S} $.
Therefore, $\rho_{\alpha, \beta}$ and $\rho_{\mathrm{id}}$ coincide on 
$$ \mathcal{E} \times_{S\times Y} ((S \times Y)\setminus T_S)  = \mathcal{E} \setminus (\mathcal{E}\times_{S\times Y} T_S). $$
Now, since any $\mathbf{G}$-bundle is (faithfully) flat over its base, $\mathcal{E} \to S\times Y$ is flat, $\mathcal{E} \setminus (\mathcal{E}\times_{S\times Y} T_S)$ is quasicompact open and schematically dense in $\mathcal{E}$ (again by \cite[TAG081I]{stacks-project}). Now, $\mathbf{G}$ is a separated scheme, and the two maps $\rho_{\alpha, \beta}, \,\rho_{\mathrm{id}} : \mathcal{E} \to \mathbf{G}$ coincide on the quasicompact open and schematically dense $\mathcal{E} \setminus (\mathcal{E}\times_{S\times Y} T_S)$, therefore they coincide on all of $\mathcal{E}$ (\cite[TAG01RH]{stacks-project}). 
\end{proof}

\subsection{Functorialities and flat connection}\label{funtorialitaGr} We prove here the functorialities of $\mathrm{Gr}_X(S)(F)$ both in the flag $F$, and in the test scheme $S$. See \cref{cor:functorialGrass} below.
We fix a smooth projective surface $X$ over $k$.

To any affine scheme $S$ over $k$ and any closed subschemes $D, Z \subset X \times S_\red$, we associate two (ind-)schemes flat over $Y := X \times S$
\[
 Y \smallsetminus D \to Y \hspace{1em}\text{and}\hspace{1em} \hZ \to Y.
\]
They induces underived fiber functors $\tilde\fibf_{Y \smallsetminus D}$ and $\tilde\fibf_{\hZ}$ over $Y$. We denote by $\tilde\fibf_{Y \smallsetminus D \cap \hZ}$ the intersection
\[
 \tilde\fibf_{Y \smallsetminus D \cap \hZ} := \tilde\fibf_{Y \smallsetminus D} \newtimes_{Y} \tilde\fibf_{\hZ}^\affinize.
\]
Let $\Lambda$ be the category $\begin{tikzcd}[cramped,column sep=small] \bullet & \bullet \ar{r} \ar{l} & \bullet \end{tikzcd}$.
We denote by $\Lambda^{D,Z}_Y$ the diagram of fiber functors
\[
 \Lambda^{D,Z}_Y = \left(
 \begin{tikzcd}
  \tilde\fibf_{Y \smallsetminus D} & \ar{l} \tilde\fibf_{Y \smallsetminus D \cap \hZ} \ar{r} & \tilde\fibf_{\hZ}^\affinize
 \end{tikzcd}
 \right) \in \FibF_Y^{\Lambda}.
\]
Notice that for any morphism of Noetherian affine schemes $S' \to S$, pulling back along the induced morphism $f \colon Y' = X \times S' \to X \times S = Y$ gives, by (an underived version of) \cref{lem:restrictionfiberfunctor}, a canonical isomorphism of diagrams
\begin{equation}\label{eq:restrictionLambda}
 \begin{tikzcd}[row sep = small]
  \makebox[0pt][r]{$f^{-1}\left(\Lambda^{D,Z}_Y\right) = \Big($}
  f^{-1}(\tilde\fibf_{Y \smallsetminus D}) \ar[phantom]{d}[sloped]{\simeq} & \ar{l} f^{-1}(\tilde\fibf_{Y \smallsetminus D \cap \hZ}) \ar[phantom]{d}[sloped]{\simeq} \ar{r} & f^{-1}(\tilde\fibf_{\hZ}^\affinize) \ar[phantom]{d}[sloped]{\simeq}
  \makebox[0pt][l]{$\Big)$}
  \\
  \makebox[0pt][r]{$\Big($} \tilde\fibf_{Y' \smallsetminus D'} & \ar{l} \tilde\fibf_{Y' \smallsetminus D' \cap \hZ'} \ar{r} & \tilde\fibf_{\hZ'}^\affinize \makebox[0pt][l]{$\Big) = \Lambda^{D',Z'}_{Y'}.$}
 \end{tikzcd}
\end{equation}
where $D' = D \times_{S_\red} S'_\red$ and $Z' = Z \times_{S_\red} S'_\red$.

Finally, the assignment $(D, Z) \mapsto \Lambda^{D,Z}_Y$ is covariant in $Z$ and contravariant in $D$.
In particular, given a commutative diagram (of closed subschemes)
\[
 \begin{tikzcd}
  Z_1 \ar[hook]{r} \ar[hook]{d} \namecell{dr}{(\alpha)} & Z_2 \ar[hook]{d} \\ D_1 \ar[hook]{r} & D_2 \ar[hook]{r} & X \times S_\red,
 \end{tikzcd}
\]
we get a co-correspondence (of diagrams)
\[
 \tau_\alpha \colon \begin{tikzcd}
 \Lambda^{D_1,Z_1}_Y \ar{r}{\overrightarrow\tau_\alpha} &
  \Lambda^{D_1,Z_2}_Y &
  \Lambda^{D_2,Z_2}_Y \ar{l}[swap]{\overleftarrow\tau_\alpha}.
 \end{tikzcd}
\]
Let $\closedpairs_{X,\dR}(S)$ denote\footnote{The $\dR$ lower script highlights that $\closedpairs_{X,\dR}$ really is the de Rham stack of a more general construction.} the category of closed immersions $Z \hookrightarrow D \hookrightarrow X \times S_\red$ (and morphisms being commutative diagrams as above).
It is cumbersome but easy to check that this construction gives a lax\footnote{Recall that a lax functor is a functor $F$ where the composition is not quite preserved. This lack of compatibility is replaced by a natural transformation evaluating to a map $F(f) \circ F(g) \to F(f \circ g)$ for each pair of composable arrows $(f,g)$.} functor to the $2$-category of correspondences in diagrams of underived fiber functors:
\[
\begin{tikzcd}[row sep=0pt]
 \closedpairs_{X,\dR}(S) \ar{r}{\mathrm{lax}} & \FibF_Y^{\Lambda,\mathrm{cocorr}}\makebox[0pt][l]{${} = \FibF_{X \times S}^{\Lambda,\mathrm{cocorr}}$}
\\
 (D, Z) \ar[mapsto]{r} & \Lambda^{D,Z}_Y
\\
 \alpha \ar[mapsto]{r} & \tau_\alpha.
\end{tikzcd}
\]
Using \cref{eq:restrictionLambda} and \ref{sublem:restrictionfflimcolim}, those assemble into a natural transformation (by lax functors)
\[
\begin{tikzcd}
 \closedpairs_{X,\dR} \ar{r}{\mathrm{lax}} & \FibF_{X \times -}^{\Lambda,\mathrm{cocorr}}
\end{tikzcd}
\]
Composition with the inclusion $i \colon \FibF_{X \times -} \to \dFibF_{X \times -}$ of fiber functors into derived fiber functors, and using the canonical morphisms $i\fibf_0 \amalg_{i\fibf_1} i\fibf_2 \to i(\fibf_0 \amalg_{\fibf_1} \fibf_2)$, we get a natural transformation (by lax $\infty$-functors, see \cite[Chap. 10 \S 3.]{GaiRozI}):
\[
\begin{tikzcd}
 \closedpairs_{X,\dR} \ar{r}{\mathrm{lax}} & \dFibF_{X \times -}^{\Lambda,\mathrm{cocorr}}.
\end{tikzcd}
\]
Taking pointwise pushouts of our $\Lambda$-indexed diagrams (of derived fiber functors) yields a natural transformation (by lax $\infty$-functors)
\[
\begin{tikzcd}[row sep=0pt]
 \makebox[0pt][r]{$\fibf^\cup \colon{}$} \closedpairs_{X,\dR} \ar{r}{\mathrm{lax}} & \dFibF_{X \times -}^\mathrm{cocorr}.
\end{tikzcd}
\]
which evaluates, for $S \in \Aff$ and $Y := X \times S$, to
\begin{equation}\label{eq:fcup}
\newsavebox{\mycorrbox}
\begin{lrbox}{\mycorrbox}%
$\left(\begin{tikzcd}[cramped,column sep=small]
  \fibf_{Y \smallsetminus D_1 \cup \hZ_1} \ar{r} &
  \fibf_{Y \smallsetminus D_1 \cup \hZ_2} &
  \ar{l} \fibf_{Y \smallsetminus D_2 \cup \hZ_2}
\end{tikzcd}\right)$
\end{lrbox}
\newsavebox{\mylaxcorrbox}
\begin{lrbox}{\mylaxcorrbox}%
$\left(\begin{tikzcd}[cramped,column sep=small, row sep=tiny]
  & \displaystyle \fibf_{Y \smallsetminus D_1 \cup \hZ_2} \newamalg_{\fibf_{Y \smallsetminus D_2 \cup \hZ_2}} \fibf_{Y \smallsetminus D_2 \cup \hZ_3} \ar{dd} &
\\
  \fibf_{Y \smallsetminus D_1 \cup \hZ_1} \ar{ur} \ar{dr} & &
  \ar{ul} \ar{dl} \fibf_{Y \smallsetminus D_3 \cup \hZ_3}
\\
  & \fibf_{Y \smallsetminus D_1 \cup \hZ_3} &
\end{tikzcd}\right).$
\end{lrbox}
\newsavebox{\mycomposablebox}
\begin{lrbox}{\mycomposablebox}%
$\left(\begin{tikzcd}[cramped,row sep=small]
(D_1, Z_1) \ar{d} \\ (D_2, Z_2) \ar{d} \\ (D_3, Z_3)
\end{tikzcd}\right)$
\end{lrbox}
\newsavebox{\mymorphbox}
\begin{lrbox}{\mymorphbox}%
$\left(\begin{tikzcd}[cramped,row sep=small]
(D_1, Z_1) \ar{d} \\ (D_2, Z_2)
\end{tikzcd}\right)$%
\end{lrbox}
\begin{tikzcd}[row sep=0pt, column sep=large,
    ,/tikz/column 1/.append style={anchor=base east}
    ,/tikz/column 2/.append style={anchor=base west}]
 \closedpairs_{X,\dR}(S) \ar{r}{\fibf^\cup(S)} & \dFibF_{X \times S}^\mathrm{cocorr}
\\
 (D, Z) \ar[mapsto]{r} & \fibf_{Y \smallsetminus D \cup \hZ} \makebox[0pt][l]{$\displaystyle {} := \fibf_{Y \smallsetminus D} \newamalg_{\fibf_{Y \smallsetminus D \cap \hZ}} \fibf_{\hZ}^\affinize$}
\\
\usebox{\mymorphbox} \ar[mapsto]{r} & \usebox{\mycorrbox}
\\
\usebox{\mycomposablebox} \ar[mapsto]{r}{\mathrm{lax}}[swap]{\mathrm{structure}} & \usebox{\mylaxcorrbox}
\end{tikzcd}
\end{equation}

\begin{lem}\label{lem:changeZ}
 Fix $S$ a test scheme and $Y := X \times S$.
 Let $(D_1, Z_1) \to (D_2, Z_2) \in \closedpairs_{X,\dR}(S)$.
 If the canonical morphism $Z_1 \to Z_2 \times_{D_2} D_1$ is an homeomorphism (i.e. induces an isomorphism between the associated reduced schemes), then the canonical morphism
 \[
  \fibf_{Y \smallsetminus D_1 \cup \hZ_1} \to \fibf_{Y \smallsetminus D_1 \cup \hZ_2}
 \]
 induces a equivalences of derived stacks over $Y$:
 \begin{align*}
  \PerfX_{Y \smallsetminus D_1 \cup \hZ_2} &\overset\sim\longrightarrow \PerfX_{Y \smallsetminus D_1 \cup \hZ_1}, \\
  \BunGX_{Y \smallsetminus D_1 \cup \hZ_2} &\overset\sim\longrightarrow \BunGX_{Y \smallsetminus D_1 \cup \hZ_1} \\
   \text{and}\hspace{1em} \CohX_{Y \smallsetminus D_1 \cup \hZ_2} &\overset\sim\longrightarrow \CohX_{Y \smallsetminus D_1 \cup \hZ_1}. \\
 \end{align*}
\end{lem}

\begin{proof}
 Consider the commutative diagram (of categorical derived stacks over $Y$)
 \[
 \begin{tikzcd}
  \PerfX_{Y \smallsetminus D_1 \cup \hZ_2} \ar{r} \ar{d} \namecell{dr}{(\sigma)} & \PerfX_{\hZ_2} \ar{d} \ar{r} \namecell{dr}{(\tau)} & \PerfX_{\hZ_1} \ar{d} \\
  \PerfX_{Y \smallsetminus D_1} \ar{r} & \PerfX_{\hZ_2\smallsetminus D_1} \ar{r} & \PerfX_{\hZ_1 \smallsetminus D_1}.
 \end{tikzcd}
 \]
 The square $(\sigma)$ is Cartesian by definition.
 Fixing a derived test scheme $T$ over $Y$, we can apply \cref{prop:localformalglueing} to $D_1 \times_Y \fibf_{\hZ_2}^\affinize(T) \subset \fibf_{\hZ_2}^\affinize(T)$. We deduce that $(\tau)$ is Cartesian as well.
 The large rectangle is thus Cartesian and the result follows (with a similar argument for the cases of $\mathbf G$-bundles or of complexes of coherent sheaves).
\end{proof}

\newcommand{\stackofcatstacksunder}[1]{\raisebox{0.3em}{$#1$}\raisebox{0.2em}{$/$}\stackofcatstacks}
\newcommand{\stacksunder}[1]{\raisebox{0.3em}{$#1$}\raisebox{0.2em}{$/$}\dSt^{\inftyCat}}
 Let $\mathcal{M} \colon Y \mapsto \mathcal{M}_Y$ denote either $\PerfX$, $\BunGX$ or $\CohX$.
  Denote by $\stackofcatstacksunder{\mathcal{M}}_X$ the stack (in $\infty$-categories) of derived stacks in $\infty$-categories under $\mathcal{M}$:
 \[
  \stackofcatstacksunder{\mathcal{M}}_X \colon S \mapsto \stacksunder{\mathcal{M}_{X \times S}}_{X \times S}
 \]
 The stack $\mathcal{M}$ induces a morphism categorical prestacks
 \[
  \begin{tikzcd}[row sep=0pt]
   \dFibF_{X \times -}\op \ar{r} & \stackofcatstacksunder{\mathcal{M}}_X \\
   \llap{$\mathrm M \colon {}$} (S, \fibf) \ar[mapsto]{r} & \mathcal{M}_{\fibf} \makebox[0pt][l]{${} \in \stacksunder{\mathcal{M}_{X\times S}}_{X \times S}.$}
  \end{tikzcd}
 \]
 
 \begin{lem}\label{lem:Ccocorr}
 Composing the lax $\infty$-functor $\fibf^\cup \colon \closedpairs_{X,\dR} \to \dFibF_{X \times -}^{\mathrm{cocorr}}$ with $\mathrm{M}$ gives a morphism (still denote by $\mathrm M$) of prestacks in $\infty$-categories
 \[
 \begin{tikzcd}[row sep=0pt]
  \closedpairs_{X,\dR} \ar{r}{\mathrm M}
  &
  \left(\left(\stackofcatstacksunder{\mathcal{M}}_X\right)\op\right)^{\mathrm{cocorr}}
  \\
   (S, (D, Z)) \ar[mapsto]{r} & \mathcal{M}_{\fibf_{X \times S \smallsetminus D \cup \hZ}}.
 \end{tikzcd}
 \]
 \end{lem}
 
 \begin{proof}
  This follows from the fact that for any $Y$, the functor $\mathcal{M}_Y \colon \dFibF_Y\op \to \dSt^{\inftyCat}_Y$ maps pushout squares to pullback squares.
 \end{proof}

 \begin{defin}
   For any test affine scheme $S$, let $\closedpairs_{X,\dR}^\mathrm{h}(S)$ the (non full) subcategory of $\closedpairs_{X,\dR}(S)$ containing every object, but keeping only the morphisms $(D_1, Z_1) \to (D_2, Z_2)$ such that $Z_1 \to Z_2 \times_{D_2} D_1$ is a homeomorphism (i.e. an isomorphism on the associated reduced schemes).
   This condition being stable under base change, this defines a sub-prestack $\closedpairs_{X,\dR}^\mathrm{h}$ of $\closedpairs_{X,\dR}$.
\end{defin}

\begin{thm}\label{thm:functglob}
The composition
\[
 \begin{tikzcd}
  \closedpairs_{X,\dR}^\mathrm{h} \subset 
  \closedpairs_{X,\dR} \ar{r}{\mathrm M}
  &
  \left(\left(\stackofcatstacksunder{\mathcal{M}}_X\right)\op\right)^{\mathrm{cocorr}}
 \end{tikzcd}
\]
 factors through $\stackofcatstacks_X \subset \left(\left(\stackofcatstacks_X\right)\op\right)^{\mathrm{cocorr}}$, therefore inducing
 \[
 \begin{tikzcd}[row sep=0pt]
  \closedpairs_{X,\dR}^{\mathrm{h}} \ar{r}{\mathrm M}
  &
   \stackofcatstacksunder{\mathcal{M}}_X
     \\
  (S,(D,Z)) \ar[mapsto]{r} & \mathcal{M}_{\fibf_{X \times S \smallsetminus D \cup \hZ}}.
 \end{tikzcd}
 \]
\end{thm}
\begin{proof}
 Fix a test scheme $S \in \Aff_k$.
 From \cref{lem:changeZ} and \cref{eq:fcup}, we get:
 \begin{enumerate}[label=(\arabic*),ref={Assertion (\arabic*)}]
  \item For each $S \in \Aff_k$, the lax $\infty$-functor $\mathrm M_S \colon \closedpairs_{X,\dR}^\mathrm{h}(S) \to \left(\left(\stacksunder{\mathcal{M}_{X \times S}}_{X \times S}\right)\op\right)^\mathrm{cocorr}$ is strict (i.e. the morphisms making the lax structure are equivalences).
  \item For each $S \in \Aff_k$, the $\infty$-functor $\mathrm M_S$ has values in the (non-full) sub-$(\infty,1)$-category $\stacksunder{\mathcal{M}_{X \times S}}_{X \times S}$ (embedded the standard way into the category of cocorrespondences).
 \end{enumerate}
 In particular, the map $\mathrm M$ factors as $\closedpairs_{X,\dR}^\mathrm{h} \to \stackofcatstacksunder{\underline{\mathcal{M}}_X}_X$ giving the announced functor.
\end{proof}

\begin{cor}\label{cor:functorialGrass}
 The flag Grassmannians $\GrX(X \times S; D, Z)$ (resp. $\Gr(X \times S; D, Z)$) assemble into a morphism
 \[
  \GrX_X \colon \closedpairs_{X,\dR}^\mathrm{h} \to \stackofstacks_X \subset \stackofcatstacks_X
  \hspace{1cm}\left(\text{resp. } \Gr_X \colon \closedpairs_{X,\dR}^\mathrm{h} \to \stackofstacks_k\right)
 \]
 of stacks $(\Aff_k)\op \to \inftyCat$.
\end{cor}

\begin{proof}
 Consider the morphism
 \[
 \begin{tikzcd}[row sep=0pt]
  \closedpairs_{X,\dR}^\mathrm{h} \ar{r} & \stackofcatstacksunder{\mathbf{Bun}^\mathbf{G}}_X \\
  (S,(D,Z)) \ar[mapsto]{r} & \BunGX_{X \times S}\left(\fibf_{X \times S \smallsetminus D \cup \hZ}\right).
 \end{tikzcd}
 \]
 Taking pointwise the homotopy fiber of the map $\BunGX_{X \times S}\left(X \times S\right) \to \BunGX_{X \times S}\left(\fibf_{X \times S \smallsetminus D \cup \hZ}\right)$ gives the announced functoriality of the Grassmannian $\closedpairs_{X,\dR}^\mathrm{h} \to \stackofcatstacks_X$. This functor has values in the substack of derived stacks in $\infty$-groupoids $\stackofstacks_X$.
 
 The morphism $\Gr_X$ is obtained from $\GrX_X$ be composing with the natural morphism $\stackofstacks_X \to \stackofstacks_k$ mapping a derived stack over $X \times S$ to its derived stack of relative global sections.
\end{proof}

\begin{rems}
 \item The result of \cref{cor:functorialGrass}, and in particular the fact that the flag Grassmannian can be seen as a functor out of the de Rham stack of closed pairs $(D,Z)$, also implies that the flag Grassmannian carries a natural flat connection.
 \item\label{rem:alsoforPerf} Using \cref{thm:functglob}, together with remarks \ref{flagGrassPerfgeneral} and \ref{localglobalGrPerfCohminus}, one can straightforwardly define functorial flags Grassmannians $\GrX_{X, \mathcal{E}_0}^{\mathrm{Perf}}$, $\Gr_{X,\mathcal{E}_0}^{\mathrm{Perf}}$, $\GrX_{X,\mathcal{E}_0}^{\mathrm{Coh}^{-}}$ and $\Gr_{X,\mathcal{E}_0}^{\mathrm{Coh}^{-}}$ for perfect or coherent complexes. Although our focus is on $\mathbf G$-bundles in this section, we will keep this level of generality in the upcoming technical lemmas (as it does not impact the complexity of the arguments and may be useful for further work in this direction).
\end{rems}

\subsection{Topological invariance of the flag Grassmannian}\label{topinv}
We prove here that the flag Grassmannian is, in an appropriate sense, insensitive to non reduced structures on flags.
We already know it carries a flat connection. With the following lemma, we add that it is invariant under reduced equivalence.

\begin{lem}\label{topinvariance}
Let $S$ be a test scheme and $Z \subset D \subset X \times S_\red$ be closed subschemes. The canonical diagram
\[
 \begin{tikzcd}
  Z_\red \ar[hook]{r} \namecell{dr}{(\alpha)} \ar[hook]{d} & Z \ar[hook]{d} \\ D_\red \ar[hook]{r} & D \ar[hook]{r} & X \times S_\red
 \end{tikzcd}
\]
defines a morphism $\alpha$ in $\closedpairs_{X,\dR}(S)$.
The cocorrespondence (of diagrams of fiber functors over $Y := X \times S$)
\[
 \tau_\alpha \colon \begin{tikzcd}
 \Lambda^{D_\red,Z_\red}_Y \ar{r}{\overrightarrow\tau_\alpha} &
  \Lambda^{D_\red,Z}_Y &
  \Lambda^{D,Z}_Y \ar{l}[swap]{\overleftarrow\tau_\alpha}
 \end{tikzcd}
\]
is an isomorphism (i.e. both $\overrightarrow\tau_\alpha$ and $\overleftarrow\tau_\alpha$ are invertible).
\end{lem}

\begin{proof}
 The (ind-)schemes $Y \smallsetminus D$ and $\hZ$ only depend on the underlying closed subsets of $D$ and $Z$, respectively.
 As a consequence, the same is true of the fiber functors $\tilde \fibf_{Y \smallsetminus D}$ and $\tilde \fibf_{\hZ}^\affinize$. The result follows.
\end{proof}

\begin{cor}\label{topinvarianceofGr}
 For any $(D,Z) \in \closedpairs_{X,\dR}(S)$, there is a canonical equivalence
 \[
  \GrX(X \times S; D, Z) \simeq \GrX(X \times S; D_\red, Z_\red) \in \dSt_{X \times S}.
 \]
\end{cor}

The following Corollary establish the topological invariance of the flag Grassmannian:  the values of $\Gr_X (S)$ on two reduced equivalent flags agree (up to a canonical isomorphism). 

\begin{cor}\label{topinvbis} Let $X$ be a complex projective smooth algebraic surface, $S$ an arbitrary Noetherian affine scheme over $\mathbb{C}$, and $Y=X \times S$. If $(D,Z), (D', Z')  \in \flags_{X}(S)$ are such that $D_{\red}= D'_{\red} $ and $Z_{\red}= Z'_{\red}$, then there is a canonical equivalence
\[
 \GrX_X(S)(D, Z) := \GrX(Y; D, Z) \simeq \GrX(Y; D_\red, Z_\red) = \GrX(Y; D'_\red, Z'_\red) \simeq \GrX(Y; D', Z') =: \GrX_X(S)(D', Z').
\]
\end{cor}

\begin{rems}
 \item It would be tempting, by \cref{topinvbis}, to consider the quotient of $\flags_{X}(S)$ modulo the reduced equivalence relation, and then define our Grassmannian as a functor out of this quotient. However, this would not be so helpful for our purposes, since unfortunately the corresponding version of $\flags^{\bullet, \mathrm{good}}_{X}(S)$ modulo reduced equivalence will no longer be $2$-Segal.
 \item\label{rem:alsoforPerf2} As in \cref{rem:alsoforPerf}, the above results (corollaries \ref{topinvarianceofGr} and \ref{topinvbis}) can be trivially generalized to perfect or coherent complexes, i.e. to $\GrX_{\mathcal{E}_0}^{\mathrm{Perf}}(X\times S; -,-)$ and $\GrX_{\mathcal{E}_0}^{\mathrm{Coh}^{-}}(X \times S; -, -)$.
 \end{rems}

\subsection{Functorialities of the local flag Grassmannian}
Since by \cref{lem:benjexo}, the local Grassmannian $\GrX_X^\loc$ is equivalent to the the global one $\GrX_X$, we get covariant functorialities
\[
 \GrX_X^\loc \colon \closedpairs_{X,\dR}^\mathrm{h} \to \stackofstacks_X \hspace{1em} \text{and} \hspace{1em} \Gr_X^\loc \colon \closedpairs_{X,\dR}^\mathrm{h} \to \stackofstacks_k.
\]
There is however a more interesting and less intuitive \emph{contravariant} functoriality, mapping a morphism in $\closedpairs_{X,\dR}^\mathrm{h}(S)$ corresponding to a reduced-Cartesian commutative square
\[
 \begin{tikzcd}
  Z_1 \ar[hook]{r} \ar[hook]{d} \ar[phantom]{dr}{(\alpha)} & Z_2 \ar[hook]{d} \\ D_1 \ar[hook]{r} & D_2 \ar[hook]{r} & X \times S_\red
 \end{tikzcd}
\]
to a pullback morphism $\alpha^* \colon \GrX^\loc_X(Y; D_2, Z_2) \to \GrX^\loc_X(Y; D_1, Z_1)$:
\begin{equation}\label{eq:cocorr-for-projection}
\begin{tikzcd}
 \makebox[0pt][r]{$\displaystyle \GrX^\loc_X(Y; D_2, Z_2) := {}$}
 \displaystyle \BunGX_{\widehat D_2} \newtimes_{\BunGX_{\widehat D_2 \smallsetminus D_2 \cup \hZ_2}} \{\mathrm{trivial}\} \ar{d}
 \\
 \displaystyle \BunGX_{\widehat D_1} \newtimes_{\BunGX_{\widehat D_1 \smallsetminus D_2 \cup \hZ_1}} \{\mathrm{trivial}\}
 \\
 \displaystyle \BunGX_{\widehat D_1} \newtimes_{\BunGX_{\widehat D_1 \smallsetminus D_1 \cup \hZ_1}} \{\mathrm{trivial}\} \ar{u}[sloped]{\sim}[swap]{\text{ see \cref{magiclemma} below}}
 \makebox[0pt][l]{$\displaystyle {} =: \GrX^\loc_X(Y; D_1, Z_1).$}
\end{tikzcd}
\end{equation}

 One could prove the above pullback morphisms are functorial enough to get a morphism
 \[
  (\GrX_X^\loc, (-)^*) \colon (\closedpairs_{X,\dR}^\mathrm{h})\op \to \stackofstacks_X.
 \]
 We will however be more interested in the compatibility between the pullback along one morphism of flags and the covariant functoriality we established so far.
 This compatibility can be thought as a base-change formula, as usual only valid under additional assumptions -- see \cref{rem:basechange} below.

 Let us construct this pullback formally, and prove it is partially compatible with the covariant functoriality.
 \newcommand{\dsq}{{\boxminus}}
 Denote by $\dsq := \Lambda^2_0 \times \Delta^1$ the category depicted as
 \[
  \begin{tikzcd}[row sep=small, column sep=small]
   \bullet \ar{r} & \bullet \\
   \bullet \ar{r} \ar{d} \ar{u} & \bullet \ar{d} \ar{u} \\
   \bullet \ar{r} & \bullet \rlap{.}
  \end{tikzcd}
 \]
 Let $\dFibF_{X \times -}^\dsq$ denote the stack of $\dsq$-diagrams of fiber functors, and by $(\dFibF_{X \times -}^\dsq)^\mathrm{cocorr}$ the $(\infty, 2)$-categorical stack of cocorrespondences therein.
  Informally, objects of $(\dFibF_{X \times S}^\dsq)^\mathrm{cocorr}$ are $\dsq$-diagrams of fiber functors,  morphisms are cocorrespondences of diagrams, thus commutative diagrams of the shape
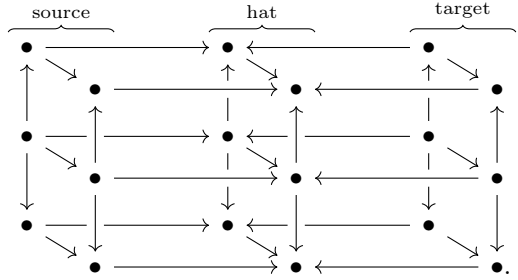
\[
 \begin{tikzcd}[row sep=0, column sep=small]
  \bullet \ar{dr} \ar{rrrr} \arrow[r, start anchor=north west, end anchor=north east, no head, decorate, decoration={brace}, "\text{source}" above=3pt] & \phantom{\bullet} &&& \bullet \ar{dr} \ar[from=dd] \arrow[r, start anchor=north west, end anchor=north east, no head, decorate, decoration={brace}, "\text{hat}" above=3pt] & \phantom{\bullet} &&& \bullet \ar{llll} \ar{dr} \ar[from=dd] \arrow[r, start anchor=north west, end anchor=north east, no head, decorate, decoration={brace}, "\text{target}" above=3pt] & \phantom{\bullet}
  \\[4pt]
  & \bullet &&& & \ar[crossing over, from={llll}] \bullet &&& & \bullet \ar[crossing over]{llll}
  \\[6pt]
  \bullet \ar{uu} \ar{dr} \ar{dd} \ar{rrrr} & &&& \bullet \ar{dr} \ar{dd} & &&& \bullet \ar{dr} \ar{dd} \ar{llll}
  \\[4pt]
  & \bullet \ar[crossing over]{uu} \ar[crossing over]{rrrr} &&& & \bullet \ar[crossing over]{uu} &&& & \bullet \ar{uu} \ar{dd} \ar[crossing over]{llll}
  \\[6pt]
  \bullet \ar{dr} \ar{rrrr} & &&& \bullet \ar{dr} & &&& \bullet \ar{llll} \ar{dr}
  \\[4pt]
  & \bullet \ar{rrrr} \ar[from=uu, crossing over] &&& & \bullet \ar[from=uu, crossing over] &&& & \bullet \ar{llll} \rlap{.}
 \end{tikzcd}
\]
Finally, $2$-morphisms are maps of $\dsq$-diagrams between the hats of the cocorrespondences involved (together with the obvious commutativity constraints).

Denote by $\closedpairs_{X,\dR}^{\Delta^1}$ the stack of arrows in $\closedpairs_{X,\dR}$.
Using the same method as in the construction of $\fibf^\cup$ (see \cref{eq:fcup}), we obtain a natural transformation by lax $\infty$-functors $\fibf_\loc^\cup \colon \closedpairs_{X,\dR}^{\Delta^1} \to (\dFibF_{X \times S}^\dsq)^\mathrm{cocorr}$ evaluating, on $S \in \Aff_k$, to

\[
\newsavebox{\locmycomposablebox}
\begin{lrbox}{\locmycomposablebox}%
$\left(\begin{tikzcd}[cramped,column sep=small, row sep=small]
(D_1, Z_1) \ar{d} \ar{r} & (D_2, Z_2) \ar{d} \\ (D_1', Z_1') \ar{d} \ar{r} & (D_2', Z_2') \ar{d} \\ (D_1'', Z_1'') \ar{r} & (D_2'', Z_2'')
\end{tikzcd}\right)$
\end{lrbox}
\newsavebox{\locobjbox}
\begin{lrbox}{\locobjbox}%
$\left(\begin{tikzcd}[cramped,column sep=small]
(D_1, Z_1) \ar{r} & (D_2, Z_2)
\end{tikzcd}\right)$%
\end{lrbox}
\newsavebox{\locmorphbox}
\begin{lrbox}{\locmorphbox}%
$\left(\begin{tikzcd}[cramped,row sep=small, column sep=small]
(D_1, Z_1) \ar{d} \ar{r} & (D_2, Z_2) \ar{d} \\ (D_1', Z_1') \ar{r} & (D_2', Z_2')
\end{tikzcd}\right)$%
\end{lrbox}
\newsavebox{\locimageobj}
\begin{lrbox}{\locimageobj}%
$\left(\begin{tikzcd}[cramped,row sep=small, column sep=small]
  \fibf_{\widehat D_1 \smallsetminus D_1 \cup \hZ_1} \ar{r} & \fibf_{\widehat D_1}^\affinize
\\
  \fibf_{\widehat D_1 \smallsetminus D_2 \cup \hZ_1} \ar{r} \ar{u} \ar{d} & \fibf_{\widehat D_1}^\affinize \ar{u}[swap,sloped]{=} \ar{d}
\\
  \fibf_{\widehat D_2 \smallsetminus D_2 \cup \hZ_2} \ar{r} & \fibf_{\widehat D_2}^\affinize
\end{tikzcd}\right)$%
\end{lrbox}
\newsavebox{\locimagemorph}
\begin{lrbox}{\locimagemorph}%
$\left(\begin{tikzcd}[cramped,row sep=0, column sep=tiny]
\fibf_{\widehat D_1 \smallsetminus D_1 \cup \hZ_1}
  \ar{dr} \ar{rrrr} & &&&
\fibf_{\widehat D_1' \smallsetminus D_1 \cup \hZ_1'}
  \ar{dr} \ar[from=dd] & &&&
\fibf_{\widehat D_1' \smallsetminus D_1' \cup \hZ_1'}
  \ar{llll} \ar{dr} \ar[from=dd] &
  \\[2pt]
  & \fibf_{\widehat D_1}^\affinize &&& & \ar[crossing over, from={llll}] \fibf_{\widehat D_1'}^\affinize &&& & \fibf_{\widehat D_1'}^\affinize \ar[crossing over]{llll}
\\[8pt]
\fibf_{\widehat D_1 \smallsetminus D_2 \cup \hZ_1} \ar{uu} \ar{dr} \ar{dd} \ar{rrrr} & &&& 
\fibf_{\widehat D_1' \smallsetminus D_2 \cup \hZ_1'} \ar{dr} \ar{dd} & &&& 
\fibf_{\widehat D_1' \smallsetminus D_2' \cup \hZ_1'} \ar{dr} \ar{dd} \ar{llll}
  \\[2pt]
  & \fibf_{\widehat D_1}^\affinize \ar[crossing over]{uu} \ar[crossing over]{rrrr} &&& & \fibf_{\widehat D_1'}^\affinize \ar[crossing over]{uu} &&& & \fibf_{\widehat D_1'}^\affinize \ar{uu} \ar{dd} \ar[crossing over]{llll}
  \\[8pt]
  \fibf_{\widehat D_2 \smallsetminus D_2 \cup \hZ_2} \ar{dr} \ar{rrrr} & &&& \fibf_{\widehat D_2' \smallsetminus D_2 \cup \hZ_2'} \ar{dr} & &&& \fibf_{\widehat D_2' \smallsetminus D_2' \cup \hZ_2'} \ar{llll} \ar{dr}
  \\[2pt]
  & \fibf_{\widehat D_2}^\affinize \ar{rrrr} \ar[from=uu, crossing over] &&& & \fibf_{\widehat D_2'}^\affinize \ar[from=uu, crossing over] &&& & \fibf_{\widehat D_2'}^\affinize \ar{llll}
 \end{tikzcd}%
\right)$
\end{lrbox}
\newsavebox{\locmylaxcorrbox}
\begin{lrbox}{\locmylaxcorrbox}%
$\left(\text{Similar to the case of \eqref{eq:fcup}}\right).$
\end{lrbox}
\begin{tikzcd}[row sep=0pt, column sep=large,
    ,/tikz/column 1/.append style={anchor=base east}
    ,/tikz/column 2/.append style={anchor=base west}]
\closedpairs_{X,\dR}^{\Delta^1}(S) \ar{r}{\fibf^\cup_\loc(S)} & (\dFibF_{X \times S}^\dsq)^\mathrm{cocorr}
\\
\usebox{\locobjbox} \ar[mapsto]{r}{\mathrm{objects}} & \usebox{\locimageobj}
\\
\usebox{\locmorphbox} \ar[mapsto]{r}{\mathrm{morphisms}} & \usebox{\locimagemorph}
\\
\usebox{\locmycomposablebox} \ar[mapsto]{r}{\mathrm{lax}}[swap]{\mathrm{structure}} & \usebox{\locmylaxcorrbox}
\end{tikzcd}
\]

 Similarly to \cref{lem:Ccocorr}, composing $\fibf^\cup_\loc$ with $\BunGX$ yields a natural transformation by lax $\infty$-functors
  \[
   \closedpairs_{X,\dR}^{\Delta^1} \to (\dFibF_{X \times S}^\dsq)^\mathrm{cocorr} \to (((\stackofstacks_X)\op)^\dsq)^\mathrm{cocorr}
  \]
 mapping $(S,\alpha \colon (D_1, Z_1) \to (D_2, Z_2))$ to the commutative diagram
 \[
  \begin{tikzcd}[row sep=small,column sep=small]
   \BunGX_{\widehat D_1 \smallsetminus D_1 \cup \hZ_1} \ar{d} & \BunGX_{\widehat D_1} \ar{l} \ar[equals]{d} \\
   \BunGX_{\widehat D_1 \smallsetminus D_2 \cup \hZ_1} & \BunGX_{\widehat D_1} \ar{l} \\
   \BunGX_{\widehat D_2 \smallsetminus D_2 \cup \hZ_2} \ar{u} & \BunGX_{\widehat D_2}. \ar{l} \ar{u} \\
  \end{tikzcd}
 \]
 Taking the homotopy fibers of the horizontal morphisms over the trivial bundle, and setting $Y := X \times S$ and $\GrX_X^\loc(\alpha) := \BunGX_{\widehat D_1} \times_{\BunGX_{\widehat D_1 \smallsetminus D_2 \cup \hZ_1}} \{\mathrm{trivial}\}$, we get the cocorrespondence in $\dSt_Y$ depicted in \cref{eq:cocorr-for-projection}:
 \[
 \begin{tikzcd}
  \GrX^\loc_X(Y; D_1, Z_1) \ar{r}{\overrightarrow\chi_\alpha} & \GrX_X^\loc(\alpha) & \GrX^\loc_X(Y; D_2, Z_2). \ar{l}[swap]{\overleftarrow\chi_\alpha}
 \end{tikzcd}
 \]
 By functoriality of homotopy fibers, we get a natural transformation by lax $\infty$-functors
 \[
  G \colon \closedpairs_{X,\dR}^{\Delta^1} \to (\dFibF_{X \times S}^\dsq)^\mathrm{cocorr} \to (((\stackofstacks_X)\op)^\dsq)^\mathrm{cocorr} \to (((\stackofstacks_X)^{\Lambda^2_2})\op)^\mathrm{cocorr}
 \]
 Consider now the stack in $(\infty,1)$-categories $(\stackofstacks_X)^{\Delta^1}$. It embeds naturally in the stack in $(\infty,2)$-categories $(((\stackofstacks_X)^{\Lambda^2_2})\op)^\mathrm{cocorr}$, via the assignment 
 \[
  \newsavebox{\objinclobj}
  \begin{lrbox}{\objinclobj}%
  $\left(\begin{tikzcd}[cramped,row sep=small, column sep=small]
   F_1 \ar{r} & F_2
  \end{tikzcd}\right)$%
  \end{lrbox}
  \newsavebox{\objinclmorph}
  \begin{lrbox}{\objinclmorph}%
  $\left(\begin{tikzcd}[cramped,row sep=small, column sep=small]
   F_1 \ar{r} \ar{d} & F_2 \ar{d} \\ F_1' \ar{r} & F_2'
  \end{tikzcd}\right)$%
  \end{lrbox}
  \newsavebox{\objinclobjimg}
  \begin{lrbox}{\objinclobjimg}%
  $\left(\begin{tikzcd}[cramped,row sep=small, column sep=small]
   F_1 \ar{r} & F_2 & F_2 \ar{l}[swap]{=}
  \end{tikzcd}\right)$%
  \end{lrbox}
  \newsavebox{\objinclmorphimg}
  \begin{lrbox}{\objinclmorphimg}%
  $\left(\begin{tikzcd}[cramped,row sep=small, column sep=small]
   F_1' \ar{r} & F_2' & F_2' \ar{l}[swap]{=} \\
   F_1 \ar{r} \ar{d}[sloped]{=} \ar{u} & F_2 \ar{d}[sloped]{=} \ar{u} & F_2 \ar{d}[sloped]{=} \ar{l}[swap]{=} \ar{u} \\
   F_1 \ar{r} & F_2 & F_2 \ar{l}[swap]{=} \\
  \end{tikzcd}\right).$%
  \end{lrbox}
  \begin{tikzcd}[row sep=0pt, column sep=large,
    ,/tikz/column 1/.append style={anchor=base east}
    ,/tikz/column 2/.append style={anchor=base west}]
    \usebox{\objinclobj} \ar[mapsto]{r} & \usebox{\objinclobjimg} \\
    \usebox{\objinclmorph} \ar[mapsto]{r} & \usebox{\objinclmorphimg} 
  \end{tikzcd}
 \]
 
\begin{thm}\label{thm:functloc}
 Denote by $\closedpairs_{X,\dR}^{\mathrm{h},\Delta^1_\mathrm{t}}$ the categorical prestack mapping $S \in \Aff_k$ to the (non-full) sub-category of the category $\closedpairs_{X,\dR}^{\mathrm{h}}(S)^{\Delta^1}$ containing every object, but only morphisms corresponding to commutative squares
 \[
 \begin{tikzcd}[row sep=small, column sep=small]
(D_1, Z_1) \ar{d} \ar{r} & (D_2, Z_2) \ar{d} \\ (D_1', Z_1') \ar{r} & (D_2', Z_2')
\end{tikzcd}
 \]
 such that\footnote{One shows easily with a little point-set topology that this condition is stable under (vertical) composition as well as base change.} $\widebar{D_2 \smallsetminus D_1} \cap D_1' = \widebar{D_2 \smallsetminus D_1} \cap Z_1'$ (as closed topological subspaces of $X \times S$).
 Denote by $G^\mathrm{h}$ the restriction of $G$ to the substack $\closedpairs_{X,\dR}^{\mathrm{h},\Delta^1_\mathrm{t}}$ of $\closedpairs_{X,\dR}^{\Delta^1}$.
  \begin{enumerate}[label={\textrm{(\alph*)}}, ref={\cref{thm:functloc}~\textrm{(\alph*)}}]
   \item For any $S \in \Aff_k$, the lax structure on $G^\mathrm{h}(S)$ is strict, and $G^\mathrm{h}$ is thus an honest morphism of $\infty$-categorical stacks.
   \item\label{thm:functloc:GrX} The morphism $G^\mathrm{h}$ has values in the substack $(\stackofstacks_X)^{\Delta^1}$ of $(((\stackofstacks_X)^{\Lambda^2_2})\op)^\mathrm{cocorr}$, thus inducing a morphism
   \[
    \GrX^\loc_X \colon \closedpairs_{X,\dR}^{\mathrm{h},\Delta^1_\mathrm{t}} \to (\stackofstacks_X)^{\Delta^1}.
   \]
  \end{enumerate}
\end{thm}

\begin{rems}
 \item\label{rem:basechange}The main consequence of the above theorem is the existence of a pullback as announced at the beginning of this section: For any morphism $(D_1, Z_1) \to (D_2, Z_2)$, seen as an object of $\closedpairs_{X,\dR}^{\mathrm{h},\Delta^1}(S)$, we get the pullback morphism $\GrX_X^\loc(X \times S; D_2, Z_2) \to \GrX_X^\loc(X \times S; D_1, Z_1)$ announced in \cref{eq:cocorr-for-projection}. Its functoriality implies the following base change formula:
 \[
 \begin{array}{c}
 \begin{tikzcd}[row sep=small, column sep=small,ampersand replacement=\&]
(D_1, Z_1) \ar{d}{a} \ar{r}{b} \& (D_2, Z_2) \ar{d}{c} \\ (D_1', Z_1') \ar{r}{d} \& (D_2', Z_2')
\end{tikzcd} \\
\text{such that}
\\
\widebar{D_2 \smallsetminus D_1} \cap D_1' = \widebar{D_2 \smallsetminus D_1} \cap Z_1'
\end{array}
\hspace{1em} \leadsto \hspace{1em}
\begin{tikzcd}
 \GrX_X^\loc(X \times S; D_2, Z_2) \ar{r}{b^*} \ar{d}{c_*} \namecell{dr}{\circlearrowright} & \GrX_X^\loc(X \times S; D_1, Z_1) \ar{d}{a_*} \\
 \GrX_X^\loc(X \times S; D_2', Z_2') \ar{r}{d^*} & \GrX_X^\loc(X \times S; D_1', Z_1') \rlap{.}
\end{tikzcd}
 \]
 \item When restricting to degenerated arrows $\closedpairs_{X,\dR}^{\mathrm{h}} \subset \closedpairs_{X,\dR}^{\mathrm{h},\Delta^1}$, we find back the covariant functoriality $\GrX_X^\loc \colon \closedpairs_{X,\dR}^{\mathrm{h}} \to \stackofstacks_X$ without relying on the equivalence with the global Grassmannian.
 \item\label{rem:alsoforPerf3} As in \cref{rem:alsoforPerf,rem:alsoforPerf2}, the above theorem extends mutatis mutandis to a version of the flag Grassmannian for perfect complexes and almost perfect complexes.
\end{rems}

The proof of \cref{thm:functloc} is very similar to that of \cref{thm:functglob}, and will rely on two crucial 
Lemmas \ref{magiclemma} and \ref{magiclemma2}, the proof of which requires some preliminary results.
\vspace{1em}

\paragraph{\texorpdfstring{$\mathcal M$}{M}-coCartesian squares}

\begin{defin}
 A commutative square of $\infty$-categories
 \[
  \begin{tikzcd}
   A \ar{r} \ar{d} & B \ar{d} \\ C \ar{r} & D
  \end{tikzcd}
 \]
 is $1$-Cartesian if the induced $\infty$-functor $A \to B \times_D C$ is fully faithful.
\end{defin}

\begin{lem}\label{lem:pullbacksofncart}
 Consider three commutative squares $(\sigma)$, $(\sigma')$ and $(\sigma'')$ of $\infty$-categories and morphisms $(\sigma) \to (\sigma') \leftarrow (\sigma'')$. Denote by $(\tau)$ the square obtained as the pullback $(\sigma) \times_{(\sigma')} (\sigma'')$. We represent this data as the commutative diagrams
 \[
  \begin{tikzcd}[row sep=0, column sep=small]
   A \ar{rrr} \ar{dd} \ar{dr} &&& A' \ar{dd} \ar{dr} &&& A'' \ar{lll} \ar{dd} \ar{dr} \\[2pt]
   & B \ar[crossing over]{rrr} &&& B' &&& B'' \ar[crossing over]{lll} \ar{dd} \\[8pt]
   C \ar{rrr} \ar{dr} \namecell{ur}{(\sigma)} &&& C' \ar{dr} \namecell{ur}{(\sigma')} &&& C'' \ar{lll} \ar{dr} \namecell{ur}{(\sigma'')} \\[2pt]
   & D \ar{rrr} \ar[from=uu, crossing over] &&& D' \ar[from=uu, crossing over]  &&& D'' \ar{lll}
  \end{tikzcd}
  \hspace{4em}
  \begin{tikzcd}
   A \mathbin{\underset{A'}{\times}} A'' \ar{d} \ar{r} & B \mathbin{\underset{B'}{\times}} B'' \ar{d} \\ C \mathbin{\underset{C'}{\times}} C'' \ar{r} & D \mathbin{\underset{D'}{\times}} D''\rlap{.} \namecell{ul}{(\tau)}
  \end{tikzcd}
 \]
 If both $(\sigma)$ and $(\sigma'')$ are Cartesian, and if $(\sigma')$ is $1$-Cartesian, then $(\tau)$ is Cartesian.
 If $(\sigma)$, $(\sigma')$ and $(\sigma'')$ are all $1$-Cartesian, so is $(\tau)$.
\end{lem}
\begin{proof}
 Denote by $f$ the canonical morphism $f \colon A \to B \times_D C$. We also set $f'$ and $f''$ similarly, and 
 \[
  g \colon A \mathbin{\underset{A'}{\times}} A'' \to B \mathbin{\underset{B'}{\times}} B'' \underset{D \mathbin{\underset{D'}{\times}} D''}{\times} C \mathbin{\underset{C'}{\times}} C''.
 \]
 The result follows from the assumptions together with the trivial observation that $g = f \times_{f'} f''$.
\end{proof}

\begin{defin}\label{defBGcoc}
 Fix a commutative square $(\sigma)$ of derived fiber functors over a base $Y$ and $\mathcal M$ a derived stack over $Y$. Denote by $(\sigma)_{\mathcal M}$ the associated commutative square of derived pre-stacks over $Y$
 \[
  \begin{tikzcd}
   \fibf_1 \ar{r} \ar{d} \namecell{dr}{(\sigma)} & \fibf_2 \ar{d} &\hspace{2em}&
   \mathcal M_{\fibf_4} \ar{r} \ar{d} \namecell{dr}{(\sigma)_{\mathcal M}} & \mathcal M_{\fibf_2} \ar{d} \\
   \fibf_3 \ar{r} & \fibf_4, &&
    \mathcal M_{\fibf_3} \ar{r} & \mathcal M_{\fibf_1}.
  \end{tikzcd}
 \]
 \begin{enumerate}
  \item We say that $(\sigma)$ is $\mathcal M$-coCartesian if $(\sigma)_{\mathcal M}$ is Cartesian (i.e. is pointwise Cartesian).
  \item We say that $(\sigma)$ is $\mathcal M$-$1$-coCartesian if $(\sigma)_{\mathcal M}$ is pointwise $1$-Cartesian.
 \end{enumerate}
\end{defin}

A number of our previous results can be formulated using the notion of $\PerfX$-, $\CohX$- or $\BG$-coCartesian squares (e.g. \cref{prop:localformalglueing}).
We will need this consequence of \ref{sublem:bgperfinj}
\begin{lem}\label{lem:PerfcocAndBGcoc}
 Let $(\sigma)$ be a commutative square in $\FibF_Y$ for some base $Y$:
 \[
  \begin{tikzcd}
   \fibf_1 \ar{r} \ar{d} \namecell{dr}{(\sigma)} & \fibf_2 \ar{d} \\
   \fibf_3 \ar{r} & \fibf_4 \rlap{.}
  \end{tikzcd}
 \]
 If $(\sigma)$ is $\PerfX$-$1$-coCartesian, then $(\sigma)$ is $\BG$-$1$-coCartesian.
\end{lem}
\begin{proof}
Denote by $\fibf$ the pushout $\fibf_2 \amalg_{\fibf_1} \fibf_3$ computed in $\dFibF_{Y}$, and by $F \colon \fibf \to \fibf_4$ the induced morphism.
 By assumption, for any $S$ over $Y$, the $\infty$-functor $\Perf(\fibf_4(S)) \to \Perf(\fibf(S))$ is fully faithful.
 Using \ref{sublem:bgperfinj}, we deduce that $\BunGX_{\fibf_4}(S) = \BunG(\fibf_4(S)) \to \BunG(\fibf(S)) = \BunGX_{\fibf}(S)$ is also fully faithful.
\end{proof}
\vspace{1em}

\paragraph{Two keys lemmas}

After these preliminaries, we are ready to prove the key lemmas allowing for the proof of \cref{thm:functloc}.
\begin{lem}\label{magiclemma}
 Let $S \in \Aff_k$ and $Y := X \times S$.
 Consider a diagram of closed subschemes of $Y$
 \[
  \begin{tikzcd}[row sep=0pt, column sep=tiny]
   & D_1 \ar[hook]{dd} \ar[hook]{rr} && D_2 \\ Z_1' \ar[hook]{dr} \\ & D_1'
  \end{tikzcd}
 \]
 and assume $\widebar{D_2 \smallsetminus D_1} \cap D_1' = \widebar{D_2 \smallsetminus D_1} \cap Z_1'$ (as topological spaces). 
 Then the natural morphism of derived fiber functors
 \[
  \begin{tikzcd}
   \fibf_{\widehat D_1' \smallsetminus D_2 \cup \hZ_1'} \ar{r} & \fibf_{\widehat D_1' \smallsetminus D_1 \cup \hZ_1'}
  \end{tikzcd}
 \]
 induces equivalences of derived stacks over $Y$
 \[
  \begin{tikzcd}[row sep=tiny]
   \PerfX_{\widehat D_1' \smallsetminus D_2 \cup \hZ_1'} & \PerfX_{\widehat D_1' \smallsetminus D_1 \cup \hZ_1'} \ar{l}[swap]{\sim}\rlap{\hspace{1em}and}
   \\
   \BunGX_{\widehat D_1' \smallsetminus D_2 \cup \hZ_1'} & \BunGX_{\widehat D_1' \smallsetminus D_1 \cup \hZ_1'}. \ar{l}[swap]{\sim}
  \end{tikzcd}
 \]
\end{lem}

\begin{proof}
\newcommand{\shortDprime}{\widehat D_{1,T}^{\prime\affinize}}
\newcommand{\shortZprime}{\hZ_{1,T}^{\prime\affinize}}
 For a derived test scheme $T$ over $X \times S$, we fix the notations:
 \[
  \shortDprime := \fibf^\affinize_{\widehat D_1'}(T)
  \hspace{2em}
  \shortZprime := \fibf^\affinize_{\hZ_1'}(T).
 \]
 We also by $D_1^T$, $D_2^T$ and $C$ the closed subschemes of $\shortDprime$ induced respectively by the closed subschemes $D_1$, $D_2$ and\footnote{We endow $\widebar{D_2 \smallsetminus D_1}$ with its reduced scheme structure.} $\widebar{D_2 \smallsetminus D_1}$ of $X \times S$.
 Because $\widebar{D_2 \smallsetminus D_1} \cap D_1' = \widebar{D_2 \smallsetminus D_1} \cap Z_1'$, the formal completion of $\shortDprime$ at $C$ is canonically equivalent to $\shortZprime$.
 Applying \cite[7.4.1.1]{Lurie_SAG} to $A = \R\Gamma(\shortDprime, \mathcal O)$, $B = \R\Gamma(\shortZprime, \mathcal O)$ and $\mathcal C = \QCoh(\shortDprime \smallsetminus D_1^T)$, we deduce after passing to dualizable objects, that the natural square
 \[
  \begin{tikzcd}
   \Perf\left(\shortDprime \smallsetminus D_1^T\right) \ar{r} \ar{d} & \Perf\left(\shortZprime \smallsetminus D_1^T\right) \ar{d} \\
   \Perf\left(\shortDprime \smallsetminus D_2^T\right) \ar{r} & \Perf\left(\shortZprime \smallsetminus D_2^T\right)
  \end{tikzcd}
 \]
 is Cartesian. It follows that the square
 \[
  \begin{tikzcd}
   \fibf_{\hZ_1' \smallsetminus D_2}  \ar{r} \ar{d} & \fibf_{\widehat D_1' \smallsetminus D_2} \ar{d} \\
   \fibf_{\hZ_1' \smallsetminus D_1} \ar{r} & \fibf_{\widehat D_1' \smallsetminus D_1}
  \end{tikzcd}
 \]
 is $\PerfX$-coCartesian.
 By definition of the fiber functors $\fibf_{\widehat D_1' \smallsetminus D_1 \cup \hZ_1'}$ and $\fibf_{\widehat D_1' \smallsetminus D_2 \cup \hZ_1'}$, we get a commutative square
 \[
   \begin{tikzcd}[nodes=overlay, column sep={6em,between origins}, row sep={2.5em,between origins}]
   \fibf_{\hZ_1' \smallsetminus D_2} \ar{rr} \ar{dd} \ar{dr} && \fibf_{\hZ_1' \smallsetminus D_1} \ar{dd} \ar{dr} \\
   & \fibf_{\hZ_1'}^\affinize \ar[crossing over]{rr}[near start]{=} && \fibf_{\hZ_1'}^\affinize \ar{dd} \\
   \fibf_{\widehat D_1' \smallsetminus D_2} \ar{rr} \ar{dr} && \fibf_{\widehat D_1' \smallsetminus D_1} \ar{dr} \\
   & \fibf_{\widehat D_1' \smallsetminus D_2 \cup \hZ_1'} \ar{rr} \ar[from=uu, crossing over] && \fibf_{\widehat D_1' \smallsetminus D_1 \cup \hZ_1'}
  \end{tikzcd}
 \]
 in which 
 \begin{itemize}
  \item both lateral faces are coCartesian in $\dFibF_{X \times S}$ and
  \item the back square is $\PerfX$-coCartesian.
 \end{itemize}
 This implies that the front square is $\PerfX$-coCartesian as well, so that $\PerfX_{\widehat D_1' \smallsetminus D_2 \cup \hZ_1'} \to \PerfX_{\widehat D_1' \smallsetminus D_1 \cup \hZ_1'}$ is an equivalence as claimed.
 The case of $\BunGX$ follows by using \ref{sublem:bgperflocal}.
\end{proof}

\begin{lem}\label{magiclemma2}
 Let $S \in \Aff_k$ and $Y := X \times S$.
 Consider a diagram of closed subschemes of $Y$
 \[
  \begin{tikzcd}[row sep=0pt, column sep=tiny]
   Z_1 \ar[hook]{dr} \ar[hook]{dd} \\ & D_1 \ar[hook]{dd} \ar[hook]{rr} && D_2 \\ Z_1' \ar[hook]{dr} \\ & D_1'
  \end{tikzcd}
 \]
 such that $Z_1 = D_1 \cap Z_1'$ and $\widebar{D_2 \smallsetminus D_1} \cap D_1' = \widebar{D_2 \smallsetminus D_1} \cap Z_1'$ as topological subspaces of $Y$.
 The commutative diagram of fiber functors
 \[
  \begin{tikzcd}
   \fibf_{\widehat D_1 \smallsetminus D_2 \cup \hZ_1} \ar{r} \ar{d} \ar[phantom,start anchor=center,end anchor=center]{dr}{(\sigma)} & \fibf_{\widehat D_1' \smallsetminus D_2 \cup \hZ_1'} \ar{d} \\
   \fibf_{\widehat D_1}^\affinize \ar{r} & \fibf_{\widehat D_1'}^\affinize
  \end{tikzcd}
 \]
 is $\PerfX$-coCartesian (and thus also $\BG$-coCartesian).
\end{lem}
\begin{proof}
  The square $(\sigma)$ fits in the commutative diagram below.\vspace{2pt}\\
  \begin{minipage}{0.35\textwidth}
  \[
    \begin{tikzcd}
    \fibf_{\widehat D_1 \smallsetminus D_1 \cup \hZ_1} \ar{r} \ar{d}{f}
    & \fibf_{\widehat D_1' \smallsetminus D_1 \cup \hZ_1'} \ar{d}{g} \\
    \fibf_{\widehat D_1 \smallsetminus D_2 \cup \hZ_1} \ar{r} \ar{d} \ar[phantom,start anchor=center,end anchor=center]{dr}{(\sigma)} & \fibf_{\widehat D_1' \smallsetminus D_2 \cup \hZ_1'} \ar{d} \\
    \fibf_{\widehat D_1}^\affinize \ar{r} & \fibf_{\widehat D_1'}^\affinize
    \end{tikzcd}
  \]
  \end{minipage}
  \begin{minipage}{0.65\textwidth}
    Using \cref{magiclemma}, we see that $g$ induces an equivalence once composed with $\PerfX$.
    Moreover, using $Z_1 = D_1 \cap Z_1'$ and $\widebar{D_2 \smallsetminus D_1} \cap D_1' = \widebar{D_2 \smallsetminus D_1} \cap Z_1'$, we deduce the equality (of topological subspaces)
    $\widebar{D_2 \smallsetminus D_1} \cap D_1 = \widebar{D_2 \smallsetminus D_1} \cap Z_1$.
    Applying \cref{magiclemma} once more, to the case $D_1' = D_1$, implies that $f$ induces an equivalence as well, once compose with $\PerfX$.
    
    In particular, the square $(\sigma)$ is $\PerfX$-coCartesian if and only if the outer square is. We may thus assume $D_2 = D_1$ without lose of generality.
  \end{minipage}\vspace{2pt}
  
  Under this new assumption, the square $(\sigma)$ decomposes as
 \[
  \begin{tikzcd}
   \fibf_{\widehat D_1 \smallsetminus D_1 \cup \hZ_1} \ar{r} \ar{d} \ar[phantom,start anchor=center,end anchor=center]{dr}{(\tau)}
   & \fibf_{\widehat D_1' \smallsetminus D_1 \cup \hZ_1} \ar{d} \ar{r}{\psi}
   & \fibf_{\widehat D_1' \smallsetminus D_1 \cup \hZ_1'} \ar{d} \\
   \fibf_{\widehat D_1}^\affinize \ar{r}
   & \fibf_{\widehat D_1'}^\affinize \ar{r}{=}
   & \fibf_{\widehat D_1'}^\affinize \rlap{.}
  \end{tikzcd}
 \]
 It thus suffices to prove that $(\tau)$ is $\PerfX$-coCartesian and that $\psi$ induces an equivalence $\PerfX(\psi)$.
 Consider the following commutative diagram.
  \[
   \begin{tikzcd}
   \fibf_{\widehat D_1 \smallsetminus D_1} \ar{r} \ar{d}
   & \fibf_{\widehat D_1' \smallsetminus D_1} \ar{d} \\
   \fibf_{\widehat D_1 \smallsetminus D_1 \cup \hZ_1} \ar{r} \ar{d} \ar[phantom,start anchor=center,end anchor=center]{dr}{(\tau)}
   & \fibf_{\widehat D_1' \smallsetminus D_1 \cup \hZ_1} \ar{d} \\
   \fibf_{\widehat D_1}^\affinize \ar{r}
   & \fibf_{\widehat D_1'}^\affinize \rlap{.}
   \end{tikzcd}
  \]
   For formal reasons, the upper square is coCartesian. Moreover, the outer square is $\PerfX$-coCartesian by \cref{prop:localformalglueing}.
  It follows that $(\tau)$ is $\PerfX$-coCartesian.
  We now focus on the morphism $\psi$.
  We have a commutative diagram
  \[
  \begin{tikzcd}[column sep={6em,between origins},row sep={2.5em,between origins}]
   \fibf_{\hZ_1 \smallsetminus D_1} \ar{rr} \ar{dd} \ar{dr} && \fibf_{\hZ_1' \smallsetminus D_1} \ar{dd} \ar{dr} \\
   & \fibf_{\hZ_1}^\affinize \ar[crossing over]{rr} && \fibf_{\hZ_1'}^\affinize \ar{dd} \\
   \fibf_{\widehat D_1' \smallsetminus D_1} \ar{rr}[near start]{=} \ar{dr} && \fibf_{\widehat D_1' \smallsetminus D_1} \ar{dr} \\
   & \fibf_{\widehat D_1' \smallsetminus D_1 \cup \hZ_1} \ar{rr}[swap]{\psi} \ar[from=uu, crossing over] && \fibf_{\widehat D_1' \smallsetminus D_1 \cup \hZ_1'}
  \end{tikzcd}
  \]
  in which the lateral squares are coCartesian by definition.
  In particular, the morphism $\PerfX(\psi)$ is an equivalence if and only if the upper face is $\PerfX$-coCartesian.
  Since $Z_1 = D_1 \cap Z_1'$, for any derived affine scheme $T$ over $Y = X \times S$, the affinization of the formal completion of $\fibf_{\widehat Z_1'}^\affinize(T)$ at the closed subscheme induced by $D_1 \times_Y T$ is canonically equivalent to $\fibf_{\hZ_1}^\affinize(T)$. In particular, \cref{prop:localformalglueing} implies this upper square is indeed $\PerfX$-coCartesian.
\end{proof}

\paragraph{Proof of \cref{thm:functloc}}

\begin{proof}[Proof of \cref{thm:functloc}]
 Let $S \in \Aff_k$. We start with the lax $\infty$-functor 
 \[
  G^\mathrm{h} \colon \closedpairs_{X,\dR}^{\mathrm{h},\Delta^1_\mathrm{t}}(S) \to (((\stackofstacks_X)^{\Lambda^2_2})\op)^\mathrm{cocorr}(S) = (((\dSt_{X \times S})^{\Lambda^2_2})\op)^\mathrm{cocorr}.
 \]
 Fix a morphism of $\closedpairs_{X,\dR}^{\mathrm{h},\Delta^1_\mathrm{t}}(S)$, corresponding to a commutative diagram
 \[
  \begin{tikzcd}[row sep=small, column sep=small]
   (D_1, Z_1) \ar{d} \ar{r} & (D_2, Z_2) \ar{d} \\ (D_1', Z_1') \ar{r} & (D_2', Z_2').
  \end{tikzcd}
 \]
 such that $\widebar{D_2 \smallsetminus D_1} \cap D_1' = \widebar{D_2 \smallsetminus D_1} \cap Z_1'$. Its image by $\fibf^\cup_\loc(S)$ is the diagram
 \[
  \begin{tikzcd}[cramped,row sep=small, column sep=tiny]
\fibf_{\widehat D_1 \smallsetminus D_1 \cup \hZ_1}
  \ar{dr} \ar{rrrr} \ar[from=dd] \namecell{rrrrrd}{(\alpha)} & &&&
\fibf_{\widehat D_1' \smallsetminus D_1 \cup \hZ_1'}
  \ar{dr} \ar[from=dd] & &&&
\fibf_{\widehat D_1' \smallsetminus D_1' \cup \hZ_1'}
  \ar{llll} \ar{dr} \ar[from=dd] &
  \\
  & \fibf_{\widehat D_1}^\affinize &&& & \ar[crossing over, from={llll}] \fibf_{\widehat D_1'}^\affinize &&& & \fibf_{\widehat D_1'}^\affinize \ar[crossing over]{llll}
\\
\fibf_{\widehat D_1 \smallsetminus D_2 \cup \hZ_1} \ar{dr} \ar{dd} \ar{rrrr} \namecell{rrrrrd}{(\beta)} & &&& 
\fibf_{\widehat D_1' \smallsetminus D_2 \cup \hZ_1'} \ar{dr} \ar{dd} & &&& 
\fibf_{\widehat D_1' \smallsetminus D_2' \cup \hZ_1'} \ar{dr} \ar{dd} \ar{llll}
  \\
  & \fibf_{\widehat D_1}^\affinize \ar[crossing over]{uu} \ar[crossing over]{rrrr} &&& & \fibf_{\widehat D_1'}^\affinize \ar[crossing over]{uu} &&& & \fibf_{\widehat D_1'}^\affinize \ar{uu} \ar{dd} \ar[crossing over]{llll}
  \\
  \fibf_{\widehat D_2 \smallsetminus D_2 \cup \hZ_2} \ar{dr} \ar{rrrr} \namecell{rrrrrd}{(\gamma)} & &&& \fibf_{\widehat D_2' \smallsetminus D_2 \cup \hZ_2'} \ar{dr} & &&& \fibf_{\widehat D_2' \smallsetminus D_2' \cup \hZ_2'} \ar{llll} \ar{dr}
  \\
  & \fibf_{\widehat D_2}^\affinize \ar{rrrr} \ar[from=uu, crossing over] &&& & \fibf_{\widehat D_2'}^\affinize \ar[from=uu, crossing over] &&& & \fibf_{\widehat D_2'}^\affinize \ar{llll} \rlap{.}
 \end{tikzcd}
 \]
 By \cref{magiclemma2}, the three squares $(\alpha)$, $(\beta)$ and $(\gamma)$ are $\BG$-coCartesian.
 As a consequence, the image of any morphism of $\closedpairs_{X,\dR}^{\mathrm{h},\Delta^1_\mathrm{t}}(S)$ is a cocorrespondence (of $\Lambda^2_2$-diagrams) where the morphism from ``source to hat'' is an equivalence.
 It follows that the lax $\infty$-functor $G^\mathrm{h}(S)$ is in fact a strict $\infty$-functor with values in the (non-full) sub-$(\infty,1)$-category
 \[
  (\dSt_{X \times S})^{\Lambda^2_2} \subset (((\dSt_{X \times S})^{\Lambda^2_2})\op)^\mathrm{cocorr}.
 \]
 Consider now an object $(D_1, Z_1) \to (D_2, Z_2)$ of $\closedpairs_{X,\dR}^{\mathrm{h},\Delta^1_\mathrm{t}}(S)$. Its image by $G^\mathrm{h}(S)$ is the diagram (depicted in \cref{eq:cocorr-for-projection})
 \[
\begin{tikzcd}
 \makebox[0pt][r]{$\displaystyle \GrX^\loc_X(Y; D_2, Z_2) := {}$}
 \displaystyle \BunGX_{\widehat D_2} \newtimes_{\BunGX_{\widehat D_2 \smallsetminus D_2 \cup \hZ_2}} \{\mathrm{trivial}\} \ar{d}
 \\
 \displaystyle \BunGX_{\widehat D_1} \newtimes_{\BunGX_{\widehat D_1 \smallsetminus D_2 \cup \hZ_1}} \{\mathrm{trivial}\}
 \\
 \displaystyle \BunGX_{\widehat D_1} \newtimes_{\BunGX_{\widehat D_1 \smallsetminus D_1 \cup \hZ_1}} \{\mathrm{trivial}\} \ar{u}{u}
 \makebox[0pt][l]{$\displaystyle {} =: \GrX^\loc_X(Y; D_1, Z_1).$}
\end{tikzcd}
 \]
 \cref{magiclemma} implies that $u$ is an equivalence. In particular, the functor $G^\mathrm{h}(S)$ has values in the full subcategory $\dSt_{X \times S} \subset (\dSt_{X\times S})^{\Lambda^2_2}$. This concludes the proof of \cref{thm:functloc}.
\end{proof}

\section{The Flag Grassmannian as a \texorpdfstring{$2$}{2}-Segal object}\label{section:Grass2Segal}
By Beilinson and Drinfeld, the classical affine Grassmannian (in dimension 1) carries an important factorization structure (see e.g. \cite{Zhu2017}). In essence, it consists of the way a $\mathbf G$-bundle on a curve trivialized outside two points is tantamount to two bundles, each one trivialized outside one of the points.

In this section, we will construct a similar factorization structure, where disjoint points are to be replaced with pairs of flags in good position.
Part of the difficulties arising in this construction is to handle the higher coherences and homotopies the structure involves.

\subsection{The projection maps \texorpdfstring{$\GrX_X(S)(F_1 \cup F_2) \to \GrX_X(S)(F_i)$}{Grₓ(S)(F₁∪F₂) → Grₓ(S)(Fᵢ)}}\label{mappediproiezione}
We will construct a $2$-Segal structure on the flag Grassmannian. The first step is to define functorial projection maps $\GrX_X(S)(F_1 \cup F_2) \to \GrX_X(S)(F_i)$, $i=1,2$, for any pair $(F_1, F_2) \in \flags_{X,2}(S_\red)$.
This starts with the simple observation: the natural inclusions $F_i \to F_1 \cup F_2$ are morphism in $\flags_X(S_\red)$.

In particular, we get natural transformations
 \begin{equation}\label{eq:nattransf}
  \begin{tikzcd}
   \flags_{X,2}
     \ar[r, bend left=50, ""{name=U, below}, "\partial_0"{above}]
     \ar[r, bend right=50, ""{name=D, above}, "\partial_2"{below}]
     \ar[r, "\partial_1"{name=M, description}]
     \ar[from=U-|M, to=M, Rightarrow]
     \ar[from=D-|M, to=M, Rightarrow]
   & \flags_{X,1} \makebox[0pt][l]{${} = \flags_X$.}
  \end{tikzcd}
 \end{equation}
 The natural transformations $\partial_0 \Rightarrow \partial_1$ and $\partial_2 \Rightarrow \partial_1$ each correspond to a morphism $\flags_{X,2} \to \flags_X^{\Delta^1}$ to the stack of arrows of flags.
 Taking the de Rham stack on each side, we get:
\[
 \begin{tikzcd}
 \flags_{X,2}^\dR \ar{r} & \flags_{X,\dR}^{\Delta^1}.
 \end{tikzcd}
\]
\begin{lem}
 The inclusion $\flags_X^\dR \to \closedpairs_{X,\dR}$ induces an inclusion $\flags_{X,\dR}^{\Delta^1} \to 
 \closedpairs_{X,\dR}^{\mathrm{h},\Delta^1_t}$ (recall the notation from \cref{thm:functloc}).
\end{lem}
\begin{proof}
 This is straightforward point-set topology.
\end{proof}

Using the morphism $\GrX^\loc_X$ from \ref{thm:functloc:GrX}, we get
\[
 \begin{tikzcd}
 \flags_{X,2}^\dR \ar{r} & \flags_{X,\dR}^{\Delta^1} \ar{r}{\GrX^\loc_X} & (\stackofstacks_X)^{\Delta^1}.
 \end{tikzcd}
\]
Remembering this morphism swaps source and target, we thus get natural transformations
\begin{equation}\label{projnattrans}
  \begin{tikzcd}[column sep=6em]
   \flags_{X,2}^\dR
     \ar[r, "\GrX^\loc_X \circ \partial_1"{name=M, description}]
     \ar[r, bend left=50, ""{name=U, below}, "\GrX^\loc_X \circ \partial_0"{above}]
     \ar[r, bend right=50, ""{name=D, above}, "\GrX^\loc_X \circ \partial_2"{below}]
     \ar[from=M, to=U-|M, Rightarrow]
     \ar[from=M, to=D-|M, Rightarrow]
   & \stackofstacks_X.
  \end{tikzcd}
\end{equation}
Evaluating at $S \in \Aff_k$ and at a pair of flags $(F_1, F_2) \in \flags_{X,2}^\dR(S) := \flags_{X,2}(S_\red)$, we get the announced \emph{functorial} projection morphisms (in $\dSt_{X \times S}$):
\[
 \GrX_X^\loc(S)(F_1 \cup F_2) \to \GrX_X^\loc(S)(F_i).
\]

\subsection{Factorization property}\label{sectionfactorization}

The goal of this section is to prove the following factorization formula
\begin{thm}[Factorization]\label{thm:factorization}
  The natural transformations from \eqref{projnattrans} induce an equivalence
  \[
  \begin{tikzcd}
   \GrX^\loc_X \circ \partial_1 \ar{r}{\sim} & \GrX^\loc_X \circ \partial_0 \times \GrX^\loc_X \circ \partial_2.
  \end{tikzcd}
  \]
\end{thm}

\begin{rems}\label{remarkfactorization}
 \item \cref{thm:factorization} can be seen as a factorization structure on our flag Grassmannian. Indeed, after evaluating at a Noetherian derived affine scheme $S$ and at a good pair of flags $(F_1, F_2)$ in $X \times S_\red$, the statement becomes that the constructed map
\[
\begin{tikzcd}
 \GrX^\loc_X(S)(F_1 \cup F_2) \ar{r}{\sim} & \GrX^\loc_X(S)(F_1) \times \GrX^\loc_X(S)(F_2)
\end{tikzcd}
\]
is an equivalence.
\item Similarly to Remarks \ref{rem:alsoforPerf}, \ref{rem:alsoforPerf2} and \ref{rem:alsoforPerf3}, the factorization property of \cref{thm:factorization} extends to flag Grassmannians of perfect complexes, and the proof provided below also applies in this case. 
\end{rems}

The proof of \cref{thm:factorization} has prerequisites that we will address first. It is thus postponed until the end of this section.

 Let $A \in \sCAlg_k$ be a simplicial commutative algebra over $k$. For $f \in \pi_0 A$, denote by $A^f$ the (homotopy) pushout
 \begin{equation}\label{eq:Af}
  \begin{tikzcd}
   A[x] \ar{r}{x \,\mapsto f} \ar{d}[swap,sloped]{\rotatebox{90}{$x$}\,\mapsto\, \rotatebox{90}{$0$}}
   \ar[phantom]{dr}[description, near end]{\lrcorner}
   & A \ar{d} \\ A \ar{r} & A^f.
  \end{tikzcd}
 \end{equation}
For any family $f_\bullet = (f_1, \dots, f_n) $ of elements of $ \pi_0 A$, we denote by $A^{f_\bullet}$ or $A^{f_1, \dots, f_n}$ the tensor product $\bigotimes_A A^{f_i}$.
 Note that $A^f$ is equivalent to the mapping cone of the morphism $A \to A$ given by the multiplication $f$. More generally $A^{f_\bullet}$ is nothing but the Koszul complex associated to the sequence $(f_1, \dots, f_n)$.
 
Moreover, if $f_\bullet = ( f_1, \dots, f_n ) \subset \pi_0 A$ and $p$ is a positive integer, we will use the notation $f_\bullet^p = ( f_1^p, \dots, f_n^p )$.
Notice that if $f,g \in \pi_0 A$, there is a natural map $A^{fg} \to A^{f}$ (induced, in terms of pushout squares as in \eqref{eq:Af}, by $x \mapsto gx$). In particular, for any $p \geq q$ and any family $f_\bullet$ as above, we get a canonical morphism $A^{f^p_\bullet} \to A^{f^q_\bullet}$.
The pro-diagram $(A^{f_\bullet^p})_p$ corresponds, by \cite[Prop. 6.7.4]{Gaitsgory-Rozenblyum:dgindschemes}, to an explicit derived affine ind-scheme computing the formal neighbourhood of $\{f_\bullet = 0\}$ in $\Spec A$.

Given two such families $f_\bullet = (f_1,\dots, f_n)$ and $g_\bullet=(g_1,\dots,g_m)$, we denote by $f_\bullet g_\bullet$ the family in $\pi_0 A$ consisting of all possible products between an element of $f_\bullet$ and an element of $g_\bullet$, while $f_\bullet \cup g_\bullet$ denotes $(f_1,\dots,f_n,g_1,\dots,g_m)$.

\begin{lem}\label{lem:fiberprodKoszul}
	For any $A \in \sCAlg_k$ and for any families $f_\bullet = (f_1, \dots, f_n)$ and $g_\bullet = (g_1, \dots, g_m)$ of elements of $\pi_0 A$, there is a natural equivalence
	\[
	\lim_p A^{f_\bullet^p g_\bullet^p} \simeq \lim_p \left( A^{f_\bullet^p} \newtimes_{ A^{f_\bullet^p} \otimes_A A^{g_\bullet^p}} A^{g_\bullet^p} \right).
	\]
\end{lem}

\begin{sublem}\label{sublem:doublequotients}
  Let $f, g \in \pi_0 A$. There is a natural equivalence $A^{f,g} \simeq (A^f)^g$ (where in the right hand side, $g$ is seen as an element in $\pi_0(A^f)$).
\end{sublem}

\begin{proof}
 It follows from contemplating the pushout diagrams
\[
  \begin{tikzcd}
   A[x,y] \ar{r}{x \,\mapsto f} \ar{d}[swap,sloped]{\rotatebox{90}{$x$}\,\mapsto\, \rotatebox{90}{$0$}}
   & A[y] \ar{d} \ar{r}{y \,\mapsto g} & A \ar{d} \\
   A[y] \ar{r} \ar{d}[swap,sloped]{\rotatebox{90}{$y$}\,\mapsto\, \rotatebox{90}{$0$}} & A^f[y] \ar{r}{y \,\mapsto g} \ar{d}[swap,sloped]{\rotatebox{90}{$y$}\,\mapsto\, \rotatebox{90}{$0$}} & A^f \ar{d} \\
   A \ar{r} & A^f \ar{r} & A^{f,g}.
  \end{tikzcd}
\]
\end{proof}

\begin{sublem}\label{sublem:fiberproductKoszulbabycase}
 Let $f, g \in \pi_0 A$. There is a natural equivalence
 \begin{equation}\label{eq:fiberproductKoszulbabycase}
  A^{fg} \simeq A^f \newtimes_{A^{f,g}} A^g. 
 \end{equation}
\end{sublem}

\begin{proof}
  To prove this sublemma, we will rely on explicit models in the (equivalent model) category of connective commutative dg-algebras. We thus assume that $A$ is now an explicit cdga $\cdots \to A^{-1} \to A^0 \to 0$, and we fix lifts of $f$ and $g$ to $A^0$. We will abusively denote those lifts by $f$ and $g$ as well.
	
	Denote by $x$ and $y$ generators of degree $-1$, whose images by the differential are $f$ and $g$ respectively. Then the pullback on the right hand side of \eqref{eq:fiberproductKoszulbabycase} is the (homotopy) pullback
\begin{equation}\label{eq:pullbackofKoszul} A[x] \newtimes_{A[x,y]} A[y]. \end{equation}
In order to compute this homotopy pullback, it suffices to resolve $A[x] \to A[x,y]$ with a fibration.
For example, consider the algebra $A[x,\bar{x}, d\bar{x}]$, where the added generators $\bar{x}, d\bar{x}$ are in degree $-1$ and $0$ respectively, and the differential sends $\bar{x}$ to $d\bar{x}$. It is straightforward to check that the assignments
\[ x \mapsto x, \ \ \  \bar{x} \mapsto y, \ \ \  d\bar{x} \mapsto g \]
define a fibration $A[x,\bar{x}, d\bar{x}] \to A[x,y]$, replacing $A[x] \to A[x,y]$.

It follows that an explicit model for the homotopy pullback \eqref{eq:pullbackofKoszul} is given by the strict pullback
\[ B = A[x,\bar{x}, d\bar{x}]  \newtimes_{A[x,y]} A[y]. \]
Explicitely, the elements of $B$ are polynomials in $A[x,\bar{x}, d\bar{x}] $ such that, if evaluated in $\bar{x}=y, d\bar{x}=g$, reduce to polynomials involving $y$ as unique variable. Said differently, the elements of $B$ are the polynomials in $A[x,\bar{x}, d\bar{x}] $ of the form
\[ x (g-d\bar{x}) q_1(\bar{x}, d\bar{x}) + q_2 (\bar{x}, d\bar{x}) \]
where $q_1,q_2 \in A[\bar{x}, d\bar{x}] $.
We now claim that $B$ is quasi isomorphic to $A[z]$, where $z$ is a degree $-1$ generator with differential $fg$. In fact, setting $t=x(g-d\bar{x})+f\bar{x}$, we have that $B=A[t,\bar{x},d\bar{x}]$, where now the differential of the generator $t$ is $fg$. It is now easy to verify that the natural map $A[z] \to B$ sending $z$ to $t$ is a quasi-isomorphism.

Noticing that $A[z]$ is a model for $A^{fg}$ concludes the proof of the sublemma.
\end{proof}

\begin{proof}[Proof of \cref{lem:fiberprodKoszul}]
We proceed by induction on $m + n$. The case $m + n = 2$ is either trivial (if either $m$ or $n$ vanish) or follows directly from \cref{sublem:fiberproductKoszulbabycase} in the case $m = n = 1$.

	We now prove the inductive step. Let $f_\bullet = (f_1, \dots, f_n)$ and $g_\bullet = (g_1, \dots, g_m)$ be families of elements of $H^0( A)$. Suppose without loss of generality that $m>1$, and denote by $\bar{g}_\bullet$ the family $(g_1, \dots, g_{m-1})$. Then \cref{sublem:doublequotients} yields
	\[ \lim_p A^{f_\bullet^p q_\bullet^p } \simeq \lim_p \left( A^{(f_\bullet^pg_m^p)} \right)^{(f_\bullet^p \bar{g}_\bullet^p)}.
	\]
	By finality of the diagonal, the last limit can also be written as
	\[ \lim_p \left( A^{(f_\bullet^pg_m^p)} \right)^{(f_\bullet^p \bar{g}_\bullet^p)} \simeq \lim_{p,q} \left( A^{(f_\bullet^q g_m^q)} \right)^{(f_\bullet^p \bar{g}_\bullet^p)}.\]
	Setting $B_q =  A^{f_\bullet^q g_m^q}$, we can now use the inductive hypothesis to obtain an equivalence
	\[ \lim_p A^{f_\bullet^p q_\bullet^p } \simeq \lim_{p,q} \left( B_q \right)^{f_\bullet^p \bar{g}_\bullet^p} \simeq \lim_{p,q}   \left( B_q^{f_\bullet^p} \newtimes_{ B_q^{f_\bullet^p} \otimes_{B_q} B_q^{\bar{g}_\bullet^p}} B_q^{\bar{g}_\bullet^p} \right) \simeq  \lim_{p}   \left( B_p^{f_\bullet^p} \newtimes_{ B_p^{f_\bullet^p} \otimes_{B_p} B_p^{\bar{g}_\bullet^p}} B_p^{\bar{g}_\bullet^p} \right) \]
	where in the last identification we used once again a finality argument. In other words, invoking \cref{sublem:doublequotients}, we obtained a Cartesian square
	\begin{equation}\label{eq:pullbackofproalg}
	\begin{tikzcd}
	\displaystyle \lim_p A^{f_\bullet^p q_\bullet^p } \ar{r} \ar{d} & \displaystyle \lim_p A^{f_\bullet^p \cup (f_\bullet^pg_m^p)} \ar{d} \\
	\displaystyle \lim_p A^{\bar{g}_\bullet^p \cup (f_\bullet^pg_m^p)} \ar{r} & \displaystyle \lim_p A^{f_\bullet^p \cup \bar{g}_\bullet^p \cup (f_\bullet^pg_m^p)}.
	\end{tikzcd}
	\end{equation}
	Notice however \cite[Prop. 6.7.4]{Gaitsgory-Rozenblyum:dgindschemes} implies that if $h_\bullet$ is a family in $\pi_0(A)$, then $\lim_p A^{h_\bullet^p}$ only depends on the ideal of $\pi_0(A)$ generated by the elements of $h_\bullet$. Therefore, the Cartesian square \eqref{eq:pullbackofproalg} can be equivalently rewritten as
	\begin{equation}\label{eq:pullbackofproalg2}
	\begin{tikzcd}
	\displaystyle \lim_p A^{f_\bullet^p q_\bullet^p } \ar{r} \ar{d} & \displaystyle \lim_p A^{f_\bullet^p } \ar{d} \\
	\displaystyle \lim_p A^{\bar{g}_\bullet^p \cup (f_\bullet^pg_m^p)} \ar{r} & \displaystyle \lim_p A^{f_\bullet^p \cup \bar{g}_\bullet^p}.
	\end{tikzcd}
	\end{equation}
	But repeating the same arguments as before, we obtain that the bottom map of diagram \eqref{eq:pullbackofproalg2} fits in a Cartesian square
	\begin{equation}\label{eq:pullbackofproalg3}
	\begin{tikzcd}
	\displaystyle \lim_p A^{\bar{g}_\bullet^p \cup (f_\bullet^pg_m^p)} \ar{r} \ar{d} & \displaystyle \lim_p A^{f_\bullet^p \cup \bar{g}_\bullet^p} \ar{d} \\ 
	\displaystyle \lim_p A^{g_\bullet^p} \ar{r} & \displaystyle \lim_p A^{f_\bullet^p \cup g_\bullet^p}
	\end{tikzcd}
	\end{equation}
	and composing the two cartesian squares \eqref{eq:pullbackofproalg2} and \eqref{eq:pullbackofproalg3} yields the desired statement.
\end{proof}

\begin{lem}\label{lem:glueperfformal}
 Let $Y$ be a Noetherian derived affine scheme and $Z_1, Z_2 \subset Y$ be closed subschemes. Let $Z_{12} := Z_1 \cap Z_2$ and $Z = Z_1 \cup Z_2$.
 The canonical functor
 \[
  \Perf_{\widehat Z} \to \Perf_{\widehat Z_1} \newtimes_{\Perf_{\widehat Z_{12}}} \Perf_{\widehat Z_2}
 \]
 is an equivalence.
\end{lem}

\begin{proof}
	Let $A$ denote (the homotopy type of) the simplicial algebra of functions on $Y$.
	Since $A$ is Noetherian, we can and do fix generators $f_\bullet = (f_1, \dots, f_n)$ and $g_\bullet=(g_1, \dots, g_m)$ of the ideals $I$ and $J$ of $\pi_0 A$ of functions vanishing on $Z_1$ and $Z_2$ respectively.
	
	For any positive integer $p$, we write $f_\bullet^p$ for the family $(f_1^p, \dots, f_n^p)$ (and similarly for $g_\bullet^p$).
	By \cite[Prop. 6.7.4]{Gaitsgory-Rozenblyum:dgindschemes}, we have
	\[
	\widehat Z_1 \simeq \colim_p \Spec(A^{f_\bullet^p}),
	\hspace{2em}
	\widehat Z_2 \simeq \colim_p \Spec(A^{g_\bullet^p})
	\hspace{2em}\text{and} \hspace{2em} 
	\widehat Z_{12} \simeq \colim_p \Spec\left(A^{f_\bullet^p} \otimes_A A^{g_\bullet^p}\right).
	\]
	Notice that for every $p$ the morphisms 
	\[ A^{f^p_\bullet} \to A^{f_\bullet^p} \otimes_A A^{g_\bullet^p}, \ \ \ \ \ A^{g^p_\bullet} \to A^{f_\bullet^p} \otimes_A A^{g_\bullet^p}\]
	are surjective on $H^0$. Therefore, it follows from Theorems 16.2.0.1 and 16.2.3.1 of \cite{Lurie_SAG} that, for every positive integer $p$, the square
	\[  \begin{tikzcd}
	\Perf_{A^{f_\bullet^p} \newtimes_{ A^{f_\bullet^p} \otimes_A A^{g_\bullet^p}} A^{g_\bullet^p}} \ar{r} \ar{d} & \Perf_{A^{f_\bullet^p}} \ar{d} \\ 
	\Perf_{A^{f_\bullet^p}} \ar{r} & \Perf_{A^{f_\bullet^p} \otimes_A A^{g_\bullet^p}}
	\end{tikzcd}
	\]
	is Cartesian. Taking the limit over $p$ yields an equivalence
	\[  \lim_p  \Perf_{A^{f_\bullet^p} \newtimes_{ A^{f_\bullet^p} \otimes_A A^{g_\bullet^p}} A^{g_\bullet^p}}  \simeq \Perf_{\widehat Z_1} \newtimes_{\Perf_{\widehat Z_{12}}} \Perf_{\widehat Z_2}\]
	and algebraization implies that \[  \lim_p  \Perf_{A^{f_\bullet^p} \newtimes_{ A^{f_\bullet^p} \otimes_A A^{g_\bullet^p}} A^{g_\bullet^p}}  \simeq  \Perf_{\lim_{p} A^{f_\bullet^p} \newtimes_{ A^{f_\bullet^p} \otimes_A A^{g_\bullet^p}} A^{g_\bullet^p}}. \] 
	Using Lemma \ref{lem:fiberprodKoszul}, we obtain therefore that
	\begin{equation}\label{eq:fiberproductofPerf} \Perf_{\lim_{p} A^{f^p_\bullet g_\bullet^p}} \simeq \Perf_{\widehat Z_1} \newtimes_{\Perf_{\widehat Z_{12}}} \Perf_{\widehat Z_2}. \end{equation}
	
	In virtue of \cite[Prop 6.7.4]{Gaitsgory-Rozenblyum:dgindschemes}, the limit $\lim_p A^{f^p_\bullet g_\bullet^p}$ only depends on the radical of the ideal of $H^0(A)$ generated by the family $f_\bullet g_\bullet$. But the elements of the family $f_\bullet g_\bullet$ generate the ideal $IJ$, whose radical coincides with $I \cap J$. Therefore,
	\[ \widehat{Z} \simeq \colim_p \Spec(A^{f_\bullet^p g_\bullet^p}) \]
	which, together with equivalence \eqref{eq:fiberproductofPerf}, concludes the proof.
\end{proof}

From now on and up until the end of this section, we fix a test Noetherian affine scheme $S$, and $(F_1, F_2)$ a good pair of flags in $X \times S$. We write $Y = X \times S$ and $F_i = (D_i, Z_i)$.

\begin{lem}\label{lem:BGcocD12Z12}
 The following squares of fiber functors are both $\PerfX$-coCartesian and $\BG$-coCartesian (see \cref{defBGcoc}):
 \[
  \begin{tikzcd}
   \fibf_{\widehat{D_1 \cap D_2}}^\affinize \ar{r} \ar{d} \namecell{dr}{(\sigma_D)} & \fibf_{\widehat D_1}^\affinize \ar{d} &&
   \fibf_{\widehat{Z_1 \cap Z_2}}^\affinize \ar{r} \ar{d} \namecell{dr}{(\sigma_Z)} & \fibf_{\hZ_1}^\affinize \ar{d} \\
   \fibf_{\widehat D_2}^\affinize \ar{r} & \fibf_{\widehat{D_1 + D_2}}^\affinize &&
   \fibf_{\hZ_2}^\affinize \ar{r} & \fibf_{\widehat{Z_1 + Z_2}}^\affinize.
  \end{tikzcd}
 \]
\end{lem}
\begin{proof}
 After evaluating the fiber functors involved on a derived affine scheme, this becomes a straightforward application of \cref{lem:glueperfformal} and \ref{sublem:bgperflocal}.
\end{proof}

\begin{lem}\label{lem:BGcocD12Z12punctured}
 The following commutative squares of fiber functors are $\PerfX$-$1$-coCartesian and $\BG$-$1$-coCartesian (see \cref{defBGcoc}):
 \[
  \begin{tikzcd}
   \fibf_{\widehat{D_1 \cap D_2} \smallsetminus D_1 + D_2} \ar{r} \ar{d} \namecell{dr}{(\sigma_D^\circ)} & \fibf_{\widehat D_1 \smallsetminus D_1 + D_2} \ar{d} &&
   \fibf_{\widehat{Z_1 \cap Z_2} \smallsetminus D_1 + D_2} \ar{r} \ar{d} \namecell{dr}{(\sigma_Z^\circ)} & \fibf_{\hZ_1 \smallsetminus D_1 + D_2} \ar{d} \\
   \fibf_{\widehat D_2 \smallsetminus D_1 + D_2} \ar{r} & \fibf_{\widehat{D_1 + D_2} \smallsetminus D_1 + D_2}  &&
   \fibf_{\hZ_2 \smallsetminus D_1 + D_2} \ar{r} & \fibf_{\widehat{Z_1 + Z_2} \smallsetminus D_1 + D_2}.
  \end{tikzcd}
 \]
\end{lem}

\begin{proof}
 We fix a test scheme $T \in \Aff_Y$. We denote by $A$ the ring of functions of the affine ind-scheme $\fibf_{\widehat{D_1 + D_2}}(T)$, by $A_i$ that of $\fibf_{\widehat{D_i}}(T)$ for $i = 1,2$ and by $A_{12}$ that of $\fibf_{\widehat{D_1 \cap D_2}}(T)$.
 By \cref{lem:glueperfformal}, we have $A \simeq A_1 \times_{A_{12}} A_2$.
 Denote by $U$ the open subscheme $U := \fibf_{\widehat{D_1 + D_2} \smallsetminus D_1 + D_2}(T)$ of $\fibf_{\widehat{D_1 + D_2}}^\affinize(T) = \Spec(A)$.
 Finally, we denote by $U_i$ the intersection $U_i := U \times_{\Spec A} \Spec A_i$, for $i = 1, 2$ or $12$.
 Because $U \to \Spec A$ is flat, we get a (homotopy) pullback square
 \[
  \begin{tikzcd}
   \cO_U \ar{r} \ar{d} & \cO_{U_1} \ar{d} \\ \cO_{U_2} \ar{r} & \cO_{U_{12}}.
  \end{tikzcd}
 \]
 in the derived category of $U$. This implies that the canonical functor
 \[
 \begin{tikzcd}
  \Perf(U) \ar[hook]{r} & \Perf(U_1) \underset{\Perf(U_{12})}{\times} \Perf(U_2)
 \end{tikzcd}
 \]
 is fully faithful. By \cref{lem:PerfcocAndBGcoc}, we deduce that $(\sigma_D^\circ)$ is $\BG$-$1$-coCartesian.
 The case of $(\sigma^\circ_Z)$ is isomorphic, up to replacing $A$ with the ring of functions on $\widehat{Z_1 \cup Z_2}$, and changing $U$ accordingly.
\end{proof}

\begin{cor}\label{cor:BGcoctrivializations}
 The following commutative square of fiber functors is $\PerfX$-$1$-coCartesian (and thus $\BG$-$1$-coCartesian as well):
 \[
  \begin{tikzcd}
   \fibf_{\widehat{Z_1 \cap Z_2}}^\affinize \ar{r} \ar{d} \namecell{dr}{(\tau)} & \fibf_{\widehat D_1 \smallsetminus D_1 + D_2 \cup \hZ_1} \ar{d} \\
   \fibf_{\widehat D_2 \smallsetminus D_1 + D_2 \cup \hZ_2} \ar{r} & \fibf_{\widehat{D_1 + D_2} \smallsetminus D_1 + D_2 \cup \widehat{Z_1 \cup Z_2}}
  \end{tikzcd}
 \]
\end{cor}

\begin{proof}
 The commutative squares of fiber functors from \cref{lem:BGcocD12Z12} and \cref{lem:BGcocD12Z12punctured} are related by canonical morphisms $(\sigma_Z) \leftarrow (\sigma_Z^\circ) \to (\sigma_D^\circ)$. The pushout of this diagram of squares is nothing but the square $(\tau)$.
 Applying \cref{lem:pullbacksofncart} (pointwise) to the induced diagram $(\sigma_Z)_{\PerfX} \to (\sigma_Z^\circ)_{\PerfX} \leftarrow (\sigma_D^\circ)_{\PerfX}$, using \cref{lem:BGcocD12Z12} and \cref{lem:BGcocD12Z12punctured}, implies that $(\tau)_{\PerfX}$ is $1$-Cartesian as announced.
\end{proof}

We can now conclude our proof of \cref{thm:factorization}
\begin{proof}[Proof of \cref{thm:factorization}]
 Consider the morphisms of commutative squares $(\sigma_D)_{\BG} \to (\tau)_{\BG} \leftarrow (*)$, where $(*)$ is the constant square with value $*$, and the map $(*) \to (\tau)_{\BG}$ selects the trivial bundles.
 Denote by $(\alpha)_{\textrm{Gr}}$ the pullback $(\sigma_D)_{\BG} \times_{(\tau)_{\BG}} (*)$.
 Explicitly, $(\alpha)_{\textrm{Gr}}$ is the upper square in the commutative diagram
{\tiny \[
  \begin{tikzcd}[row sep={3em,between origins},column sep={10em, between origins}]
   \mathrm{Gr}_X(S)(F_1 \cup F_2) \ar{rr} \ar{rd} \ar{dd} && \mathrm{Gr}_X(S)(F_1) \ar{dr} \ar{dd}
   & \arrow[d, start anchor={[xshift=3em, yshift=0.5em]center}, end anchor={[xshift=3em, yshift=-0.5em]center}, no head, decorate, decoration={brace}, "\displaystyle(\alpha)_\mathrm{Gr}" right=3pt]
   \\
   & \mathrm{Gr}_X(S)(F_2) \ar[crossing over]{rr} && \hspace{1em}*\hspace{1em} \ar{dd} \\
   \Bun_{\widehat{D_1 + D_2}} \ar{rr} \ar{rd} \ar{dd} && \Bun_{\widehat D_1} \ar{dr} \ar{dd}
   & \arrow[d, start anchor={[xshift=3em, yshift=0.5em]center}, end anchor={[xshift=3em, yshift=-0.5em]center}, no head, decorate, decoration={brace}, "\displaystyle(\sigma_D)_{\BG}" right=3pt] \\
   & \Bun_{\widehat D_2} \ar[crossing over]{rr} \ar[from=uu,crossing over]  && \Bun_{\widehat{D_1 \cap D_2}} \ar{dd} \\
   \Bun_{\widehat{D_1 + D_2} \smallsetminus D_1 + D_2 \cup \widehat{Z_1 \cup Z_2}} \ar{rr} \ar{dr} && \Bun_{\widehat D_1 \smallsetminus D_1 + D_2 \cup \hZ_1} \ar{dr} 
   & \arrow[d, start anchor={[xshift=3em, yshift=0.5em]center}, end anchor={[xshift=3em, yshift=-0.5em]center}, no head, decorate, decoration={brace}, "\displaystyle(\tau)_{\BG}" right=3pt] \\
   & \Bun_{\widehat D_2 \smallsetminus D_1 + D_2 \cup \hZ_2} \ar{rr} \ar[from=uu,crossing over]  && \Bun_{\widehat{Z_1 \cap Z_2}}
  \end{tikzcd}
 \]}
 where the four vertical sequences are fiber sequences. Indeed, the leftmost sequence is a fiber sequence by \cref{lem:benjexo}; the two middle ones by \cref{magiclemma}. As for the rightmost sequence, it is a fiber sequence since the pair of flags is assumed to be good, so that $\widehat{D_1 \cap D_2} = \widehat{Z_1 \cap Z_2}$.
 
 Finally, we know that $(\sigma_D)_{\mathrm B \mathbf{G}}$ is a Cartesian square (\cref{lem:BGcocD12Z12}) and that $(\tau)_{\mathrm B \mathbf{G}}$ is $1$-Cartesian (\cref{cor:BGcoctrivializations}). Since $(*)$ is obviously Cartesian, we get by \cref{lem:pullbacksofncart} that $(\alpha)_{\mathrm{Gr}}$ is Cartesian. In particular, we have
 \[
\begin{tikzcd}
 \mathrm{Gr}_X(S)(F_1 \cup F_2) \ar{r}{\sim} & \mathrm{Gr}_X(S)(F_1) \times \mathrm{Gr}_X(S)(F_2).
\end{tikzcd}\qedhere
\]
\end{proof}

\subsection{The Grassmannian as a fibered stack}
In order to better handle the homotopy coherent factorization structure the Grassmannian naturally carries, we will see the Grassmannian as a fibered (derived) stack. As before, we fix $X$ a smooth projective surface over $k$.

Recall that the flag Grassmannian is a natural transformation $\Gr_X \simeq \Gr_X^\loc \colon \flags_X^\dR \to \stackofstacks_k$ between functors $(\Aff_k)\op \to \inftyCat$.
The stack of derived stacks $\stackofstacks_k$ is endowed with a canonical derived structure, and thus comes from a functor $\stackofstacks_k \colon \dAff_k \to \inftyCat$.
Consider the left Kan extension of $\flags_X$ to derived affine schemes, also denoted by $\flags_X$ (so equipped with a trivial derived structure), and consider its de Rham stack $\flags_X^\dR$. The flag Grassmannian becomes a natural transformation $\Gr_X \simeq \Gr_X^\loc \colon \flags_X^\dR \to \stackofstacks_k$ between functors $\dAff_k \to \inftyCat$.

\begin{defin}
By the universal property of the (categorical) stack of stacks, we get a pointwise coCartesian fibration
\[
 p \colon \fibGr_X \to \flags_X^\dR \in \dSt^{\inftyCat}_k
\]
whose fiber $p_S^{-1}(D,Z)$ at $S \in \dAff_k$ and $(D, Z) \in \flags_X(S_\red)$ is the $\infty$-groupoid $\Gamma(S, \GrX(S; D, Z))$.
\end{defin}

\begin{rem}\label{spiegazione} Let $S\in \Aff$ (i.e. a Noetherian \emph{underived} affine scheme). Then, $\fibGr_X(S)$ is the category whose objects are pairs $(F=(D,Z), \underline{\mathcal{E}})$, where $F \in \flags_{X}(S_\red)$ and $\underline{\mathcal{E}} \in \mathrm{Gr}_X(S)(F)=\Gamma(S, \GrX(S; D, Z))$, with morphisms $(F, \underline{\mathcal{E}}) \to (F', \underline{\mathcal{E}}')$ given by pairs $(i, \alpha)$ of morphisms $i: F \to F'$ in $\flags_{X}(S_\red)$, and $\alpha: \mathrm{Gr}_X(S)(i)(\underline{\mathcal{E}}) \simeq \underline{\mathcal{E}}'$. The composition is straightforward. Moreover, $p_S: \mathcal{G}r_X(S) \to \flags_{X}(S_\red)$ sends $(F, \underline{\mathcal{E}})$ to $F$, and $(i, \alpha): (F, \underline{\mathcal{E}}) \to (F', \underline{\mathcal{E}}')$ to $i: F\to F'$. We leave to the reader the further explicit description of functoriality of $\fibGr_X(S)$ with respect to morphisms $S' \to S$ in $\Aff$.
\end{rem}

\subsection{The simplicial flag Grassmannian}\label{subsec:2segalflaggrass}
\begin{defin}
 We denote by $\fibGr_{X,2} \to \flags_{X,2}^\dR$ the fibered derived stack obtained by pulling back $\fibGr_X \to \flags_X^\dR$ along $\partial_1 \colon \flags_{X,2}^\dR \to \flags_X^\dR$:
 \[
  \begin{tikzcd}
   \fibGr_{X,2} \ar{r}{\partial_1} \ar{d} \ar[phantom]{dr}[description, near start]{\lrcorner} & \fibGr_X \ar{d} \\
   \flags_{X,2}^\dR \ar{r}[swap]{\partial_1} & \flags_X^\dR\rlap{.}
  \end{tikzcd}
 \]
\end{defin}

\begin{rem}
 Under the equivalence between fibered stacks and morphisms to the stack of stacks, the fibered stack $\fibGr_{X,2} \to \flags_{X,2}^\dR$ corresponds to the composite
 \[
  \Gr_X \circ \partial_1 \colon \flags_{X,2}^\dR \to \stackofstacks_k.
 \]
\end{rem}

Together with \cref{lem:benjexo}, the natural transformations of \eqref{projnattrans} induce a commutative diagram
\[
 \begin{tikzcd}
  \fibGr_X \ar{d} & \ar{l}[swap]{\partial_0} \fibGr_{X,2} \ar{r}{\partial_2} \ar{d} & \fibGr_X \ar{d} \\
  \flags_X^\dR & \ar{l}{\partial_0} \flags_{X,2}^\dR \ar{r}[swap]{\partial_2} & \flags_X^\dR\rlap{.}
 \end{tikzcd}
\]
Since both degeneracy morphisms $\sigma_0,\sigma_1 \colon \flags_X \to \flags_{X,2}$ satisfy $\partial_1 \circ \sigma_i = \id$, they trivially induce degeneracy morphisms $\sigma_0,\sigma_1 \colon \fibGr_X \to \fibGr_{X,2}$. All in all, we get a truncated simplicial object
\[
 \begin{tikzcd}
  & \fibGr_{X,2} \ar[shift left=8pt]{r} \ar[shift right=8pt]{r} \ar{r} \ar{d} & \fibGr_X \ar[shift left=4pt]{l} \ar[shift right=4pt]{l} \ar[shift left=4pt]{r} \ar[shift right=4pt]{r} \ar{d} & \{*\} \ar{d} \ar{l} \\
  \ar[dotted, no head]{r} & \flags_{X,2}^\dR \ar[shift left=8pt]{r} \ar[shift right=8pt]{r} \ar{r} & \flags_X^\dR \ar[shift left=4pt]{l} \ar[shift right=4pt]{l} \ar[shift left=4pt]{r} \ar[shift right=4pt]{r} & \{*\}\rlap{.} \ar{l}
 \end{tikzcd}
\]
Note moreover that since $\flags_{X,\bullet}$ is a $2$-Segal object, and since the de Rham functor preserves fiber products, the simplicial object $\flags_{X,\bullet}^\dR$ is also $2$-Segal.
Applying \cref{prop:extendsimplicial} to $X_\bullet = \flags_{X, \bullet}^\dR$ and $Y^{\leq 2}_\bullet = \fibGr_{X, \leq 2}$, we get a 2-Segal simplicial object $\fibGr_{X, \bullet}$ in $\mathrm{Fun}(\dAff_k\op, \inftyCat)$ together with a simplicial map $p_\bullet : \fibGr_{X, \bullet} \to \flags_{X, \bullet}^\dR$.
Notice that by construction, the derived stack $\fibGr_{X,n}$ can be obtained as the pullback
\[
\begin{tikzcd}
 \fibGr_{X,n} \ar{d}[swap]{p_{n}} \ar{r} \ar[phantom]{dr}[description, near start]{\lrcorner} & \fibGr_{X,1} \rlap{${}:=\fibGr_X$} \ar{d}{p_{1}} \\
 \flags_{X,n}^\dR \ar{r}{\cup} & \flags_{X,1}^\dR \rlap{${}:=\flags_X^\dR$.}
\end{tikzcd}
\]

\begin{defin}\label{simplGrass}
 We will refer to $\fibGr_{X, \bullet}$ as the $2$-Segal \emph{simplicial flag Grassmannian}.
\end{defin}

\begin{rem}\label{Griscomm} Using Remark \ref{Fliscomm} one can easily verify that the associative algebra in correspondences induced (via \cite{stern}) by the 2-Segal object $\fibGr_{X, \bullet}$ is, in fact, commutative (i.e. $E_{\infty}$).
\end{rem}


\begin{rem}\label{remfattorizzazione}
 With this simplicial point of view, factorization (\cref{thm:factorization}) can be reinterpreted as: the following square
 \[
  \begin{tikzcd}
 \fibGr_{X, 2} \ar{r}{\partial_0, \partial_2} \ar{d}[swap]{p_2} & \fibGr_{X} \times \fibGr_{X} \ar{d}{p_1 \times p_1} \\ \flags_{X,2}^\dR \ar{r}[swap]{\partial_0, \partial_2} & \flags_{X,1}^\dR \times \flags_{X,1}^\dR
  \end{tikzcd}
 \]
 is Cartesian.
\end{rem}

\begin{rem}\label{Grtildeandisos} Factorization (\cref{thm:factorization} or Remark \ref{remfattorizzazione}) allows for the following, equivalent, description of the simplicial flag Grassmannian. Define $\widetilde{\mathcal{G}r}_{X,2}$ as the pullback $$\xymatrix{\widetilde{\mathcal{G}r}_{X,2} \ar[r]^-{(\tilde{\partial}_2, \tilde{\partial}_0)} \ar[d] & \mathcal{G}r_{X} \times \mathcal{G}r_{X} \ar[d]^-{p_1 \times p_1} \\
\flags^\dR_{X,2} \ar[r]_{(\partial_2, \partial_0)} & \flags^\dR_{X} \times \flags^\dR_{X}} 
$$ 
then use \cref{thm:factorization} to show that there is a further induced map $\tilde{\partial}_1: \widetilde{\mathcal{G}r}_{X,2} \to \mathcal{G}r_{X}$ satisfying the hypothesis of Proposition \ref{prop:extendsimplicial}, so that Proposition \ref{prop:extendsimplicial} produces for us a simplicial and 2-Segal object $\widetilde{\mathcal{G}r}_{X,\bullet}$ together with a map to $\flags_{X, \bullet}$. Using again Proposition \ref{prop:extendsimplicial} for $X_\bullet = \mathcal{G}r_{X,\bullet}$ and $Y_\bullet^{\leq 2}= \widetilde{\mathcal{G}r}_{X,\bullet}$ (and viceversa), we get induced maps, in both directions, between $\widetilde{\mathcal{G}r}_{X,\bullet}$ and $\mathcal{G}r_{X,\bullet}$ which are mutually inverse isomorphisms over  $\flags^\dR_{X, \bullet}$, by factorization.
\end{rem}

\begin{rem} A completely analog construction to the one leading to Definition \ref{simplGrass} yields a simplicial flag Grassmannian $\fibGr^{\mathrm{Perf}}_{X, \bullet}$, for the stack of perfect complexes on $X$ (see Remark \ref{thm:factorization} (b)).
\end{rem}

\subsection{Representability}

\begin{thm}\label{thmRepresentabilityfibGr}
 The projection $p \colon \fibGr_X \to \flags_X^\dR$ is representable by a derived ind-scheme, in the following sense: for any derived affine scheme $S$ and any morphism $F \colon S \to \flags_X^\dR$ (in $\dSt_k^{\inftyCat}$), the fiber $\fibGr_X^F := p^{-1}(F)$ lies in $\dSt_k \subset \dSt_k^{\inftyCat}$ and is representable by a derived ind-scheme.
\end{thm}

\begin{rem}
 The fiber $\fibGr_X^F$ tautologically identifies with the derived stack over $S$ classified by the morphism
 \[
  \begin{tikzcd}
    S \ar{r} & \flags_X^\dR \ar{r}{\Gr_X} & \stackofstacks_k.
  \end{tikzcd}
 \]
 It immediately follows that $\fibGr_X^F$ belongs to $\dSt_k$.
\end{rem}

\begin{proof}
  The proof of \cref{thmRepresentabilityfibGr} will rely on the representability criterion \cite[Chap. 2 Cor. 1.3.13]{GaiRozII}. We will thus:
  \begin{enumerate}
    \item Show the underlying underived stack $\trunc(\fibGr_X^{F})$ is representable by an ind-scheme and
    \item Show the derived stack $\fibGr_X^{F}$ admits a connective deformation theory.
  \end{enumerate}

Notice first that the stack $\trunc(\fibGr_X^{F})$ is actually a sheaf (on $\Aff_S$), by Lemma \ref{GrIsSetsValued}.

  We start by treating the case $\mathbf{G}=GL_n$.
  If $F$ is of the form $(D, Z = \emptyset)$, since being a vector bundle is an open condition inside the moduli of torsion-free sheaves, \cite[Prop. 3.13 and Remark. 3.8]{HLHJ} implies (by taking $X$ in loc. cit. to be our $X \times S$, while $S$ is the same, so that $X\times S \to S$ is indeed projective and of finite presentation) that $\trunc(\fibGr_X^{F})$ is represented by an ind-scheme which is ind-quasi-projective over $S$.
  For a general flag $F=(D, Z)$, we consider the obvious forgetful map $q \colon \trunc(\fibGr_X^{F}) \to \trunc(\fibGr_X^{(D, \emptyset)})$ between sheaves on $\Aff_S$.
  We claim that this map is representable and affine. \\
  Indeed, first notice that $\underline{\mathsf{Bun}}_{\hZ^{\affinize}} \simeq \underline{\mathsf{Bun}}_{\hZ}$ has an affine diagonal (not of finite type, in general): this follows from \cite[Theorem 1.0.1]{Wang} applied to $\underline{\mathsf{Bun}}_{Z_{n}}$, where $Z_{n}$ is the $n$-th thickening of $Z$ in $Y:=X\times S$.
  Now, the fiber of $q \colon \trunc(\fibGr_X^{F}) \to \trunc(\fibGr_X^{(D, \emptyset)})$ at $T=\Spec A \to \trunc(\fibGr_X^{(D, \emptyset)})$, corresponding to some $(\underline{\mathcal{E}}, \varphi)$, is the closed subscheme of the affine scheme
  \[
    \underline{\mathrm{Isom}}_{\underline{\mathsf{Bun}}_{\hZ^{\affinize},\, T}}(\underline{\mathcal{E}}_{|\widehat Z_T^{\affinize}}, \mathcal{E}_0)
  \]
  defined by imposing the restriction to a fixed element on the open subscheme $Y_T\smallsetminus D_T \cap \widehat Z_T^{\affinize}$. Hence, the map $q$ is indeed representable by affines.
  Therefore (see e.g. \cite[Lemma 1.7]{Richarz_BoAG}), $\trunc(\fibGr_X^{F})$ is represented by an ind-scheme, since $\trunc(\fibGr_X^{(D, \emptyset)})$ is.
  
  \noindent Now, the same technique as in \cite[Section 1.2]{Zhu2017}, allows us to deduce from the $\mathbf{G}=GL_n$ case the general $\mathbf{G}$ case.
  More precisely, embed $\mathbf{G}$ into $GL_n$ so that $GL_n/\mathbf{G}$ is affine; this gives, for $\mathbf{G}$, a closed sub-ind-scheme of the corresponding $\trunc(\fibGr_X^{F})$ for $GL_n$.
  This concludes step (1).
  
  We now focus on step (2): proving that the derived stack $\fibGr_X^F$ admits a connective deformation theory.
  \newcommand{\adjbun}{\mathcal P_\mathrm{ad}}
  \newcommand{\RGamma}{\operatorname{\mathsf{R}\Gamma}}
  Let $q \colon \Spec B \to \fibGr_X^F$ be a point.
  Set $X_B := X \times \Spec B$ and denote by $D_B$ and $Z_B$ the corresponding closed subschemes.
  The point $q$ induces a $\mathbf{G}$-bundle $\mathcal P$ on $X_B$.
  We denote by $\adjbun \in \mathrm{Vect}(X_B)$ the associated adjoint (vector) bundle.

  Using that $X$ is proper, the derived stack $\fibGr_X^F$ then admits an ind-pro-perfect tangent $B$-complex at $q$, described as
  \[
   \mathbb T_{\fibGr_X^F,q}[-1] = \RGamma\left(\widehat D_B, \adjbun\right) \newtimes_{\RGamma\left(\widehat D_B \smallsetminus D_B \cup \widehat Z_B, \adjbun\right)} 0.
  \]
  It is stable under base change, and the stack $\fibGr_X^F$ thus admits a global deformation theory.
  Moreover, as $D_B$ is a divisor, the homotopy fiber of the morphism $\RGamma(\widehat D_B, \adjbun) \to \RGamma(\widehat D_B \smallsetminus D_B, \adjbun)$ has negative Tor-amplitude.
  It follows that $\fibGr_X^F$ has a connective deformation theory, and it is therefore representable by a derived ind-scheme, by \cite[Chap. 2 Cor. 1.3.13]{GaiRozII}.
\end{proof}

\begin{rem}\label{oncarGrqproj}
  It is useful to record that the first part of the proof of \cref{thmRepresentabilityfibGr} shows that the restriction of $\fibGr_X$ to purely divisorial flags is represented by derived ind-quasi-projective ind-scheme.
\end{rem}

\begin{rem}\label{maybeitsproper}
If one knows that the union map $\cup \colon \flags^{\dR}_{X,2} \to \flags^{\dR}_X$ is (representable, in the sense of Theorem \ref{thmRepresentabilityfibGr}, and) proper, one can conclude, via the pullback  \[
    \xymatrix{\fibGr_{X, 2} \ar[r]^-{\partial_1} \ar[d] & \fibGr_{X} \ar[d] \\
    \flags^{\dR}_{X, 2} \ar[r]_-{\cup} & \flags^{\dR}_X}
  \]
that $\partial_1: \fibGr_{X, 2} \to \fibGr_{X}$ is (representable and) proper, too. We can prove that $\cup \colon \flags^{\dR}_{X,2} \to \flags^{\dR}_X$ is representable (in the sense of \ref{thmRepresentabilityfibGr}): this is essentially a consequence of the fact that, if $F\in \flags^{\dR}_X(S)$, and $f: S \to T$ is a map of derived affine schemes, then $F_1 \cup F_2 = F'_1 \cup F'_2 = f_{\mathrm{red}}^{*} F$ implies $(F_1, F_2)= (F'_1, F'_2)$, for any $(F_1, F_2)$ and $(F'_1, F'_2)$ in $\flags^{\dR}_{X,2}(T)$. On the other hand, though we believe that  $\cup \colon \flags^{\dR}_{X,2} \to \flags^{\dR}_X$ is furthermore proper, we have, at the moment, a complete proof only for $X$ a (smooth) cubic surface in $\mathbb{P}_k^{3}$.
\end{rem} 

%

\section{Hecke stacks and exotic derived structures on \texorpdfstring{$\mathrm{Bun}_X$}{Bunₓ}}\label{sessioneesotica}

\subsection{Grassmannian at a fixed ``large'' flag}

\begin{defin}\label{bigflag} Let $S$ be a Noetherian $\mathbb{C}$-scheme, and $Y \to S$ a smooth morphism of relative dimension $\geq 2$. A flag $(D,Z)$ in $Y/S$ will be called \emph{large} if for any $s \in S$, the fiber $Z_s$ meets any irreducible component of the fiber $D_s$. 
\end{defin}

\begin{prop}\label{beilinson} \emph{(A. Beilinson)} Let $S$ be a Noetherian $\mathbb{C}$-scheme, $Y \to S$ a smooth morphism of relative dimension $\geq 2$, and $(D,Z)$ a large flag in $Y/S$. Then, for any fixed $\mathbf{G}$-bundle $\overline{\mathcal{E}}$ on $Y$, the set of global sections
$$\mathrm{Bun}^{\mathbf{G}}(Y) \times_{\mathrm{Bun}^{\mathrm{G}}(Y\setminus D) \times_{\mathrm{Bun}^{\mathbf{G}}(\widehat{Z} \setminus D)} \mathrm{Bun}^{\mathbf{G}}(\widehat{Z})    } \left\{ \overline{\mathcal{E}} \right\}$$ is a singleton. In particular, for $\overline{\mathcal{E}}$ being the trivial bundle, we have that the set 
$\mathrm{Gr}(Y;D,Z)$ (see Definition \ref{def:flagGrassgeneral})\footnote{We use here, and in the previous display, the fact that a derived stack $\mathcal{X}$ and its truncation $\mathrm{t}\mathcal{X}$ share the same points with values in classical (i.e. underived) schemes.} is a singleton.
\end{prop}

\begin{proof} We thank A. Beilinson for providing the following detailed proof. First of all, the argument will be local in $Y$, so it will be enough to consider the case where $\overline{\mathcal{E}}$ is the trivial bundle. \\
Consider a function $f \in \Gamma(Y\setminus D, \mathcal{O}_Y)$, and $U \supset (X\setminus D)$ the largest open subset of $Y$ to which $f$ extends as a regular function (such a $U$ exists since $D$ is an effective Cartier divisor). Define $T: Y \setminus U$. For any point $y \in Y$, let $Y^{\wedge}_y= \mathrm{Spec} (\widehat{\mathcal{O}_{Y,y}})$. The map $\pi^{\wedge}_y: Y_y^{\wedge} \to Y$ is flat, so the base change $D_y^{\wedge}$ of $D$ along $\pi^{\wedge}_y$ is still an effective Cartier divisor on $Y_y^{\wedge}$.\\

\noindent\emph{Claim.} (a) If $y \in T$, then the restriction of $f$ to $Y_y^{\wedge} \setminus D_y^{\wedge}$ does not extend to $Y_y^{\wedge}$.\\
\noindent (b) For any point $s \in S$, the fiber $T_s$ is either empty or the union of some irreducible components of the fiber divisor $D_s$.\\

\noindent\emph{Proof of Claim.} (a) By replacing $D$ with $nD$ for $n$ sufficiently large, we may assume that $f \in \Gamma (Y, \mathcal{O}_Y (D))$. Let $\mathcal{L}:= \mathcal{O}_Y (D)/ \mathcal{O}_Y$, and $\overline{f}$ be the image of $f$ in $\Gamma (Y, \mathcal{L})$. Note that $\mathcal{L}$ is an invertible $\mathcal{O}_Y$-Module. Suppose that claim (a) is wrong, and there exists $y \in T$ such that $f$ does extend to $Y_y^{\wedge}$. Then, the section $(\pi^{\wedge}_y)^*\overline{f} \in \Gamma(Y_y^{\wedge}, (\pi^{\wedge}_y)^*(\mathcal{L}) )= \widehat{\mathcal{O}_{Y,y}}(D_y^{\wedge})/ \widehat{\mathcal{O}_{Y,y}}$ vanishes. Since $ \mathcal{O}_{Y,y} \to \widehat{\mathcal{O}_{Y,y}}$ is faithfully flat, this implies that also the image of $\overline{f}$ in $\mathcal{O}_{Y,y}(D)/\mathcal{O}_{Y,y}$ vanishes, so that $y \notin T$, which is a contradiction.\\
\noindent (b) $T_s$ is a closed subset of $D_s$, so it will be enough to show that $T_s$ does not contain a closed point of $D_s$ as a connected component. Indeed, if $y$ is such a point then, replacing $Y$ by a neighborhood of $y$, we can assume (by semi-continuity of the fibers applied to $T/S$ \cite[13.1.3]{PMIHES_1966__28__5_0}) that $T$ is quasi-finite over $S$. Thus, all the fibers $T_{s'}$ for $s'\in S$, are closed points in $Y_{s'}$, and so $\mathrm{depth}(\mathcal{O}_x)\geq 2$ for any $x \in T$. Thus $f$ is regular at $T$, which is a contradiction. The proof of the Claim is complete. \\

The Claim remains true if we replace $f$ by an element in $\Gamma(Y\setminus D, \mathcal{O}^n_Y)$, for any $n$, i.e. by a section of $\mathbb{A}_Y^n$ on $Y \setminus D$. Therefore, the Claim still holds if $f$ is replaced by any section on $Y\setminus D$ of any affine $Y$-scheme of finite type $E/X$ (since, locally on $Y$, such an $E/Y$ is a closed subscheme of some $\mathbb{A}_Y^n$). In particular, we may take $\mathcal{E}/Y$ to be a $\mathbf{G}$-bundle on $Y$.\\
Consider $$(\mathcal{E}, \varphi, \psi) \in \mathrm{Gr}(Y;D,Z) = \mathrm{Bun}^{\mathbf{G}}(Y) \times_{\mathrm{Bun}^{\mathrm{G}}(Y\setminus D) \times_{\mathrm{Bun}^{\mathbf{G}}(\widehat{Z} \setminus D)} \mathrm{Bun}^{\mathbf{G}}(\widehat{Z})    } \left\{ \mathrm{triv} \right\}$$
where $\mathcal{E}$ is a $\mathbf{G}$-bundle on $Y$, $\varphi$ (resp. $\psi$) is a trivialization of $\mathcal{E}$
on $Y\setminus D$ (resp. on $\widehat{Z}$), and $\varphi$ and $\psi$ agree on $\widehat{Z} \setminus D$. The trivialization $\varphi$ can be identified, as usual, with a section of $\mathcal{E}/Y$ on $Y\setminus D$. Let $z\in Z$. Since $$Y_z^{\wedge} \setminus D_z^{\wedge} = (\widehat{Z}^{\mathrm{aff}}\setminus D)_z \hookrightarrow (\widehat{Z}^{\mathrm{aff}})_z =Y_z^{\wedge}\, ,$$ the compatibility between $\varphi$ and $\psi$ on $\widehat{Z} \setminus D$ implies that the restriction of $\varphi$ to $Y_z^{\wedge} \setminus D_z^{\wedge}$ does extend (via $\psi$) to $Y_z^{\wedge}$. Hence, by Claim (a), $z \notin Y$. Since $z \in Z$ is arbitrary, we get $Z \cap T= \emptyset$. Therefore, for any $s \in S$, we have $Z_s \cap T_s = \emptyset$. But $(D,Z)$ is a large flag in $Y/S$, so $Z_s$ meets any irreducible component of $D_s$, and (Claim (b)) $T_s$ is either empty or a union of irreducible components of $D_s$. Therefore $T_s$ must be empty, for any $s \in S$. Thus $T=\emptyset$, and this means (by definition of $T$) that the trivialization $\varphi$ of $\mathcal{E}$ on $Y\setminus D$ extends to a trivialization $\overline{\varphi}$ on the whole $Y$, whose restriction to $\widehat{Z}^{\mathrm{aff}}$ (hence to $\widehat{Z}$), coincides with $\psi$. Moreover, since $D$ is an effective Cartier divisor in $Y$, and $\mathcal{E}$ is separated, two such trivializations $\overline{\varphi}$ and $\overline{\varphi}'$ extending the same  $\varphi$, can be identified with maps $Y \to \mathcal{E}$ that coincide on the quasi-compact schematically dense open subset $(Y\setminus D)\subset Y$, hence $\overline{\varphi}= \overline{\varphi}'$. In other words the extension, $\overline{\varphi}$ exists and is unique. This ends the proof of the Proposition. 
\end{proof}

\begin{rem} Note that the analogue of Proposition \ref{beilinson} is obviously false when the stack of $\mathbf{G}$-bundles is replaced with the stack of perfect complexes.
\end{rem}

\begin{rem} Proposition \ref{beilinson} does not imply that (with the same notations as in Proposition \ref{beilinson})  the \emph{derived} stack $$\underline{\mathsf{Bun}}^{\mathbf{G}}(Y) \times_{\underline{\mathsf{Bun}}^{\mathrm{G}}(Y\setminus D) \times_{\underline{\mathsf{Bun}}^{\mathbf{G}}(\widehat{Z} \setminus D)} \underline{\mathsf{Bun}}^{\mathbf{G}}(\widehat{Z})    } \left\{ \overline{\mathcal{E}} \right\}$$ is trivial. This should be considered as  analogous to the fact that while at the level of classical (truncated) stacks the restriction map $\mathrm{t}(\underline{\mathsf{Bun}}_{\mathbb{A}^2}^{\mathrm{G}}) \to \mathrm{t}(\underline{\mathsf{Bun}}_{\mathbb{A}^2 \setminus 0}^{\mathrm{G}})$ is an equivalence, this is no more true at the level of derived stacks (as can be checked on cotangent complexes).
\end{rem}

\begin{rem}\label{inspiteofB} Proposition \ref{beilinson} does not imply that the flag factorization property (see e.g. Remark \ref{remarkfactorization}) is trivial, i.e. it is of the form singleton=singleton (not even for the truncations), for \emph{arbitrary} pair of good flags. Take, for example, $X= \mathbb{P}^1 \times \mathbb{P}^1$, fix $x_1, x_2, y \in \mathbb{P}^1$ with $x_1 \neq x_2$, and consider the flags $F_1 =(D_1= \mathbb{P}^1 \times \{ x_1 \}, Z_1=\{ (y, x_1)\})$, and $F_2 =(D_2= \mathbb{P}^1 \times \{ x_2 \}, Z_2=\emptyset)$. Then, $(F_1, F_2) \in \flags_{X,2}(\mathbb{C})$, so flag factorization holds, but neither $F_2$ nor $F_1 \cup F_2$ are large flags.
\end{rem}

\subsection{Exotic derived structures on \texorpdfstring{$\mathsf{Bun}_X$}{Bunₓ}}\label{esotiche}

Let $F=(D,Z) \in \underline{\mathsf{Fl}}_X(\mathbb{C})$. For simplicity, let us denote by $\mathcal{H}_{\mathbf{G}, X,F}$ the derived stack $$\mathcal{H}_{\mathbf{G}, X,F}:= \underline{\mathsf{Bun}}^{\mathbf{G}}_{X} \times_{\underline{\mathsf{Bun}}^{\mathbf{G}}_{X\setminus D} \times_{\underline{\mathsf{Bun}}^{\mathbf{G}}_{\widehat{Z} \setminus D}} \underline{\mathsf{Bun}}^{\mathbf{G}}_{\widehat{Z}}    } \underline{\mathsf{Bun}}^{\mathbf{G}}_{X}$$
This will be called the derived \emph{Hecke stack of $X$ at the fixed flag} $F$.\\

\begin{prop}\label{lemexotic} If $F=(D, Z)$ is a large flag in $X/\mathbb{C}$ (Definition \ref{bigflag}), i.e. $Z$ meets any irreducible component of $D$, then the truncation of either of the canonical projections
$$\mathrm{t}(p_1), \, \mathrm{t}(p_2) : \mathrm{t}(\mathcal{H}_{\mathbf{G}, X,F}) \longrightarrow \mathrm{t}(\underline{\mathsf{Bun}}^{\mathbf{G}}_X)$$ is an equivalence of underived stacks, with quasi-inverse given by the diagonal morphism $$\Delta: \mathrm{t}(\underline{\mathsf{Bun}}^{\mathbf{G}}_X) \to \mathrm{t}(\mathcal{H}_{\mathbf{G}, X,F})\,\,\,\,\, \mathcal{E} \longmapsto (\mathcal{E}, \mathcal{E}; \mathrm{id}, \mathrm{id}).$$
\end{prop}

\begin{proof} We will prove the statement for the second projection, the argument being the same for the first projection. It is obvious that $\Delta$ is a section of $\mathrm{t}(p_2)$. So, it is enough to prove that there exists a pseudo-natural transformation $\alpha: \Delta \circ \mathrm{t}(p_2) \to \mathrm{Id}_{\mathrm{t}(\mathcal{H}_{\mathbf{G}, X,F})}$, such that $\alpha(S)$ is an isomorphism of functors $\mathrm{t}(\mathcal{H}_{\mathbf{G}, X,F})(S) \to \mathrm{t}(\mathcal{H}_{\mathbf{G}, X,F})(S)$ for any (classical) Noetherian affine scheme $S$. \\We first observe that, for any Noetherian $S$, $(D \times S, Z\times S)$ is still a large flag in $X \times S / S$. Indeed, by applying the same argument to any irreducible component of $D$, we may clearly suppose that $D$ is itself irreducible. But $D$ is defined over $\mathbb{C}$, so it is actually geometrically irreducible. So, if $s:\mathrm{Spec}(k(s)) \to S$ is a point in $S$, the fiber $D_s$ is irreducible ($k(s)$ being a field extension of $\mathbb{C}$) and the closed subscheme $Z_s \hookrightarrow D_s$ is non-empty, since $\mathrm{Spec}(k(s)) \to \mathrm{Spec}\, \mathbb{C}$ is faithfully flat and $Z\neq \emptyset$ by hypothesis. Therefore $(D \times S, Z\times S)$ is still a large flag in $X \times S / S$, as claimed.
For $\underline{\mathcal{E}}:=(\mathcal{E}_1, \mathcal{E}_2; \varphi, \psi ) \in \mathrm{t}(\mathcal{H}_{\mathbf{G}, X,F})(S)$ (where $\varphi: \mathcal{E}_1 | X_S\setminus D_S \simeq \mathcal{E}_2 | X_S\setminus D_S$, and $\psi: \mathcal{E}_1 | \widehat{Z_S} \simeq \mathcal{E}_2 | \widehat{Z_S}$), we note that $(\Delta \circ \mathrm{t}(p_2))(S)(\underline{\mathcal{E}}) =(\mathcal{E}_2, \mathcal{E}_2; \mathrm{id}, \mathrm{id} )$, so that both $(\Delta \circ \mathrm{t}(p_2))(S)(\underline{\mathcal{E}})$ and $\mathrm{Id}_{\mathrm{t}(\mathcal{H}_{\mathbf{G}, X,F})}(S)(\underline{\mathcal{E}})= (\mathcal{E}_1, \mathcal{E}_2; \varphi, \psi )$ belong to $$\mathrm{Bun}^{\mathbf{G}}(X_S) \times_{\mathrm{Bun}^{\mathrm{G}}(X_S\setminus D_S) \times_{\mathrm{Bun}^{\mathbf{G}}(\widehat{Z_S} \setminus D_S)} \mathrm{Bun}^{\mathbf{G}}(\widehat{Z_S})    } \left\{\mathcal{E}_2 \right\}$$ which is a singleton by Proposition \ref{beilinson} (since $(D \times S, Z\times S)$ is still a large flag in $X \times S / S$). Therefore, there exists a (unique) isomorphism $$\alpha(S)(\underline{\mathcal{E}}): (\mathcal{E}_2, \mathcal{E}_2; \mathrm{id}, \mathrm{id} )=(\Delta \circ \mathrm{t}(p_2))(S)(\underline{\mathcal{E}})\simeq \mathrm{Id}_{\mathrm{t}(\mathcal{H}_{\mathbf{G}, X,F})}(S)(\underline{\mathcal{E}})= (\mathcal{E}_1, \mathcal{E}_2; \varphi, \psi )$$ which, we recall, is of the form $(\overline{\varphi}^{-1}, \mathrm{id}_{\mathcal{E}_2})$ where  $\overline{\varphi}: \mathcal{E}_1 \simeq \mathcal{E}_2$ is the unique common extension of $\varphi$ and $\psi$ to all of $X_S$.\\ Finally, the fact that $\alpha(S)$ is a natural transformation between functors $\mathrm{t}(\mathcal{H}_{\mathbf{G}, X,F})(S) \to \mathrm{t}(\mathcal{H}_{\mathbf{G}, X,F})(S)$ can be shown as follows. Let $$(f_1,f_2): \underline{\mathcal{E}}:=(\mathcal{E}_1, \mathcal{E}_2; \varphi, \psi )\longrightarrow \underline{\mathcal{E}}':=(\mathcal{E}'_1, \mathcal{E}'_2; \varphi', \psi' )$$ be a morphism in $\mathrm{t}(\mathcal{H}_{\mathbf{G}, X,F})(S)$. We need to prove that the diagram $$\xymatrix{(\mathcal{E}_2, \mathcal{E}_2; \mathrm{id}, \mathrm{id} ) \ar[d]_-{(f_2,f_2)}\ar[r]^-{\alpha(S)(\underline{\mathcal{E}})} & (\mathcal{E}_1, \mathcal{E}_2; \varphi, \psi ) \ar[d]^-{(f_1,f_2)} \\ (\mathcal{E}'_2, \mathcal{E}'_2; \mathrm{id}, \mathrm{id} ) \ar[r]_-{\alpha(S)(\underline{\mathcal{E}}')} & (\mathcal{E}'_1, \mathcal{E}'_2; \varphi', \psi' )
} $$ commutes. Recalling the definition of $\alpha(S)$, this commutativity easily boils down to showing that the following two compositions 
\begin{equation} \label{havetoprove}
\xymatrix{\mathcal{E}_2 \ar[r]^-{\overline{\varphi}^{-1}} & \mathcal{E}_1 \ar[r]^-{f_1} & \mathcal{E}'_1
} \qquad \xymatrix{\mathcal{E}_2 \ar[r]^-{f_2} & \mathcal{E}'_2 \ar[r]^-{\overline{\varphi'}^{-1}} & \mathcal{E}'_1
} 
\end{equation} agree.

Note that, since $(f_1,f_2)$ is a morphism in $\mathrm{t}(\mathcal{H}_{\mathbf{G}, X,F})(S)$, we have in particular a commutative diagram of isomorphisms $$\xymatrix{\mathcal{E}_1 |_{X_S\setminus D_S} \ar[rr]^-{f_1|_{X_S\setminus D_S}} \ar[d]_-{\varphi} &&
\mathcal{E}'_1|_{X_S\setminus D_S} \ar[d]^-{\varphi'} \\ \mathcal{E}_2 |_{X_S\setminus D_S} \ar[rr]_-{f_2|_{X_S\setminus D_S}} &&
\mathcal{E}'_2|_{X_S\setminus D_S}.
}$$
In particular, by definition of $\alpha(S)$, we deduce that the two composite maps in (\ref{havetoprove}) do agree on $\mathcal{E}_2 |_{X_S \setminus D_S}$.
Since $D$ is an effective Cartier divisor in $X$, the same is true for $D_S$ in $X_S$, hence $X_S\setminus D_S$ is quasi-compact and schematically dense in $X_S$. Now,
 $$ \mathcal{E}_2|_{X_S\setminus D_S} =\mathcal{E}_2 \times_{X_S} (X_S\setminus D_S)  = \mathcal{E}_2 \setminus (\mathcal{E}_2 \times_{X_S} D_S) $$
and $\mathcal{E}_2 \to X_S$ is flat (any $\mathbf{G}$-bundle is (faithfully) flat over its base), so that  $\mathcal{E}_2 \setminus (\mathcal{E}_2 \times_{X_S} D_S)$ is quasicompact open and schematically dense in $\mathcal{E}_2$ (\cite[TAG081I]{stacks-project}). Now, $\mathcal{E}'_1$ is a separated scheme, and the two maps $\mathcal{E}_2 \to \mathcal{E}'_1$ in (\ref{havetoprove}) coincide on the quasicompact open and schematically dense $\mathcal{E}_2 \setminus \mathcal{E}_2\times_{X_S} D_S$, therefore they coincide on all of $\mathcal{E}_2$ (\cite[TAG01RH]{stacks-project}). Therefore, $\alpha(S)$ is indeed a natural transformation, for any $S$. We leave to the reader the verification of the pseudo-functoriality behaviour of $\alpha(S)$ with respect to maps $S' \to S$. We simply note that this pseudo-functoriality follows from the uniqueness of the extensions $\overline{\varphi}$ and $\overline{\varphi_{S'}}$ of the isomorphisms $\varphi :\mathcal{E}_1 |_{X_S \setminus D_S} \simeq \mathcal{E}_2 |_{X_S \setminus D_S}$ and $\varphi_{S'}: \mathcal{E}_1 |_{X_{S'} \setminus D_{S'}} \simeq \mathcal{E}_2 |_{X_{S'} \setminus D_{S'}}$, and is simplified by the fact that the diagrams defining the pseudo-naturality of the pseudo-natural transformations $\Delta \circ \mathrm{t}(p_2)$ and $\mathrm{Id}_{\mathrm{t}(\mathcal{H}_{\mathbf{G}, X,F})}$ are actually strictly commutative.

\end{proof}

As a consequence of Proposition \ref{lemexotic}, for each large flag $F$ in $X/\mathbb{C}$, we may view $\mathcal{H}_{\mathbf{G}, X,F}$ as a derived enhancement of the underived stack $\mathrm{t}(\underline{\mathsf{Bun}}^{\mathbf{G}}_X)$, that will be denoted as $\underline{\mathsf{Bun}}^{\mathbf{G}}_{F,X}$. As the following computations show, in general $\underline{\mathsf{Bun}}^{\mathbf{G}}_{F,X}$ is an \emph{exotic} derived structure on $\mathrm{t}(\underline{\mathsf{Bun}}^{\mathbf{G}}_{X})$, i.e. $\underline{\mathsf{Bun}}^{\mathbf{G}}_{F,X} \neq \underline{\mathsf{Bun}}^{\mathbf{G}}_{X}$, or more precisely, neither of the canonical projection maps $\underline{\mathsf{Bun}}^{\mathbf{G}}_{F,X} \to \underline{\mathsf{Bun}}^{\mathbf{G}}_{X}$ is an equivalence of derived stacks.

\begin{eg}\label{P2O3}
Let $X= \mathbb{P}^2$, $D= \mathbb{P}^1$, and $Z$ be any point in $D$, and define $F=(D,Z)$. Consider the (rigid) vector bundle $\mathcal{E}:= \mathcal{O} \oplus \mathcal{O}(3)$ on $X$, that we identify, notationally with the $\mathbb{C}$-point $(\mathcal{E}, \mathcal{E}, \mathrm{id})$ of $\mathcal{H}_{\mathsf{GL}_2, \mathbb{P}^2,F}$. We will denote by  $\mathbb{T}_{F; \mathcal{E}} := \mathbb{T}_{\mathcal{E}}\mathcal{H}_{\mathsf{GL}_2, \mathbb{P}^2,F}$, and by $\mathbb{T}_{\mathcal{E}} := \mathbb{T}_{\mathcal{E}}\underline{\mathsf{Bun}}^{\mathbf{G}}_X$. We wish to compare $\mathbb{T}_{F; \mathcal{E}}$ with $\mathbb{T}_{\mathcal{E}}$. In order to do this, we use the fiber sequences of complexes of $\mathbb{C}$-vector spaces
\begin{equation}\label{exofibseq1}\mathbb{T}_{F; \mathcal{E}} \longrightarrow \mathbb{T}_{\mathcal{E}} \times \mathbb{T}_{\mathcal{E}} \longrightarrow \mathbb{T}_{\mathcal{E}}':= \mathbb{T}_{(\mathcal{E}, \mathcal{E}, \mathrm{id})}(\underline{\mathsf{Bun}}^{\mathbf{G}}_{X\setminus D} \times_{\underline{\mathsf{Bun}}^{\mathbf{G}}_{\widehat{Z} \setminus D}} \underline{\mathsf{Bun}}^{\mathbf{G}}_{\widehat{Z}} ) \end{equation}
\begin{equation}\label{exofibseq2}\mathbb{T}_{\mathcal{E}}' \longrightarrow \mathbb{T}_{\mathcal{E}}(\underline{\mathsf{Bun}}^{\mathbf{G}}_{X\setminus D}) \times \mathbb{T}_{\mathcal{E}}(\underline{\mathsf{Bun}}^{\mathbf{G}}_{\widehat{Z}}) \longrightarrow \mathbb{T}_{\mathcal{E}}(\underline{\mathsf{Bun}}^{\mathbf{G}}_{\widehat{Z} \setminus D}).\end{equation}
If we use homogeneous coordinates $[x_0,x_1,x_2]$ on $X= \mathbb{P}^2$, and identify $D=\left\{ x_0 =0\right\}$, $Z=\left\{ x_0 = x_1 =0\right\}$, we have $$X\setminus D = U_0 \simeq \mathbb{A}_{t_1,t_2}^2=\mathrm{Spec}(\mathbb{C}[t_1, t_2])\, , \,t_i := x_i/x_0,\, i=1,2 $$ $$\widehat{Z}^{\mathrm{aff}} \simeq (U_{2})_{Z}^{\wedge} \simeq (\mathbb{A}_{s_0,s_1}^2)_{(0,0)}^{\wedge}= \mathrm{Spec}(\mathbb{C}[[s_0,s_1]]),\, \,s_j = x_j / x_0, j=0,1$$ $$ \widehat{Z} \setminus D \simeq \mathrm{Spec}(\mathbb{C}[[s_0,s_1]][s_0^{-1}]).$$
We have used above the usual notation $U_i :=\left\{ x_i \neq 0\right\}$ for the standard open subsets in $\mathbb{P}^2$.  Since $X\setminus D$, $\widehat{Z}^{\mathrm{aff}}$, and $\widehat{Z} \setminus D$ are affine schemes, $\mathrm{H^i}(\mathbb{T}_{\mathcal{E}})= 0$ for $i\neq -1, 0, 1$, and $$\mathrm{H^{-1}}(\mathbb{T}_{\mathcal{E}})= \mathrm{Ext}_{\mathbb{P}^2}^0(\mathcal{E},\mathcal{E}) \simeq \mathbb{C}^{12} \qquad \mathrm{H^{0}}(\mathbb{T}_{\mathcal{E}})= \mathrm{Ext}_{\mathbb{P}^2}^1(\mathcal{E}, \mathcal{E}) =0 \qquad \mathrm{H^{1}}(\mathbb{T}_{\mathcal{E}})= \mathrm{Ext}_{\mathbb{P}^2}^2(\mathcal{E},\mathcal{E}) \simeq \mathbb{C}\, ,$$ the fiber sequences (\ref{exofibseq1}) and (\ref{exofibseq2}) yield exact sequences
\begin{equation}\label{questaqua} \xymatrix{0 \to \mathrm{H}^{-1}(\mathbb{T}'_{\mathcal{E}}) \ar[r]  & \mathrm{H}^{-1}(\mathbb{T}_{\mathcal{E}}(\underline{\mathsf{Bun}}^{\mathbf{G}}_{X\setminus D}) \times \mathbb{T}_{\mathcal{E}}(\underline{\mathsf{Bun}}^{\mathbf{G}}_{\widehat{Z}})) \ar[r]^-{\rho} & \mathrm{H}^{-1}(\mathbb{T}_{\mathcal{E}}(\underline{\mathsf{Bun}}^{\mathbf{G}}_{\widehat{Z} \setminus D})) \ar[r] & \mathrm{H}^{0}(\mathbb{T}'_{\mathcal{E}})  \to 0 
}
\end{equation}
\begin{equation}\label{codestala} \xymatrix{0 \to \mathrm{H}^{-1}(\mathbb{T}_{F; \mathcal{E}}) \ar[r]  & \mathbb{C}^{12} \oplus \mathbb{C}^{12} \ar[r]^-{g} & \mathrm{H}^{-1}(\mathbb{T}'_{\mathcal{E}}) \ar[r] & \mathrm{H}^{0}(\mathbb{T}_{F; \mathcal{E}})  \to 0 
}
\end{equation}
\begin{equation}\label{quellala} \xymatrix{0 \to \mathrm{H}^{0}(\mathbb{T}'_{\mathcal{E}}) \ar[r]  & \mathrm{H}^{1}(\mathbb{T}_{F;\mathcal{E}}) \ar[r] & \mathbb{C} \oplus \mathbb{C}  \to 0 
}
\end{equation}
In (\ref{questaqua}), we can compute that $\mathrm{H}^{-1}(\mathbb{T}'_{\mathcal{E}}) \simeq \mathrm{ker} (\rho) \simeq \mathbb{C}^{12}$, and that $\mathrm{H}^{0}(\mathbb{T}'_{\mathcal{E}}) \simeq \mathrm{coker}(\rho)$ is infinite dimensional since, e.g. the linearly independent family 
$$\left\{ \begin{pmatrix} \frac{1}{s_0^N}\sum_{n\geq 0} s^n_1 & 0\\
0 & 0
\end{pmatrix}
\right\}_{N\geq 1}$$
of endomorphisms of $\mathbb{C}[[s_0,s_1]][s_0^{-1}]^{\oplus \,2}$ does not belong to the image of $\rho$.\\ Moreover, in (\ref{codestala}) we have that $\mathrm{ker}(g) \simeq \mathbb{C}^{12}$ embedded as the diagonal inside $\mathbb{C}^{12} \oplus \mathbb{C}^{12}$. In fact, obviously the diagonal sits inside $\mathrm{ker}(g)$; on the other hand,  if $(a,b) \in \mathrm{ker}(g) $, and we interpret $a$ and $b$ as maps $\overline{a}, \, \overline{b}: X \to \mathcal{E}^{\vee} \otimes \mathcal{E}$, we have that $\overline{a}_{|X\setminus D}= \overline{b}_{|X\setminus D}$ (since $(a,b) \in \mathrm{ker}(g)$). Now, $X\setminus D$ is schematically dense in $X$ ($D$ being an effective Cartier divisor), and $\mathcal{E}^{\vee} \otimes \mathcal{E}$ is separated, hence $\overline{a}= \overline{b}$, i.e. $a=b$.\\
By putting all these results together, we finally get that $$\mathrm{H}^{-1}(\mathbb{T}_{F; \mathcal{E}}) \simeq \mathrm{H}^{-1}(\mathbb{T}_{ \mathcal{E}}) \simeq \mathbb{C}^{12} \qquad \mathrm{H}^{0}(\mathbb{T}_{F; \mathcal{E}}) \simeq \mathrm{H}^{0}(\mathbb{T}_{ \mathcal{E}}) = 0 \qquad \dim_{\mathbb{C}}\, \mathrm{H}^{1}(\mathbb{T}_{F; \mathcal{E}})= \infty \,, \,\, \mathrm{while} \,\, \mathrm{H}^{1}(\mathbb{T}_{\mathcal{E}})\simeq \mathbb{C}.$$ In particular,
$\underline{\mathsf{Bun}}^{\mathbf{G}}_{F,\mathbb{P}^2}$ is an exotic derived structure on $\mathrm{t}(\underline{\mathsf{Bun}}^{\mathbf{G}}_{\mathbb{P}^2})$. 

\end{eg}

\begin{eg}\label{P2TP2}
Keeping the same notations of Example \ref{P2O3}, we give here an example of a \emph{non-rigid} bundle $\mathcal{E}$ on $X=\mathbb{P}^2$ for which $\mathbb{T}_{F; \mathcal{E}}$ and $\mathbb{T}_{\mathcal{E}}$ are again different, where $F=(D,Z)$ is the same flag of Example \ref{P2O3}. We omit the detailed computations, that however follow the same lines of the previous Example. We consider $\mathcal{E}:= \mathcal{O}_{\mathbb{P}^2} \oplus \mathrm{T}_{\mathbb{P}^2}$ (where $\mathrm{T}_{\mathbb{P}^2}$ is the tangent bundle). Again $\mathrm{H^i}(\mathbb{T}_{\mathcal{E}})= 0$ for $i\neq -1, 0, 1$, and in this case we have $$\mathrm{H^{-1}}(\mathbb{T}_{\mathcal{E}})= \mathrm{Ext}_{\mathbb{P}^2}^0(\mathcal{E},\mathcal{E}) \simeq \mathbb{C}^{10} \qquad \mathrm{H^{0}}(\mathbb{T}_{\mathcal{E}})= \mathrm{Ext}_{\mathbb{P}^2}^1(\mathcal{E}, \mathcal{E}) \simeq \mathbb{C} \qquad \mathrm{H^{1}}(\mathbb{T}_{\mathcal{E}})= \mathrm{Ext}_{\mathbb{P}^2}^2(\mathcal{E},\mathcal{E}) =0.$$ The sequence (\ref{questaqua}) is the same but $\mathrm{H}^{-1}(\mathbb{T}'_{\mathcal{E}}) \simeq \mathrm{ker} (\rho) \simeq \mathbb{C}^{14}$, and $\mathrm{H}^{0}(\mathbb{T}'_{\mathcal{E}}) \simeq \mathrm{coker}(\rho)$ is infinite dimensional since, e.g. the linearly independent family 
$$\left\{ \begin{pmatrix} \frac{1}{s_0^N}\sum_{n\geq 0} s^n_1 & 0 & 0\\
0 & 0 &0\\
0 & 0 &0
\end{pmatrix}
\right\}_{N\geq 1}$$
of endomorphisms of $\mathbb{C}[[s_0,s_1]][s_0^{-1}]^{\oplus \, 3}$ does not belong to the image of $\rho$.\\
The sequences (\ref{codestala}) and (\ref{quellala}) are replaced by the sequence
{\small \begin{equation}\label{questaltra} \xymatrix{0 \to \mathrm{H}^{-1}(\mathbb{T}_{F; \mathcal{E}}) \ar[r]  & \mathbb{C}^{10} \oplus \mathbb{C}^{10} \ar[r]^-{g} & \mathrm{H}^{-1}(\mathbb{T}'_{\mathcal{E}}) \ar[r] & \mathrm{H}^{0}(\mathbb{T}_{F; \mathcal{E}})  \ar[r] & \mathbb{C}\oplus \mathbb{C} \ar[r] & \mathrm{H}^{0}(\mathbb{T}'_{\mathcal{E}}) \ar[r] &  \mathrm{H}^{1}(\mathbb{T}_{F; \mathcal{E}})\to 0
}
\end{equation}}and the same argument used in Example \ref{P2O3} shows that $ \mathrm{H}^{-1}(\mathbb{T}_{F; \mathcal{E}}) \simeq \mathbb{C}^{10}$. Moreover, (\ref{questaltra}) shows that $\mathrm{H}^{1}(\mathbb{T}_{F; \mathcal{E}})$ is infinite dimensional, since $\mathrm{H}^{0}(\mathbb{T}'_{\mathcal{E}})$ is. The upshot is that

{\small $$\mathrm{H}^{-1}(\mathbb{T}_{F; \mathcal{E}}) \simeq \mathrm{H}^{-1}(\mathbb{T}_{ \mathcal{E}}) \simeq \mathbb{C}^{10} \qquad \mathrm{H}^{0}(\mathbb{T}_{F; \mathcal{E}}) \simeq \mathbb{C}^k\, \textrm{for some}\, 4\leq k\leq 6,  \, \,\,\mathrm{H}^{0}(\mathbb{T}_{ \mathcal{E}}) \simeq \mathbb{C} \qquad \dim_{\mathbb{C}}\, \mathrm{H}^{1}(\mathbb{T}_{F; \mathcal{E}})= \infty \,,  \,\, \mathrm{H}^{1}(\mathbb{T}_{\mathcal{E}})=0.$$}
\end{eg}

\begin{rem}\label{explanationtry}
Example \ref{P2TP2} shows that the deformation space associated to the exotic structure $\underline{\mathsf{Bun}}^{\mathbf{G}}_{F,X}$ is not in general a subspace of the non-exotic one, but rather the latter is a quotient of the former. This is a general phenomenon that  can be explained heuristically by observing that the choice of a flag $F$ does not simply put constraints on the usual defomation space, thus selecting a subspace thereof. The deformation space of $\underline{\mathsf{Bun}}^{\mathbf{G}}_{F,X}$ rather describes deformations of \emph{more data} (bundle and trivializations), so the map from this deformation space to the non-exotic deformation space is expected to have non trivial fibers. In other words, the map of derived stacks $\underline{\mathsf{Bun}}^{\mathbf{G}}_{F,X} \to \underline{\mathsf{Bun}}^{\mathbf{G}}_{X}$ is rather a forgetful map, not an inclusion.
\end{rem}

\section{Flag-chiral product and flag-factorization objects}\label{sectionfusion}

In this Section we suggest a replacement of the chiral tensor product (defined for sheaves on the Ran space of a curve, see \cite{Beilinson_Drinfeld_Chiral_2004, FG}) for sheaves or stacks over $\mathsf{Fl}_X$, $X$ being a surface. We call this the \emph{flag-chiral tensor product}. As a consequence (and inspired by \cite{samchiral}), we define flags-analogs of factorization algebras (called \emph{flag-factorization algebras}) and of factorization categories (called \emph{flag-factorization categories}).\\

\subsection{Naive flag-chiral product}\label{warmup} As a warm-up, we define a flag-chiral product on $\flags^{\flat}_X$, where $\flags^{\flat}_X$ is defined as the composition
$$\flags^{\flat}_X: \xymatrix{\Aff\op \ar[rr]^-{\flags_X} && \mathbf{PoSets} \ar[r] & \mathbf{Sets} }$$ $\mathbf{PoSets} \to \mathbf{Sets}$ being the forgetful functor (forgetting the partial order).\\ Let us denote by $\mathsf{Shv}(-)$ a 3-functors formalism\footnote{A 6-functor formalism, i.e. a 3-functors formalism in which all the three functors involved have adjoints, in the same spirit was defined previously in \cite{GaiRozI}, and the reader might well prefer to use this source.} in the sense of \cite[Def. 2.4]{Scholze_6fun}, i.e. a lax symmetric monoidal functor $\mathsf{Corr}(\mathcal{C}_{\mathbf{Sch}_{\mathbb{C}}}, E) \to \mathbf{Cat}_{\infty}$, where $\mathcal{C}_{\mathbf{Sch}_{\mathbb{C}}}$ is an appropriate subcategory of $\mathbf{Sch}_{\mathbb{C}}$, and  $E$ is an appropriate class of maps in $\mathbf{Sch}_{\mathbb{C}}$. Here the word appropriate depends  on the specific 3-functors formalism chosen. An example of particular interest for us will be $\mathcal{C}_{\mathbf{Sch}_{\mathbb{C}}}$ consisting of separated schemes of finite type over $\mathbb{C}$, $E$ all morphisms between them, and  $\mathsf{Shv}(-)= \mathcal{D}-\mathsf{Mod}$ given by $\mathcal{D}$-modules (see \cite[Appendix to Lecture VIII:D-modules]{Scholze_6fun}). The reader will find other examples in the appendices of \cite{Scholze_6fun}.
Since $\flags^{\flat}_{X, \bullet}$ is a $2$-Segal object (Theorem \ref{thm:2SegGood}), $\flags_X^\flat$ is an algebra in $\mathsf{Corr}(\mathcal{C}_{\mathbf{Sch}_{\mathbb{C}}}, E)$ (e.g. for the choices corresponding to $\mathsf{Shv}(-)= \mathcal{D}-\mathsf{Mod}$), hence by composing with $\mathsf{Shv}(-)$, we get a monoidal structure $\otimes^{\flat,\, \mathrm{ch}}$ on $\mathsf{Shv}(\flags^{\flat}_X)$, that will be called the \emph{naive flag-chiral monoidal structure}. In particular, considering the correspondence
$$ \xymatrix{& \flags^{\flat}_{X,2} \ar[rd]^-{p:=\partial^{\flat}_1} \ar[ld]_-{q:=(\partial^{\flat}_2, \partial^{\flat}_0)} & \\ \flags^{\flat}_X \times \flags^{\flat}_X && \flags^{\flat}_X}$$ we have the binary product\footnote{For direct and inverse images, we keep the notations of \cite[Def. 2.4]{Scholze_6fun}.} $$\mathcal{F}\otimes^{\flat,\, \mathrm{ch}} \mathcal{G} = p_{!} q^*(\mathcal{F}\boxtimes \mathcal{G}), \,\,\,\, \mathcal{F}, \mathcal{G} \in \mathsf{Shv}(\flags^{\flat}_X).$$
Though not needed, the union map $\partial_1^{\flat}$ is probably proper for general $X$ (it is, e.g. when $X$ is a smooth cubic surface, see Remark \ref{maybeitsproper}). One can prove that $\otimes^{\flat,\, \mathrm{ch}}$ is actually a \emph{symmetric} monoidal structure on $\mathsf{Shv}(\flags^{\flat}_X)$.

\subsection{Flag-chiral product and flag-factorization algebras}\label{chiralsection} We sketch here a version of the naive flag-chiral product for sheaves on $\flags_X$ (so that maps between flags are also taken into account). The reason for this comes from \cite[3.4.6]{Beilinson_Drinfeld_Chiral_2004} where factorization algebras over a given curve $C$ (defined using the Ran space of $C$) are given an equivalent formulation using the category fibered in posets over $\mathbf{Sch}_{\mathbb{C}}$ of relative effective Cartier divisors on $C/\mathbb{C}$ with maps given by inclusions of the support of the divisors. These maps are crucial for such a comparison. The drawback of the necessity of considering maps between divisors is that the corresponding fibered category is not a stack in groupoids anymore but rather  a stack in categories (namely, in posets), and  to our knowledge, no 3-functors formalism is available for sheaves on such objects. So, we have to take a different approach. Specifically, we follow the ideas, notations and constructions given by Sam Raskin in  \cite{samchiral}, where he also shows (\cite[Prop. 6.19.4 and Rem. 6.19.5]{samchiral}) that his definition recovers the definition of factorization algebras of \cite{FG} (the latter being an $\infty$-categorical elaboration of the original approach in \cite{Beilinson_Drinfeld_Chiral_2004}). \\

We denote by $\mathsf{PreStk}^{\mathrm{lax}}$ the category 
of lax-prestacks on $\dAff_{\mathbb{C}}$ (i.e. of admissible functors $\dAff_{\mathbb{C}}^{\mathrm{op}} \to \mathbf{Cat}$ (\cite[4.14]{samchiral}), and 
$\mathsf{PreStk}_{\mathrm{corr}}^{\mathrm{lax}}$ the category\footnote{Both $\mathsf{PreStk}_{\mathrm{corr}}^{\mathrm{lax}}$ and $\mathsf{PreStk}^{\mathrm{lax}}$ can be seen as categories (i.e. $(\infty, 1)$-categories) or as 2-categories (i.e. $(\infty, 2)$-categories); we will specify when the 2-category structure is used.} of correspondences in $\mathsf{PreStk}^{\mathrm{lax}}$ (\cite[4.28]{samchiral}). Note that, for any lax-prestack $\mathcal{X}$, there is a category $\mathsf{QCoh}(\mathcal{X})$ of quasi-coherent complexes on $\mathcal{X}$, defined as the category of natural transformations $\mathcal{X} \to \mathsf{QCoh}$ (\cite[4.15]{samchiral}). As observed in Remark \ref{Fliscomm}, $\flags_{X}$ is actually a unital commutative (i.e. $E_{\infty}$) algebra in correspondences, and we thus view it as an object, denoted as $\flags^{\mathrm{ch}}_{X}$,  in $\mathsf{CAlg}(\mathsf{PreStk}_{\mathrm{corr}}^{\mathrm{lax}})$. This notation is meant to parallel Raskin's $\mathsf{Ran}_{\mathcal{X}}^{\mathrm{un, ch}}$ in \cite[4.28]{samchiral}. The same is true for $\flags_{X}^\dR$, whose unital commutative algebra structure will be denoted as $\flags^{\dR, \mathrm{ch}}_{X} \in \mathsf{CAlg}(\mathsf{PreStk}_{\mathrm{corr}}^{\mathrm{lax}})$. Having these two commutative algebras $\flags^{\mathrm{ch}}_{X}$ and $\flags^{\dR, \mathrm{ch}}_{X}$ at our disposal, we may consider \emph{multiplicative sheaves $\mathcal{C}$ of categories} over any of them (\cite[Def. 5.21.1]{samchiral}), and, given such a multiplicative sheaf $\mathcal{C}$ of categories, we have a notion of \emph{multiplicative object} in $\mathcal{C}$ (\cite[Def. 5.27.1]{samchiral}). In particular, the sheaves of categories $\mathsf{QCoh}_{\flags^{\mathrm{ch}}_{X}}$ and $\mathsf{QCoh}_{\flags^{\dR, \mathrm{ch}}_{X}}$ both carry a canonical multiplicative structure (\cite[5.21.2]{samchiral}). Therefore, we can give the following

\begin{defin}\label{pre-flagchiralcatsandalgs} A \emph{flag pre-chiral category} on $X$ is a unital multiplicative sheaf of categories on the unital commutative algebra $\flags^{\mathrm{ch}}_{X}$.\\
A \emph{flag pre-chiral crystal category} on $X$ is a unital multiplicative  sheaf of categories on the unital commutative algebra $\flags^{\dR, \mathrm{ch}}_{X}$.\\
A \emph{quasi-coherent flag pre-factorization algebra} on $X$ is a multiplicative object in the flag pre-chiral category $\mathsf{QCoh}_{\flags^{\mathrm{ch}}_{X}}$ on $X$.\\
A \emph{$\mathcal{D}$-module flag pre-factorization algebra} on $X$ is a multiplicative object in the flag pre-chiral crystal category $\mathsf{QCoh}_{\flags^{\dR, \mathrm{ch}}_{X}}$ on $X$.
\end{defin}

We denote by $\mathsf{FlCats}^{\mathrm{ch},\mathrm{pre}}_{X}$ (respectively, $\mathsf{FlCrysCats}^{\mathrm{ch, pre}}_{X}$) the $2$-category of flag pre-chiral categories (resp. of flag pre-chiral crystal categories) on $X$. Analogously, we denote by $\mathsf{FlFact}^{\mathrm{pre}}_{X}$ (respectively, $\mathsf{FlFact}^{\dR, \mathrm{pre}}_{X}$) the category of quasi-coherent flag pre-factorization algebras (resp. $\mathcal{D}$-module flag pre-factorization algebras) on $X$.\\

Unzipping Definition \ref{pre-flagchiralcatsandalgs}, we get the following (partial but more concrete) descriptions. Let us consider the correspondence (i.e. morphism in $\mathsf{PreStk}_{\mathrm{corr}}^{\mathrm{lax}}$)
$$ \xymatrix{& \flags_{X,2} \ar[rd]^-{p:=\partial_1} \ar[ld]_-{q:=(\partial_2, \partial_0)} & \\ \flags_X \times \flags_X && \flags_X}$$ and the canonical unit section (i.e. $0$th degeneracy map of $\flags_{X, \bullet}$)
$\mathsf{e}:=\sigma_0: \mathrm{Spec}\, \mathbb{C} \to \flags_X$, given by the empty flag. Then, a flag pre-chiral category $\mathcal{C}$ on $X$ have, in particular, canonical binary product and unit maps (\cite[5.21.1]{samchiral})
\begin{equation}\label{formulasformultcats} 
\mu_{\mathcal{C}}: q^*(\mathcal{C}\boxtimes \mathcal{C}) \longrightarrow p^*(\mathcal{C}) \,\,\, \mathrm{in}\,\,\, \mathsf{ShvCat}_{\flags_{X,2}}\,,  \qquad \epsilon_{\mathcal{C}}: \mathsf{QCoh}_{\mathrm{Spec}\, \mathbb{C}}=\mathsf{Vect} \longrightarrow \mathsf{e}^*(\mathcal{C}) \,\, \,\mathrm{in}\,\,\, \mathsf{ShvCat}_{\mathrm{Spec}\, \mathbb{C}} = \mathsf{Cat}_{\mathbb{C}}
\end{equation}
(where $\mathsf{Vect}$ denotes the category of complexes of $\mathbb{C}$-vector spaces) which are both \emph{equivalences}, and satisfy appropriate commutativity and associativity conditions. And we have similar maps for $n$-ary products, for any $n$, satisfying appropriate commutativity and associativity conditions. An analogous description is available when $\mathcal{C}$ is a flag pre-chiral crystal category on $X$.\\ 
In the special case of $\mathcal{C}= \mathsf{QCoh}_{\flags^{\mathrm{ch}}_{X}}$, we have $$\mathsf{QCoh}_{\flags^{\mathrm{ch}}_{X}}\boxtimes \mathsf{QCoh}_{\flags^{\mathrm{ch}}_{X}} \simeq \mathsf{QCoh}_{\flags^{\mathrm{ch}}_{X} \times \flags^{\mathrm{ch}}_{X}}\,, \qquad q^*(\mathsf{QCoh}_{\flags^{\mathrm{ch}}_{X} \times \flags^{\mathrm{ch}}_{X}}) \simeq \mathsf{QCoh}_{\flags_{X,2}}\, , \qquad  p^*(\mathsf{QCoh}_{\flags^{\mathrm{ch}}_{X}}) \simeq \mathsf{QCoh}_{\flags_{X,2}}\, ,$$ and the (binary) product and unit maps have the following obvious from  $$\mu_{\mathsf{QCoh}_{\flags^{\mathrm{ch}}_{X}}}= \mathrm{Id} : \mathsf{QCoh}_{\flags_{X,2}} \longrightarrow \mathsf{QCoh}_{\flags_{X,2}}\, , \qquad \epsilon_{\mathsf{QCoh}_{\flags^{\mathrm{ch}}_{X}}}= \mathrm{Id}: \mathsf{Vect} \longrightarrow \mathsf{Vect}.$$
Therefore, if $\mathsf{A}$ is a quasi-coherent flag pre-factorization algebra on $X$, by definition, (\cite[4.19.3, 5.5]{samchiral})
$$\mathsf{A} \in \Gamma(\flags_X, \mathsf{QCoh}_{\flags_{X}})= \mathsf{Hom}_{\mathsf{ShvCat}_{\flags_X}}(\mathsf{QCoh}_{\flags_{X}}, \mathsf{QCoh}_{\flags_{X}})\simeq \mathsf{QCoh}(\flags_{X}):=\mathsf{Nat}(\flags_{X}, \mathsf{QCoh}) \in \mathsf{DGCat}_{\mathrm{cont}}\,, $$
and we have, in particular, canonical binary product and unit maps (\cite[5.27]{samchiral})
\begin{equation}\label{formulasformultobj} 
\mu_{\mathsf{A}}: q^*(\mathsf{A}\boxtimes \mathsf{A}) \longrightarrow p^*(\mathsf{A}) \,\,\, \mathrm{in}\,\,\, \Gamma(\flags_{X,2}, \mathsf{QCoh}_{\flags_{X,2}})=\mathsf{QCoh}(\flags_{X,2})  \,,
\end{equation}
\begin{equation}\label{formulasformultobjbis} \epsilon_{\mathsf{A}}: \mathbb{C} \longrightarrow
 \mathsf{e}^*(\mathsf{A})= \mathsf{A}_{|(\emptyset,\emptyset)} \,\, \,\mathrm{in}\,\,\, \mathsf{Vect}
\end{equation}
which are both \emph{equivalences}, and satisfy appropriate commutativity and associativity conditions. And we have similar maps for $n$-ary products, for any $n$, satisfying appropriate commutativity and associativity conditions. An analogous description is available when $\mathsf{A}$ is $\mathcal{D}$-module flag pre-factorization algebra on $X$. Note that in this case, $\mathsf{A} \in \Gamma(\flags_X, \mathsf{QCoh}_{\flags^{\dR}_{X}})\simeq \mathsf{Nat}(\flags^{\dR}_{X}, \mathsf{QCoh})$.\\

In order to omit the prefix ``pre-'' in Definition \ref{pre-flagchiralcatsandalgs}, i.e. to pass from flag pre-chiral (crystal) categories (respectively, pre-factorization algebras) on $X$ to flag-chiral (crystal) categories (resp., flag-factorization algebras) on $X$, we use the intuition from \cite[3.4.6]{Beilinson_Drinfeld_Chiral_2004} where sheaves of factorization algebras (in the Cartier approach discussed there) are required to be insensitive to the scheme structures of the divisors. This suggests the definition below. \\ Let $\mathcal{C}$ be either a flag pre-chiral category or a flag pre-chiral crystal category on $X$, and $\mathsf{A}$ either a quasi-coherent flag pre-factorization algebra or a $\mathcal{D}$-module flag pre-factorization algebra on $X$. For $S \in \Aff_{\mathbb{C}}$, (i.e. a classical underived\footnote{Note that if $\mathcal{C}$ is a flag pre-chiral crystal category on $X$, or if $\mathsf{A}$ is a $\mathcal{D}$-module flag pre-factorization algebra on $X$, we may allow any $S \in \dAff_{\mathbb{C}}$, since by definition $\flags^{\dR}_X(S)= \flags_X (\mathrm{t}_0(S)_{\mathrm{red}}).$} affine Noetherian scheme), and a morphism $\alpha : F_1 =(D_1, Z_1) \to (D_2, Z_2)=F_2$ a morphism in either $\flags_X(S)$ or in $\flags^{\dR}_X(S)$, we have an induced functor $\mathcal{C}_S(\alpha): \mathcal{C}_S(F_1) \to \mathcal{C}_S(F_2)$ and a morphism $\mathsf{A}_S(\alpha):\mathsf{A}_S(F_1) \to \mathsf{A}_S(F_2)$. We say that the morphism $\alpha$ is a \emph{reduced isomorphism} if the induced maps on the underlying reduced subschemes $(D_1)_{\mathrm{red}} \to (D_2)_{\mathrm{red}} $ and $(Z_1)_{\mathrm{red}} \to (Z_2)_{\mathrm{red}}$ are both isomorphisms.

\begin{defin}\label{flagchiralcatsandalgs}  A flag pre-chiral category (respectively, a flag pre-chiral crystal category) $\mathcal{C}$ on $X$ is a \emph{flag-chiral category} (resp. a \emph{flag-chiral crystal category}) on $X$, if for any (underived) $S \in \Aff_{\mathbb{C}}$, and any reduced isomorphism $\alpha$ in $\flags_X(S)$ (resp. in $\flags^{\dR}_X(S)$), the induced functor $\mathcal{C}_S(\alpha)$ is an equivalence.\\
A quasi-coherent flag pre-factorization algebra (respectively, a $\mathcal{D}$-module flag pre-factorization algebra) $\mathsf{A}$ on $X$ is a \emph{quasi-coherent flag-factorization algebra} (resp., a $\mathcal{D}$-module flag-factorization algebra) if for any $S \in \dAff_{\mathbb{C}}$, and any reduced isomorphism $\alpha$ in $\flags_X(S)$ (resp. in $\flags^{\dR}_X(S)$), the induced morphism $\mathsf{A}_S(\alpha)$ is an equivalence.\\
The full sub-2-category of $\mathsf{FlCats}^{\mathrm{ch},\mathrm{pre}}_{X}$ (respectively, of $\mathsf{FlCrysCats}^{\mathrm{ch, pre}}_{X}$) whose objects are flag-chiral categories (resp., flag-chiral crystal categories) is denoted by $\mathsf{FlCats}^{\mathrm{ch}}_{X}$ (resp., $\mathsf{FlCrysCats}^{\mathrm{ch}}_{X}$).\\
The full subcategory of $\mathsf{FlFact}^{\mathrm{pre}}_{X}$ (respectively, of $\mathsf{FlFact}^{\dR, \mathrm{pre}}_{X}$) whose objects are quasi-coherent flag-factorization algebras (resp., $\mathcal{D}$-module flag-factorization algebras) is denoted by $\mathsf{FlFact}_{X}$ (respectively, $\mathsf{FlFact}^{\dR}_{X}$).
\end{defin}
 
\begin{eg}\label{exOFl} Obviously $\mathcal{O}_{\flags_X}$ (respectively, $\mathcal{O}_{\flags^{\dR}_X}$) is an object in $\mathsf{FlFact}_{X}$ (respectively, $\mathsf{FlFact}^{\dR}_{X}$). Note that for any test affine underived  scheme $S$, and any map $F_1 \to F_2$ in $\flags_X(S)$ (respectively, in $\flags^{\dR}_X(S)$), the induced map (which is part of the data defining $\mathcal{O}$, see \cite[(4.16.1)]{samchiral}) $(\mathcal{O}_{\flags_X})_S (F_1)= \mathcal{O}_S \to (\mathcal{O}_{\flags_X})_S (F_2)= \mathcal{O}_S$ is the identity (resp., the induced map $(\mathcal{O}_{\flags^{\dR}_X})_S (F_1)= \mathcal{O}_{S_{\mathrm{red}}} \to (\mathcal{O}_{\flags^{\dR}_X})_S (F_2)= \mathcal{O}_{S_{\mathrm{red}}}$ is the identity.)
\end{eg}
 
By Remark \ref{Griscomm}, the 2-Segal objects $\mathcal{G}r_{X, \bullet}$ and $\mathcal{G}r^{\dR}_{X, \bullet}$ give rise to commutative algebras in $\mathsf{PreStk}_{\mathrm{corr}}^{\mathrm{lax}}$, that will be denoted, respectively, as $\mathcal{G}r^{\mathrm{ch}}_{X}$ and $\mathcal{G}r^{\dR, \mathrm{ch}}_{X}$. So, the same arguments used above for $\flags^{\mathrm{ch}}_X$ and $\flags^{\dR, \mathrm{ch}}_X$, allow for a definition of \emph{$\mathsf{Gr}$-chiral categories} and of \emph{$\mathsf{Gr}$-chiral crystal categories}.\\ Moreover, as the sheaves of categories $\mathsf{QCoh}_{\mathcal{G}r^{\mathrm{ch}}_{X}}$ and $\mathsf{QCoh}_{\mathcal{G}r^{\dR, \mathrm{ch}}_{X}}$ are, as always, multiplicative and unital sheaves of categories on, respectively, $\mathcal{G}r^{\mathrm{ch}}_{X}$ and $\mathcal{G}r^{\dR, \mathrm{ch}}_{X}$, we also have the notions of \emph{quasi-coherent $\mathsf{Gr}$-factorization algebra} and of  \emph{$\mathcal{D}$-module $\mathsf{Gr}$-factorization algebra}. A more detailed investigations of these objects is deferred to a future paper.\\

\begin{eg}\label{exOGrandpushfwd} As in Example \ref{exOFl}, $\mathcal{O}_{\mathcal{G}r^{\mathrm{ch}}_{X}}$ (respectively, $\mathcal{O}_{\mathcal{G}r^{\dR, \mathrm{ch}}_{X}}$) is a quasi-coherent $\mathsf{Gr}$-factorization algebra (respectively, a $\mathcal{D}$-module $\mathsf{Gr}$-factorization algebra). Moreover, if $p: \mathcal{G}r^{\mathrm{ch}}_{X} \to \flags^{\mathrm{ch}}_{X}$ (respectively $p^{\dR}: \mathcal{G}r^{\dR, \mathrm{ch}}_{X} \to \flags^{\dR, \mathrm{ch}}_{X}$) denotes the projection maps, both are maps of commutative algebras in $\mathsf{PreStk}_{\mathrm{corr}}^{\mathrm{lax}}$, and $p_* (\mathcal{O}_{\mathcal{G}r^{\mathrm{ch}}_{X}})$ (resp., $(p^{\dR})_*(\mathcal{O}_{\mathcal{G}r^{\dR, \mathrm{ch}}_{X}})$) is an object in $\mathsf{FlFact}_{X}$ (respectively, in $\mathsf{FlFact}^{\dR}_{X}$).
\end{eg}

\begin{rem}\label{classicalcomparison} When $X$ is a curve, the definitions \ref{pre-flagchiralcatsandalgs} and \ref{flagchiralcatsandalgs} admit obvious analogs, where $\flags_X$ is replaced by the lax-prestack $\underline{\mathsf{Car}}_{X}$ (of relative effective Cartier divisors on $X$, with morphisms given by inclusions of Cartier). We expect that, in this version, Definition \ref{flagchiralcatsandalgs} recovers Raskin's definitions of chiral category and of factorization algebras (\cite[Def. 6.2.1]{samchiral}).
\end{rem}

\begin{rem}\label{GrXasflagfactspace} One can adapt the formalism of \cite{samchiral} in order to define the notion of a \emph{flag-factorization (crystal) space} over $X$, in such a way that the flag-Grassmannian $\mathcal{G}r_X$ is such a factorization (crystal) space. We leave the details to the interested reader.
\end{rem}

\appendix

\section{Extensions into simplicial objects}\label{BenjExtension}
In this appendix, we will prove the technical \cref{prop:extendsimplicial}, used in \cref{subsec:2segalflaggrass}.

\begin{prop}\label{prop:extendsimplicial}
 Let $\cC$ be an $\infty$-category with all countable limits. Let $X_\bullet \colon \Delta\op \to \cC$ be a simplicial object.
 Denote by $\Delta_{\leq 2}$ the full subcategory of $\Delta$ spanned by $[0]$, $[1]$ and $[2]$. 
 Let $Y^{\leq 2}_\bullet \colon \Delta_{\leq 0}\op \to \cC$ and let $p_{\bullet}^{\leq 2} \colon Y_{\bullet}^{\leq 2} \to X_{|\Delta_{\leq 2}}$ be a morphism (of truncated simplicial objects).
 Assume that
 \begin{enumerate}[label={(\alph*)}, ref={assumption (\alph*)}]
  \item \label{assumptionlvl0} The morphism $Y_0 \to X_0$ is an equivalence;
  \item \label{assumptionpullback} The square 
  \[
   \begin{tikzcd}
    Y_2 \ar{r}{\partial_1} \ar{d} \ar[phantom]{dr}[description, near start]{\lrcorner} & Y_1 \ar{d} \\ X_2 \ar{r}[swap]{\partial_1} & X_1
   \end{tikzcd}
  \]
  is Cartesian.
 \end{enumerate}
 Then $Y_\bullet^{\leq 2}$ and $p_\bullet^{\leq 2}$ extend naturally into a morphism of simplicial objects $p_\bullet \colon Y_\bullet \to X_\bullet$ such that:
 \begin{enumerate}[ref={condition (\arabic*)}]
  \item\label{cond:deltaint} For any $n \geq 2$, the following commutative square is Cartesian:
  \[
  \begin{tikzcd}
    Y_n \ar{r}{\delta_Y^n} \ar{d}[swap]{p_n} \ar[phantom]{dr}[description, near start]{\lrcorner} & Y_1 \ar{d}{p_1} \\
    X_n \ar{r}[swap]{\delta_X^n} & X_1,
  \end{tikzcd}
  \]
  where $\delta_X^n$ and $\delta_Y^n$ are the images by $X$ and $Y$ of the map $\ell_n \colon [1] \to [n]$ sending $0$ to $0$ and $1$ to $n$.
  \item The morphism $p_\bullet$ is relatively $2$-Segal. In particular, if $X_\bullet$ is $2$-Segal, so is $Y_\bullet$.
 \end{enumerate}
\end{prop}

\begin{proof}
\newcommand{\interior}{\mathrm{int}}
\newcommand{\Deltaint}{\Delta^\interior}
 Denote by $\Deltaint$ the subcategory of $\Delta$ containing every object, but only morphisms $f \colon [n] \to [m]$ satisfying $f(0) = 0$ and $f(n) = m$. We also set $\Deltaint_{\leq 2} := \Deltaint \cap \Delta_{\leq 2}$.
 Note that $\ell_n$ is the only morphism $[1] \to [n]$ in $\Deltaint$.
 The restriction of the morphism $p_\bullet^{\leq 2}$ to $\Deltaint_{\leq 2}$ can be represented as a diagram indexed $p' \colon D \to \cC$ by the pullback
 \[
  \begin{tikzcd}
   D \ar{r} \ar{d} \ar[phantom]{dr}[description, near start]{\lrcorner} & (\Deltaint)^I \ar{d}{\mathrm{ev}_0} \\
    \Deltaint_{\leq 2} \ar{r} & \Deltaint
  \end{tikzcd}
 \]
 where $I = \Delta^1$ is the interval category.
 We define $p_\bullet^\interior$ as the right Kan extension of $p'$ along the inclusion $D \to (\Deltaint)^I$.
 By our assumptions (a) and (b), $p'$ is itself a right Kan extension its restriction to $\Deltaint_{\leq 1} \times_{\Deltaint} (\Deltaint)^I$. As a consequence, the morphism $p_\bullet^\interior$ satisfies \ref{cond:deltaint}.
 
 Notice that $p_\bullet^\interior$ already defines  all the degeneracy maps as well as the face morphisms $Y_n \to Y_{n_1}$ except $\partial_0$ and $\partial_n$.
 In particular, in order to extend $p_\bullet^\interior$ to a fully fledge morphism of simplicial objects, we only need to construct \emph{in a canonical way} the face morphisms $\partial_0$ and $\partial_n$.
 To do so, we need the following trivial observation. \vspace{1em}
 
 Observation:
 Let $\epsilon_i \colon [n-1] \to [n]$ be the morphism in $\Delta$ whose image is $[n] \smallsetminus \{i\}$ (so that $\partial_i$ of a simplicial object is the image of $\epsilon_i$)
 For $i = 0$ or $n$, there exists a unique morphism $g_i \colon [2] \to [n]$ in $\Deltaint$ such that the following diagrams commute
 \[
 \begin{array}{p{3em}p{18em}cp{18em}p{3em}}
 i = 0 & \centering
  \begin{tikzcd}[ampersand replacement=\&]
   {[1]} \ar{r}{\epsilon_0} \ar{d}[swap]{\ell_{n-1}} \& {[2]} \ar{d}{g_0} \\ {[n-1]} \ar{r}[swap]{\epsilon_0} \& {[n]}.
  \end{tikzcd}
  &\centering
  \vline
  &\centering
  \begin{tikzcd}[ampersand replacement=\&]
   {[1]} \ar{r}{\epsilon_2} \ar{d}[swap]{\ell_{n-1}} \& {[2]} \ar{d}{g_n} \\ {[n-1]} \ar{r}[swap]{\epsilon_n} \& {[n]}.
  \end{tikzcd}
  & \hfill i = n
 \end{array}
 \]
 Indeed, the only possible choice is $g_i(0) = 0$, $g_i(2) = n$ and $g_0(1) = 1$ or $g_n(1) = n-1$.
 As a consequence, we can canonically form the plain diagrams
 \[
 \begin{array}{p{3em}p{18em}cp{18em}p{3em}}
 i = 0 & \centering
  \begin{tikzcd}[nodes=overlay, column sep={4em,between origins}, row sep={2.5em,between origins}, ampersand replacement=\&]
   Y_n \ar{rrr}{Y(g_0)} \ar{ddd} \ar[dashed]{dr}{\partial_0} \&\&\& Y_2 \ar{ddd} \ar{dr}{\partial_0} \\
   \& Y_{n-1} \ar[crossing over]{rrr}[near start]{\delta_Y^{n-1}} \ar[phantom]{dddrrr}[description, very near start]{\resizebox{1em}{!}{$\lrcorner$}}  \&\&\& Y_1 \ar{ddd} \\ \\
   X_n \ar{rrr}[near end]{X(g_0)} \ar{dr}[swap]{\partial_0} \&\&\& X_2 \ar{dr}{\partial_0} \\
   \& X_{n-1} \ar{rrr}[swap]{\delta_X^{n-1}} \ar[from=uuu, crossing over] \&\&\& X_1
  \end{tikzcd}
  &\centering
  \vline
  &\centering
  \begin{tikzcd}[nodes=overlay, column sep={4em,between origins}, row sep={2.5em,between origins}, ampersand replacement=\&]
   Y_n \ar{rrr}{Y(g_n)} \ar{ddd} \ar[dashed]{dr}{\partial_n} \&\&\& Y_2 \ar{ddd} \ar{dr}{\partial_n} \\
   \& Y_{n-1} \ar[crossing over]{rrr}[near start]{\delta_Y^{n-1}} \ar[phantom]{dddrrr}[description, very near start]{\resizebox{1em}{!}{$\lrcorner$}}  \&\&\& Y_1 \ar{ddd} \\ \\
   X_n \ar{rrr}[near end]{X(g_n)} \ar{dr}[swap]{\partial_n} \&\&\& X_2 \ar{dr}{\partial_n} \\
   \& X_{n-1} \ar{rrr}[swap]{\delta_X^{n-1}} \ar[from=uuu, crossing over] \&\&\& X_1
  \end{tikzcd}
  & \hfill i = n
 \end{array}
 \]
 and there exist essentially unique dashed arrows completing the diagrams, as the front squares are Cartesian.
 It follows that $p_\bullet^\interior$ extends to a morphism of simplicial objects $p_\bullet \colon Y_\bullet \to X_\bullet$.
 
 Remains to prove that the projection $p_\bullet$ is relatively $2$-Segal.
  For $Z$ a simplicial object (actually, either $X$ or $Y$), for any integer $n \geq 2$, and any $0 < j < n$, we consider the following commutative diagrams
 \[
 \begin{array}{p{3em}p{18em}cp{18em}p{3em}}
 i = 0 & \centering
  \begin{tikzcd}[ampersand replacement=\&]
  Z_{n+1} \ar{dr}{q^0_j} \ar[bend left]{drr}{\partial_j} \ar[bend right]{ddr}[swap]{\partial_0} \\
   \& Z_n^{(j),0} \ar{d} \ar{r}{p_j^0} \ar[phantom]{dr}[description, near start]{\lrcorner} \& Z_n \ar{d}{\partial_0} \\
   \& Z_n \ar{r}[swap]{\partial_j} \& Z_{n-1}
  \end{tikzcd}
  &\centering
  \vline
  &\centering
  \begin{tikzcd}[ampersand replacement=\&]
  Z_{n+1} \ar{dr}{q^n_j} \ar[bend left]{drr}{\partial_j} \ar[bend right]{ddr}[swap]{\partial_{n+1}} \\
   \& Z_n^{(j),n} \ar{d} \ar{r}{p_j^n} \ar[phantom]{dr}[description, near start]{\lrcorner} \& Z_n \ar{d}{\partial_n} \\
   \& Z_n \ar{r}[swap]{\partial_i} \& Z_{n-1}
  \end{tikzcd}
  & \hfill i = n
 \end{array}
 \]
 where $Z_n^{(i),0}$, $Z_n^{(i),n}$, $q_i^0$ and $q_i^n$ are defined by the squares being Cartesian.

 Notice then that, for any $0 < j < n$, in the commutative diagram\\

 \[
  \begin{tikzcd}
   Y_n \ar{r}[swap]{\partial_j} \ar{d} \ar[bend left=2em]{rr}{\delta_Y^n} & Y_{n-1} \ar{d} \ar{r}[swap]{\delta_Y^{n-1}} & Y_1 \ar{d} \\
   X_n \ar{r}{\partial_j} \ar[bend right=2em]{rr}[swap]{\delta_X^n} & X_{n-1} \ar{r}{\delta_X^{n-1}} & X_1,
  \end{tikzcd}
 \]
 both the rightmost square and the outer rectangle are Cartesian. It follows that the leftmost square is also Cartesian.
 In particular, in the commutative diagram, for $i = 0, n$ and $0 < j < n$,\\
 
 \[
   \begin{tikzcd}[nodes=overlay, column sep={4em,between origins}, row sep={2.5em,between origins}]
   Y_n^{(j),i} \ar{rrr}{p_j^i} \ar{ddd} \ar{dr} &&& Y_n \ar{ddd} \ar{dr}{\partial_i} \\
   & Y_{n} \ar[crossing over]{rrr}[near start]{\partial_j} &&& Y_{n-1} \ar{ddd} \\ \\
   X_n^{(j),i} \ar{rrr}[near end]{p_j^i} \ar{dr} &&& X_n \ar{dr}{\partial_i} \\
   & X_{n} \ar{rrr}[swap]{\partial_j} \ar[from=uuu, crossing over] &&& X_{n-1},
  \end{tikzcd}
 \]
 
 \bigskip
 
\noindent the upper, bottom and front squares are Cartesian. It follows that so is the back square.
 Finally, we can form the commutative diagram
 \[
  \begin{tikzcd}
    Y_{n+1} \ar{r}[swap]{q_j^i} \ar{d} \ar[bend left=2em]{rr}{\partial_j} & Y_n^{(j),i} \ar{d} \ar{r}[swap]{p_j^i} & Y_n \ar{d} \\
    X_{n+1} \ar{r}{q_j^i} \ar[bend right=2em]{rr}[swap]{\partial_j} & X_n^{(j),i} \ar{r}{p_j^i} & X_n.
  \end{tikzcd}
 \]
 By what precedes, both the rightmost square, and the outer rectangle are Cartesian. In particular, so is the leftmost square. This being true for $i = 0,n$ and $0 < j < n$, this is precisely the relative $2$-Segal condition on the projection $p_\bullet$.
\end{proof}

\section{Actions}\label{sectionactions}
In this Appendix we construct some natural actions of appropriate flags analogs of the loop and positive loop group on the flag Grassmannian. In order to simplfy the exposition, we work in the underived setting: all test affine schemes and all stacks (namely the stack of $\mathbf{G}$-bundles) will be implicitly \emph{underived}. Generalizing the constructions to the derived setting would pose no problems but would increase the length (and possibly the readability) of this section. Though we focus on the stack of $\mathbf{G}$-bundles, as in the main text, similar results hold for the stack of perfect complexes together with the choice of a fixed perfect complex $\mathcal{E}_0$ on $X$ (replacing the trivial $\mathbf{G}$-bundle).\\

\subsection{Flags-version of the (positive) loop group} Let $X$ be our smooth projective surface and $S$ be an arbitrary test scheme in $\Aff$. As usual, we will sometimes write $Y$ for $X \times S$, and $\mathcal{E}_0$ will denote the trivial $\mathbf{G}$-bundle (we will leave the context to specify which space $\mathcal{E}_0$ actually lives on).\\

\begin{defin}\label{defin:autgroup-fixedflag}
Let $F=(D,Z) \in \flags_X(S)$, $Y:=X \times S$. We define 
\begin{itemize}
\item the sheaf of groups over $Y$ $$\mathsf{LG}_X(S,F):= * \times_{\mathbf{Bun}_{\widehat{D}\smallsetminus D} \times_{\mathbf{Bun}_{(\widehat{D}\smallsetminus D) \times_{Y} \wh{Z}^{\affinize}}} \mathbf{Bun}_{\wh{Z}}}\,\, *$$
where the two maps from $*$ coincide and are given by the trivial $\mathbf{G}$-bundle $\mathcal{E}_0$. In other words, $\mathsf{LG}_X(S,F)$ is the loop stack $\Omega_{\mathcal{E}_0}(\mathbf{Bun}_{\widehat{D}\smallsetminus D} \times_{\mathbf{Bun}_{(\widehat{D}\smallsetminus D) \times_{Y} \wh{Z}^{\affinize}}} \mathbf{Bun}_{\wh{Z}})$ at the point corresponding to the trivial $\mathbf{G}$-bundle.
\item the group
$$LG_X(S,F):= \mathrm{Aut}_{\mathrm{Bun}(\widehat{D}\smallsetminus D)}(\mathcal{E}_0)\times_{\mathrm{Aut}_{\mathrm{Bun}((\widehat{D}\smallsetminus D) \times_{Y} \Zhataff)}(\mathcal{E}_0)} \mathrm{Aut}_{\mathrm{Bun}(\hZ)}(\mathcal{E}_0)$$
i.e. $LG_X(S,F)=\Gamma(X \times S, \mathsf{LG}_X(S,F))$.
\item the sheaf of groups over $Y$ $$\mathsf{L^{+}G}_X(S,F):= * \times_{\mathbf{Bun}_{\widehat{D} }} \,\, *$$ i.e. the loop stack $\Omega_{\mathcal{E}_0}(\mathbf{Bun}_{\widehat{D}})$ of $\mathbf{Bun}_{\widehat{D}}$ at the point corresponding to the trivial $\mathbf{G}$-bundle.
\item the group $L^{+}G_X(S,F) $ of global sections of $\mathsf{L^{+}G}_X(S,F)$, i.e. $L^{+}G_X(S,F) = \mathrm{Aut}_{\mathrm{Bun}(\widehat{D})}(\mathcal{E}_0)$
\end{itemize}
\end{defin}

\begin{rem}\label{mappp} Consider the obvious projection $\GrX_X(S)(F) \to \mathbf{Bun}_{\widehat{D}}$ (see Lemma \ref{lem:benjexo}). The fiber of this map at the trivial bundle map $* \to \mathbf{Bun}_{\widehat{D}}$ coincides with the sheaf $\mathsf{LG}_X(S,F)$, and the induced map $\mathsf{LG}_X(S)(F) \to \GrX_X(S)(F)$ of stacks over $X \times S$ is a section of the canonical projection $\GrX_X(S)(F) \to \mathsf{LG}_X(S)(F)$. Moreover, there is an obvious morphism $\mathsf{L^{+}G}_X(S,F) \to \mathsf{LG}_X(S,F)$ of sheaves of groups, given by restricting the trivializations. \end{rem}

\begin{rem}\label{comparisonwithG} Since we are dealing with $\mathbf{G}$-principal bundles, we have $$\mathsf{LG}_X(S,F) \simeq \mathbf{G}_{\widehat{D}\smallsetminus D, X\times S} \times_{\mathbf{G}_{(\widehat{D}\smallsetminus D) \times_{X \times S} \Zhataff, X\times S}} \mathbf{G}_{\hZ,X\times S}$$ and
$$\mathsf{L^{+}G}_X(S,F) \simeq \mathbf{G}_{\widehat{D},X\times S}\,\,,$$
where for a map $V \to W$ ($V$ being allowed to be a fiber functor on $W$), we write $\mathbf{G}_{V,W}$ for the functor
$$\mathbf{G}_{V,W}: \Aff_{W}\op \to \mathbf{Grp}\,, \,\,\, (T \to W) \longmapsto \mathrm{Hom}_{\mathbf{Sch}_{\mathbb{C}}}(T\times_{W} V, \mathbf{G})= \mathrm{Hom}_{\mathbf{Sch}_{W}}(T\times_{W} V, \mathbf{G}\times W).$$
\end{rem}

\begin{rem}\label{actiononGr} Quite generally, given $\mathcal{X} \to \mathcal{Y}$, and $y: * \to \mathcal{Y}$ in $\St_T$, the loop stack $\Omega_y\mathcal{Y}$ acts naturally on the fiber $\mathcal{X}_y=\mathcal{X}\times_{f,\mathcal{Y},y} \, *$. In particular, the sheaf in groups $\mathsf{LG}_X(S,F)$ acts\footnote{In other words, the data we give below make $\GrX_X(S)(F)$ into a lax presheaf of groupoids on the classifying prestack $\mathrm{B}\mathsf{LG}_X(S)(F)$. See, e.g. \cite{romagny_actions} for generalities about actions of (pre)sheaves of groups on (pre)stacks} on the stack $\GrX_X(S)(F)$  by changing the trivialisations. Namely, if $(T \to X\times S) \in \Aff_Y$, we have 
$$\mathsf{LG}_X(S,F)(T)= \{ (\varphi, \psi) \, | \varphi \in \mathrm{Aut}_{\mathrm{Bun}((\widehat{D}\smallsetminus D) \times_Y T)}(\mathcal{E}_0)\,, \, \psi \in  \mathrm{Aut}_{\mathrm{Bun}(\hZ\times_Y T)}(\mathcal{E}_0) \, ,\, \varphi |_{(\widehat{D}\smallsetminus D) \times_{Y} \hZ \times_Y T}= \psi |_{(\widehat{D}\smallsetminus D) \times_{Y} \hZ \times_Y T}     \}$$ 
while 
{\small $$\Gr_X(S)(F)(T)=\{ (\mathcal{E}, \phi, \chi) \, |\, \mathcal{E} \in \mathrm{Bun}(\widehat{D}\times_Y T), \phi: \mathcal{E}|_{(\widehat{D}\smallsetminus D) \times_Y T)} \simeq \mathcal{E}_0\, , \,   \chi: \mathcal{E}|_{\hZ \times_Y T} \simeq \mathcal{E}_0 \, ,\, \phi |_{(\widehat{D}\smallsetminus D) \times_Y \hZ\times_Y T} =\chi |_{(\widehat{D}\smallsetminus D) \times_Y \hZ\times_Y T}      \} $$}
(with obvious morphisms). The (left) action of $\mathsf{LG}_X(S,F)(T)$ on $\GrX_X(S)(F)(T)$ is then given, on objects, by $$(\varphi, \psi)\cdot (\mathcal{E}, \phi, \chi) := (\mathcal{E}, \varphi \circ \phi, \psi \circ \chi).$$ 
This action should be considered as a higher dimensional analog of the well known action of $\mathbf{G}_{\mathcal{K}_x}$ on the local affine Grassmannian at a point $x$ inside a curve (see \cite[4.5]{BD_Hitchin} or  \cite[2]{MV_2007}).
As a consequence of the above action, passing to global sections, the group $G_X(S,F)$ acts on the groupoid $\mathrm{Gr}_X(S)(F)$ (again by change of trivializations, with similar formulas). 
\end{rem}

\begin{rem}\label{GcaractiononGr}
$\mathsf{L^{+}G}_X(S,F)$ acts on the stack $\GrX_X(S)(F)$, by changing trivializations, the action being induced by the $\mathsf{LG}_X(S,F)$-action of Remark \ref{actiononGr} together with the map $\mathsf{L^{+}G}_X(S,F) \to \mathsf{LG}_X(S,F)$ of Remark \ref{mappp}. This action should be considered as a higher dimensional analog of the well known action of $\mathbf{G}_{\mathcal{O}_x}$ on the local affine Grassmannian at a point $x$ inside a curve (see \cite[4.5]{BD_Hitchin} or  \cite[2]{MV_2007}).\\
As a consequence of the above action, passing to global sections, the group $L^{+}G_X(S,F)$ acts on the groupoid $\mathrm{Gr}_X(S)(F)$.
\end{rem}

%

\begin{prop}\label{groupfunctonflags} For any $S \in \Aff$, we have functors
\begin{enumerate} 
\item $$\mathsf{L^{+}G}_X(S,-): \flags_X(S) \longrightarrow \mathbf{ShGrp}\op_{X\times S}$$
\item $$\mathsf{LG}_X(S,-):\flags_X(S) \longrightarrow \mathbf{ShGrp}^{\mathrm{corr}}_{X\times S}$$
\item and a natural transformation $\mathsf{L^{+}G}_X(S,-) \to \mathsf{LG}_X(S,-)$ (see Remark \ref{mappp}).
\end{enumerate}
\end{prop}

\begin{proof} To ease notation, we put $Y=X\times S$. Let $i: F_1=(D_1,Z_1) \to (D_2 , Z_2)=F_2$ be a morphism in $\flags_X(S)$. Since $\widehat{D}_1 \hookrightarrow \widehat{D}_2$, (1) is obvious. In order to prove (2), we observe that restrictions along the maps $\widehat{D}_1 \smallsetminus D_1 \to \widehat{D}_2 \smallsetminus D_1$, $\hZ_1 \to \hZ_2$, $(\widehat{D}_2 \smallsetminus D_2) \cap \hZ_2 \to  (\widehat{D}_2 \smallsetminus D_1) \cap \hZ_2$, $(\widehat{D}_1 \smallsetminus D_1) \cap \hZ_1 \to (\widehat{D}_2 \smallsetminus D_1) \cap \hZ_2$ yield maps $$\mathbf{Bun}_{\widehat{D}_2 \smallsetminus D_1} \to \mathbf{Bun}_{\widehat{D}_1 \smallsetminus D_1}\,\,, \,\,\,
\mathbf{Bun}_{\hZ_2 } \to \mathbf{Bun}_{\hZ_1 }$$ $$\mathbf{Bun}_{(\widehat{D}_2 \smallsetminus D_1) \cap \hZ_2} \to \mathbf{Bun}_{(\widehat{D}_2 \smallsetminus D_2) \cap \hZ_2}\,\,, \,\,\,
\mathbf{Bun}_{(\widehat{D}_2 \smallsetminus D_1) \cap \hZ_2} \to \mathbf{Bun}_{(\widehat{D}_1 \smallsetminus D_1) \cap \hZ_1}.$$ Taking loops at $\mathcal{E}_0$ of the previous maps (and observing that taking loops at $\mathcal{E}_0$ commutes with fiber products of stacks over $X\times S$), we get a correspondence in  $\mathbf{ShGrp}_{X\times S}$ 
\begin{equation}\label{corrLG}
\begin{tikzcd}[column sep={12em,between origins}]
& \makebox[0pt][r]{$\mathsf{H}_{X}(S)(F_1,F_2):={}$} \displaystyle \Omega_{\mathcal{E}_0} \left(\mathbf{Bun}_{\widehat{D}_2 \smallsetminus D_1}\newtimes_{\mathbf{Bun}_{(\widehat D_2 \smallsetminus D_1) \cap \hZ_2}} \mathbf{Bun}_{\hZ_2} \right) \ar[rd] \ar[ld] & \\
\displaystyle \Omega_{\mathcal{E}_0} \left(\mathbf{Bun}_{\widehat D_1 \smallsetminus D_1}\newtimes_{\mathbf{Bun}_{(\widehat D_1 \smallsetminus D_1) \cap \hZ_1}} \mathbf{Bun}_{\hZ_1} \right) & & \displaystyle \Omega_{\mathcal{E}_0} \left(\mathbf{Bun}_{\widehat D_2 \smallsetminus D_2}\newtimes_{\mathbf{Bun}_{(\widehat D_2 \smallsetminus D_2) \cap \hZ_2}} \mathbf{Bun}_{\hZ_2} \right)
\end{tikzcd}
\end{equation}
i.e. a morphism\footnote{Recall that in the category of correspondences in a category with fiber products, a morphism from a given source to a given target is the same thing as a morphism from the given target to the given source. } $\mathsf{LG}_X(S,i): \mathsf{LG}_X(S,F_1) \to \mathsf{LG}_X(S,F_2)$ in $\mathbf{ShGrp}^{\mathrm{corr}}_{X\times S}$. We leave to the reader verifying the existence of canonical compatibility isomorphisms\footnote{Note that $\mathbf{ShGrp}^{\mathrm{corr}}_{X\times S}$ is a 2-category in the usual way.}  $\mathsf{LG}_X(S,i\circ j) \simeq \mathsf{LG}_X(S,i) \circ \mathsf{LG}_X(S,j)$  for composeable morphisms $i, j$ in $\flags_X(S)$, i.e. that if $\xymatrix{(D_1, Z_1) \ar[r]^-{i} & (D_2, Z_2) \ar[r]^-{j} & (D_3, Z_3)}$ are morphisms in $\flags_X(S)$, the following diagram is cartesian 
$$\xymatrix{\Omega_{\mathcal{E}_0} \, \mathbf{Bun}_{\widehat{D_3}\smallsetminus D_1  } \times_{\mathbf{Bun}_{\widehat{D_3}\smallsetminus D_1  \cap \hZ_3 }} \mathbf{Bun}_{\widehat{Z_3}} \ar[r] \ar[d] & \Omega_{\mathcal{E}_0} \, \mathbf{Bun}_{\widehat{D_3}\smallsetminus D_2  } \times_{\mathbf{Bun}_{\widehat{D_3}\smallsetminus D_2  \cap \hZ_3 }} \mathbf{Bun}_{\widehat{Z_3}} \ar[d] \\
\Omega_{\mathcal{E}_0} \, \mathbf{Bun}_{\widehat{D_2}\smallsetminus D_1  } \times_{\mathbf{Bun}_{\widehat{D_2}\smallsetminus D_1  \cap \hZ_2 }} \mathbf{Bun}_{\widehat{Z_2}} \ar[r] & \Omega_{\mathcal{E}_0} \, \mathbf{Bun}_{\widehat{D_2}\smallsetminus D_2  } \times_{\mathbf{Bun}_{\widehat{D_2}\smallsetminus D_2  \cap \hZ_2 }} \mathbf{Bun}_{\widehat{Z_2}}  }$$ the maps being the obvious restrictions
\end{proof}


The following result is a long but straightforward verification. Recall (e.g. \cite[Definition 1.3]{romagny_actions}) that for a morphism of stacks endowed with an action of the a (pre)sheaf of groups $H$, being an $H$-equivariant morphism between $H$-stacks consists of data and is not just a property of the morphism.
\begin{lem}\label{actionequivariance}
Let $S \in \Aff$. 
\begin{itemize}
\item There is a canonical action of $\mathsf{LG}_X(S,-):\flags_X(S) \longrightarrow \mathbf{ShGrp}^{\mathrm{corr}}_{X\times S}$ on $\GrX_X(S)(-):\flags_X(S) \longrightarrow \St_{X\times S}$. Precisely, for any $i : F_1 \to F_2$ in $\flags_X(S)$, and any $T  \in \Aff_{X\times S}$, the morphism $$\GrX_X(S)(\alpha)(T): \GrX_X(S)(F_1)(T) \to \GrX_X(S)(F_2)(T)$$ has an $\mathsf{H}_X(S)(F_1,F_2)$-equivariant structure, where $\mathsf{H}_X(S)(F_1,F_2)$
is the sheaf of groups on $X\times S$ defining the correspondence $\mathsf{LG}_X(S,i): \mathsf{LG}_X(S,F_1) \to \mathsf{LG}_X(S,F_2)$
$$ \xymatrix{& \mathsf{H}_X(S)(F_1,F_2) \ar[rd] \ar[ld] & \\ \mathsf{LG}_X(S,F_1) && \mathsf{LG}_X(S,F_2)}$$
(see diagram \eqref{corrLG}).
\item Analogously, there is a canonical action of $\mathsf{L^+ G}_X(S,-):\flags_X(S) \longrightarrow \mathbf{ShGrp}_{X\times S}$ on $\GrX_X(S)(-):\flags_X(S) \longrightarrow \St_{X\times S}$ (see \cref{GcaractiononGr} and \cref{groupfunctonflags} (3)).
\end{itemize}
\end{lem}

By applying the (limit preserving) functor $\mathbf{St}_{X\times S} \to \mathbf{St}_{\mathbb{C}} \, : \mathcal{X} \mapsto \underline{\mathcal{X}}$ to $\mathsf{LG}_X(S,-)$, $\mathsf{L^+ G}_X(S,-)$, and $\GrX_X(S)$, we obtain analogous actions of $\underline{\mathsf{LG}}_X(S,-)$, $\underline{\mathsf{L^+ G}}_X(S,-)$ on $\Gr_X(S)$.
We leave to the reader to establish the functoriality in $S$ of $\underline{\mathsf{LG}}_X(S,-)$, $\underline{\mathsf{L^+G}}_X(S,-)$, and of their actions on $\Gr_X(S)$. \\

\subsection{Action on \texorpdfstring{$\fibGr^{\flat}_{X}$}{Grₓ}}
\begin{defin}\label{daflatguys} Let $\mathbf{PoSets} \to \mathbf{Sets}$ denote the forgetful functor, sending a poset $(S, \leq)$ to the set $S$.
We denote by $\flags^{\flat}_X,\, \fibGr^{\flat}_X$ the composite functors
$$\flags^{\flat}_X,\, \fibGr^{\flat}_X : \xymatrix{\Aff\op \ar[rr]^-{\flags_X,\, \fibGr_X} && \mathbf{PoSets} \ar[r] & \mathbf{Sets}. }$$
\end{defin} By integrating $\underline{\mathsf{L^+G}}_X(S,-)$ over all flags and all $S$'s, we now define an object $\mathcal{L}^{+}G_X \to \flags^{\flat}_X$ and its action on $\fibGr^{\flat}_X \to \flags^{\flat}_X$.\\ Namely, let us consider the functor\footnote{Since we aim at constructing an object acting on $\fibGr^{\flat}_{X}$ over $\flags^{\flat}_X$, we have chosen to disregard \emph{morphisms} between flags, i.e. the partial order structure on $\flags_X(S)$.} $\Aff\op \to \mathbf{Sets}$ $$\mathcal{L}^{+}G_X\, : \,S \longmapsto \left\{ (F, \alpha) \, | \, F\in \flags^{\flat}_X(S), \, \alpha \in L^{+}G_X(S)(F) \right\}.$$ There is an obvious morphism of functors $h: \mathcal{L}^{+}G_X \to \flags^{\flat}_X$, such that, for any $F: S \to \flags^{\flat}_X$, the fiber of $h$ at such $F \in \flags^{\flat}_X(S)$ is exactly the group $L^{+}G_X(S, F)$ of Definition \ref{defin:autgroup-fixedflag}. We will say that $h: \mathcal{L}^{+}G_X \to \flags^{\flat}_X$ is \emph{fibered in groups}. By Proposition \ref{groupfunctonflags}, $\mathcal{L}^{+}G_X$ acts on $\fibGr^{\flat}_{X}$ \emph{over} $\flags^{\flat}_X$, i.e., for any $F: S \to \flags^{\flat}_X$ the fiber of $h$ at $F$ acts on the fiber of $p_1: \fibGr^{\flat}_{X} \to \flags^{\flat}_X $ at $F$. In other words, there is a morphism \begin{equation}\label{morphismact}\mathrm{act}:\mathcal{L}^{+}G_X \times_{\flags^{\flat}_X} \fibGr^{\flat}_{X} \longrightarrow \fibGr^{\flat}_{X}\end{equation} satisfying the usual axioms of a group action over $\flags^{\flat}_X$.
Note that this action is analogous to the family of actions of $(L^+G)_{X^I}$ on the Beilinson-Drinfeld $Gr_{X^I}$ for variable non-empty finite sets $I$ (see e.g. \cite[Prop 3.19]{Zhu2017}). \\

\subsection{Equivariant sheaves} 
Using the action morphism (\ref{morphismact}) $$\mathrm{act}: \mathcal{L}G^{+}_X \times_{\flags^{\flat}_X} \fibGr^{\flat}_{X} \longrightarrow \fibGr^{\flat}_{X}$$ defined in Section \ref{sectionactions}, one can define the $\infty$-category $\mathbf{Shv}^{\mathcal{L}G^{+}_X}(\fibGr^{\flat}_{X}; \mathbb{Q}_{\ell})$ (respectively, $\mathcal{D}\mathbf{-Mod}^{\mathcal{L}G^{+}_X}(\fibGr^{\flat}_{X})$) of $\mathcal{L}G^{+}_X$-equivariant $\ell$-adic sheaves on $\fibGr^{\flat}_{X}$ (respectively, of $\mathcal{L}G^{+}_X$-equivariant $\mathcal{D}$-Modules on $\fibGr^{\flat}_{X}$), and then prove that the naive chiral product of \S $\,$ \ref{warmup} can be upgraded to monoidal structures on $\mathbf{Shv}^{\mathcal{L}G^{+}_X}(\fibGr^{\flat}_{X}; \mathbb{Q}_{\ell})$ and on $\mathcal{D}\mathbf{-Mod}^{\mathcal{L}G^{+}_X}(\fibGr^{\flat}_{X})$. These should be flags-analogs of the equivariant chiral monoidal structure considered in \cite[Ch. 2]{butson} for the Ran-space of a curve. These $\infty$-categories and their monoidal structures, together with their non-naive counterparts (obtained by developing an equivariant version of \S $\,$ \ref{chiralsection}), in the derived setting, will be investigated in a future work.

\bibliographystyle{plain}
\bibliography{dahema}

\end{document}

%% file: newcommands_all.tex
\usepackage{comment}
\ifpersonal
\newcommand*{\personal}[1]{\textcolor[rgb]{0.6,0.6,1}{(Personal: #1)}}

\newenvironment{newversion}{\color{purple}}{}
\usepackage[textsize=tiny]{todonotes}
\else
\newcommand*{\personal}[1]{\ignorespaces}
\excludecomment{oldversion}

\usepackage[disable]{todonotes}
\fi

\newcommand{\R}{\mathbb R}

\newcommand{\rB}{\mathrm B}

\newcommand{\rH}{\mathrm H}

\newcommand{\cC}{\mathcal C}

\newcommand{\cF}{\mathcal F}

\newcommand{\cO}{\mathcal O}

\DeclareFontFamily{U}{BOONDOX-calo}{\skewchar\font=45 }
\DeclareFontShape{U}{BOONDOX-calo}{m}{n}{<-> s*[1.05] BOONDOX-r-calo}{}
\DeclareFontShape{U}{BOONDOX-calo}{b}{n}{<-> s*[1.05] BOONDOX-b-calo}{}
\DeclareMathAlphabet{\mathcalboondox}{U}{BOONDOX-calo}{m}{n}

\makeatletter
\let\save@mathaccent\mathaccent
\newcommand*\if@single[3]{%
	\setbox0\hbox{${\mathaccent"0362{#1}}^H$}%
	\setbox2\hbox{${\mathaccent"0362{\kern0pt#1}}^H$}%
	\ifdim\ht0=\ht2 #3\else #2\fi
}
\newcommand*\rel@kern[1]{\kern#1\dimexpr\macc@kerna}
\newcommand*\widebar[1]{\@ifnextchar^{{\wide@bar{#1}{0}}}{\wide@bar{#1}{1}}}
\newcommand*\wide@bar[2]{\if@single{#1}{\wide@bar@{#1}{#2}{1}}{\wide@bar@{#1}{#2}{2}}}
\newcommand*\wide@bar@[3]{%
	\begingroup
	\def\mathaccent##1##2{%
		\let\mathaccent\save@mathaccent
		\if#32 \let\macc@nucleus\first@char \fi
		\setbox\z@\hbox{$\macc@style{\macc@nucleus}_{}$}%
		\setbox\tw@\hbox{$\macc@style{\macc@nucleus}{}_{}$}%
		\dimen@\wd\tw@
		\advance\dimen@-\wd\z@
		\divide\dimen@ 3
		\@tempdima\wd\tw@
		\advance\@tempdima-\scriptspace
		\divide\@tempdima 10
		\advance\dimen@-\@tempdima
		\ifdim\dimen@>\z@ \dimen@0pt\fi
		\rel@kern{0.6}\kern-\dimen@
		\if#31
		\overline{\rel@kern{-0.6}\kern\dimen@\macc@nucleus\rel@kern{0.4}\kern\dimen@}%
		\advance\dimen@0.4\dimexpr\macc@kerna
		\let\final@kern#2%
		\ifdim\dimen@<\z@ \let\final@kern1\fi
		\if\final@kern1 \kern-\dimen@\fi
		\else
		\overline{\rel@kern{-0.6}\kern\dimen@#1}%
		\fi
	}%
	\macc@depth\@ne
	\let\math@bgroup\@empty \let\math@egroup\macc@set@skewchar
	\mathsurround\z@ \frozen@everymath{\mathgroup\macc@group\relax}%
	\macc@set@skewchar\relax
	\let\mathaccentV\macc@nested@a
	\if#31
	\macc@nested@a\relax111{#1}%
	\else
	\def\gobble@till@marker##1\endmarker{}%
	\futurelet\first@char\gobble@till@marker#1\endmarker
	\ifcat\noexpand\first@char A\else
	\def\first@char{}%
	\fi
	\macc@nested@a\relax111{\first@char}%
	\fi
	\endgroup
}
\makeatother

\newcommand{\hA}{\widehat A}

\newcommand{\hZ}{{\widehat Z}}

\newcommand{\Coh}{\mathrm{Coh}}

\newcommand{\Tor}{\mathrm{Tor}}

\newcommand{\trunc}{\mathrm{t}_0}

\newcommand{\CAlg}{\mathrm{CAlg}}

\newcommand{\St}{\mathrm{St}}
\newcommand{\dSt}{\mathsf{dSt}}

\newcommand{\QCoh}{\mathrm{QCoh}}
\newcommand{\Perf}{\mathrm{Perf}}

\newcommand{\Cat}{\mathrm{Cat}}

\newcommand{\bfCoh}{\mathbf{Coh}}
\newcommand{\bfPerf}{\mathbf{Perf}}
\newcommand{\bfQCoh}{\mathbf{QCoh}}

\newcommand{\Catstmon}{\Cat_\infty^{\mathrm{Ex},\otimes}}

\newcommand{\Bun}{\mathrm{Bun}}

\newcommand{\id}{\mathrm{id}}

\newcommand{\red}{\mathrm{red}}

\newcommand{\op}{^\mathrm{op}}

\newcommand{\Hilb}{\operatorname{Hilb}}

\usetikzlibrary{decorations.markings} 
\tikzset{
  closed/.style = {decoration = {markings, mark = at position 0.5 with { \node[transform shape, xscale = .8, yscale=.4] {/}; } }, postaction = {decorate} },
  open/.style = {decoration = {markings, mark = at position 0.5 with { \node[transform shape, scale = .7] {$\circ$}; } }, postaction = {decorate} }
}

\DeclareMathOperator{\Fun}{Fun}

\DeclareMathOperator{\Hom}{Hom}

\DeclareMathOperator{\Map}{Map}

\DeclareMathOperator{\Spec}{Spec}
\DeclareMathOperator{\Spf}{Spf}

\DeclareMathOperator*{\colim}{colim}
\DeclareMathOperator*{\holim}{holim}

\let\originalleft\left
\let\originalright\right
\renewcommand{\left}{\mathopen{}\mathclose\bgroup\originalleft}
\renewcommand{\right}{\aftergroup\egroup\originalright}

\makeatletter
\def\DeclareMathBinOp{\@ifstar{\declaremathbinop@star}{\declaremathbinop@nostar}}
\def\declaremathbinop@star#1#2{\def#1{\test@subnexp@star#2}}
\def\test@subnexp@star#1{\@ifnextchar_{\isol@subnexp@star#1}{\test@exp@star#1}}
\def\test@exp@star#1{\@ifnextchar^{\isol@expnsub@star#1}{\mathbin{#1}}}
\def\isol@subnexp@star#1_#2{\@ifnextchar^{\eval@subnexp@star#1_#2}{\mathbin{\operatorname*{#1}_{#2}}}}
\def\eval@subnexp@star#1_#2^#3{\mathbin{\operatorname*{#1}_{#2}^{#3}}}
\def\isol@expnsub@star#1^#2{\@ifnextchar_{\eval@expnsub@star#1^#2}{\mathbin{\operatorname*{#1}^{#2}}}}
\def\eval@expnsub@star#1^#2_#3{\mathbin{\operatorname*{#1}_{#3}^{#2}}}

\def\declaremathbinop@nostar#1#2{\def#1{\test@subnexp@nostar{#2}}}
\def\test@subnexp@nostar#1{\@ifnextchar_{\isol@subnexp@nostar#1}{\test@exp@nostar#1}}
\def\test@exp@nostar#1{\@ifnextchar^{\isol@expnsub@nostar#1}{\mathbin{#1}}}
\def\isol@subnexp@nostar#1_#2{\@ifnextchar^{\eval@subnexp@nostar#1_{{#2}}}{\mathbin{\underset{#2}{#1}}}}
\def\eval@subnexp@nostar#1_#2^#3{\mathbin{\overset{#3}{\underset{#2}{#1}}}}
\def\isol@expnsub@nostar#1^#2{\@ifnextchar_{\eval@expnsub@nostar#1^{#2}}{\mathbin{\overset{#2}{#1}}}}
\def\eval@expnsub@nostar#1^#2_#3{\mathbin{\overset{#2}{\underset{#3}{#1}}}}

\def\declaremathbinop@toto#1#2{\def#1{\testoo@subnexp@toto{#2}}}
\def\testoo@subnexp@toto#1{\@ifnextchar_{\isoloo@subnexp@toto#1}{\testoo@exp@toto#1}}
\def\testoo@exp@toto#1{\@ifnextchar^{\isoloo@expnsub@toto#1}{\mathbin{#1}}}
\def\isoloo@subnexp@toto#1_#2{\@ifnextchar^{\eval@subnexp@toto#1_{{#2}}}{\mathbin{\underset{#2}{#1}}}}
\def\eval@subnexp@toto#1_#2^#3{\mathbin{\overset{#3}{\underset{#2}{#1}}}}
\def\isoloo@expnsub@toto#1^#2{\@ifnextchar_{\eval@expnsub@toto#1^{#2}}{\mathbin{\overset{#2}{#1}}}}
\def\eval@expnsub@toto#1^#2_#3{\mathbin{\overset{#2}{\underset{#3}{#1}}}}
\makeatother

\DeclareMathBinOp*{\newtimes}{\times}
\DeclareMathBinOp*{\newamalg}{\amalg}
\DeclareMathBinOp*{\newotimes}{\otimes}